\documentclass[10pt]{amsproc}
\usepackage[utf8]{inputenc}
\usepackage[dvipsnames]{xcolor}

\usepackage{etoolbox}
\makeatletter
\patchcmd{\@thm}{\let\thm@indent\indent}{\let\thm@indent\noindent}{}{}
\patchcmd{\@thm}{\thm@headfont{\scshape}}{\thm@headfont{\bfseries}}{}{}

\usepackage{mathtools}

\usepackage{multirow}
\usepackage{longtable}

\usepackage{amsmath}
\usepackage{amsfonts}
\usepackage{amssymb}
\usepackage{float}
\usepackage{amsthm}
\usepackage{comment}
\usepackage{amscd}
\usepackage{amsxtra}
\usepackage{epsfig}

\usepackage{listings}
\usepackage{color}

\definecolor{dkgreen}{rgb}{0,0.6,0}
\definecolor{gray}{rgb}{0.5,0.5,0.5}
\definecolor{mauve}{rgb}{0.58,0,0.82}

\usepackage{color}
\usepackage[all]{xy}
\usepackage{verbatim}

\usepackage{eucal}
\usepackage{mathrsfs}

\usepackage{graphicx}
\usepackage{tikz-cd}

\usepackage{url}
\usepackage{verbatim}

\usepackage{xy}
\xyoption{all}
\newcommand{\xar}[1]{\xrightarrow{{#1}}}

\usepackage{amsthm}

\usepackage{caption} 
\makeatletter
\def\@tocline#1#2#3#4#5#6#7{\relax
  \ifnum #1>\c@tocdepth 
  \else
    \par \addpenalty\@secpenalty\addvspace{#2}%
    \begingroup \hyphenpenalty\@M
    \@ifempty{#4}{%
      \@tempdima\csname r@tocindent\number#1\endcsname\relax
    }{%
      \@tempdima#4\relax
    }%
    \parindent\z@ \leftskip#3\relax \advance\leftskip\@tempdima\relax
    \rightskip\@pnumwidth plus4em \parfillskip-\@pnumwidth
    #5\leavevmode\hskip-\@tempdima
      \ifcase #1
       \or\or \hskip 1em \or \hskip 2em \else \hskip 3em \fi%
      #6\nobreak\relax
    \hfill\hbox to\@pnumwidth{\@tocpagenum{#7}}\par
    \nobreak
    \endgroup
  \fi}
\makeatother

\DeclareMathOperator{\Hom}{Hom}
\DeclareMathOperator{\Aut}{Aut}
\DeclareMathOperator{\Ext}{Ext}

\DeclareMathOperator{\fib}{fib}

\DeclareMathOperator{\End}{End}

\DeclareMathOperator{\hocolim}{hocolim}

\DeclareMathOperator{\Sym}{Sym}

\numberwithin{equation}{section}

\newtheorem{lemma}{Lemma}[subsection]

\newtheorem{corollary}[lemma]{Corollary}

\newtheorem{theorem}[lemma]{Theorem}

\newtheorem*{thmno}{Theorem}
\newtheorem*{propo}{Proposition}
\newtheorem{prop}[lemma]{Proposition}
\newtheorem{conjecture}[lemma]{Conjecture}

\newtheorem{abc}{abc}
\newtheorem{ABCthm}[abc]{Theorem}
\newtheorem{ABCconj}[abc]{Conjecture}
\newtheorem{ABCcor}[abc]{Corollary}

\newtheorem*{conj-moore}{Conjecture~\ref{moore-splitting}}
\newtheorem*{conj-cent}{Conjecture~\ref{centrality-conj}}
\newtheorem*{conj-tmf}{Conjecture~\ref{tmf-conj}}
\newtheorem*{thm-main}{Theorem~\ref{main-thm}}
\newtheorem*{cor-bpn}{Corollary~\ref{bpn}}
\newtheorem*{thm-string}{Theorem~\ref{mstring}}

\theoremstyle{definition}

\newtheorem{definition}[lemma]{Definition}
\newtheorem{warning}[lemma]{Warning}
\newtheorem{example}[lemma]{Example}

\newtheorem{construction}[lemma]{Construction}
\newtheorem{remark}[lemma]{Remark}
\newtheorem{notation}[lemma]{Notation}

\newcommand{\FF}{\mathbf{F}}
\newcommand{\Z}{\mathbf{Z}}
\newcommand{\QQ}{\mathbf{Q}}
\newcommand{\cc}{\mathbf{C}}

\newcommand{\cC}{\mathcal{C}}

\newcommand{\RR}{\mathbf{R}}

\newcommand{\Ln}{L_n}
\newcommand{\Sp}{\mathrm{Sp}}

\newcommand{\M}{\mathcal{M}}

\newcommand{\Mell}{\M_\mathrm{ell}}

\newcommand{\Mod}{\mathrm{Mod}}
\newcommand{\cL}{\mathcal{L}}

\newcommand{\HP}{\mathbf{H}P}
\newcommand{\CP}{\mathbf{C}P}
\newcommand{\RP}{\mathbf{R}P}

\newcommand{\Map}{\mathrm{Map}}

\newcommand{\dZ}{\mathcal{Z}}

\newcommand{\Alg}{\mathrm{Alg}}

\newcommand{\Sq}{\mathrm{Sq}}

\newcommand{\Eoo}{{\mathbf{E}_\infty}}

\newcommand{\Top}{\mathrm{Top}}

\newcommand{\ul}[1]{\underline{#1}}
\newcommand{\ol}[1]{\overline{#1}}

\newcommand{\wt}[1]{\widetilde{#1}}

\newcommand{\E}[1]{\mathbf{E}_{{#1}}}
\newcommand{\mmod}{/\!\!/}

\renewcommand{\S}{\mathbb{S}}

\newcommand{\Gr}{\mathrm{Gr}}

\newcommand{\SL}{\mathrm{SL}}
\newcommand{\GL}{\mathrm{GL}}
\newcommand{\BGL}{\mathrm{BGL}}

\newcommand{\MO}{\mathrm{MO}}
\newcommand{\MU}{\mathrm{MU}}
\newcommand{\BU}{\mathrm{BU}}
\renewcommand{\O}{\mathrm{O}}

\newcommand{\BO}{\mathrm{BO}}

\newcommand{\KO}{\mathrm{KO}}
\newcommand{\ku}{\mathrm{ku}}
\newcommand{\ko}{\mathrm{bo}}
\newcommand{\bu}{\mathrm{bu}}
\newcommand{\bo}{\mathrm{bo}}

\newcommand{\tmf}{\mathrm{tmf}}
\newcommand{\Tmf}{\mathrm{Tmf}}

\newcommand{\BP}[1]{\mathrm{BP}\langle{#1}\rangle}
\newcommand{\U}{\mathrm{U}}
\newcommand{\SU}{\mathrm{SU}}
\newcommand{\BPP}{\mathrm{BP}}

\renewcommand{\H}{\mathrm{H}}

\newcommand{\BSO}{\mathrm{BSO}}
\newcommand{\BSU}{\mathrm{BSU}}
\newcommand{\BSpin}{\mathrm{BSpin}}
\newcommand{\MSO}{\mathrm{MSO}}
\newcommand{\MSpin}{\mathrm{MSpin}}
\newcommand{\MString}{\mathrm{MString}}
\newcommand{\BString}{\mathrm{BString}}

\newcommand{\Conf}{\mathrm{Conf}}

\newcommand{\THH}{\mathrm{THH}}

\newcommand{\cQ}{\mathcal{Q}}

\newcommand{\B}{\mathrm{B}}

\newcommand{\fr}[1]{\mathfrak{#1}}
\newcommand{\TT}{{\Theta}}
\newcommand{\HH}{\mathbf{H}}
\newcommand{\cR}{\mathcal{R}}

\title{Higher chromatic Thom spectra via unstable homotopy theory}
\author{S. K. Devalapurkar}
\email{sdevalapurkar@math.harvard.edu}

\begin{document}

\maketitle

\begin{abstract}
    We investigate implications of an old conjecture in unstable homotopy theory related to the Cohen-Moore-Neisendorfer theorem and a conjecture about the $\E{2}$-topological Hochschild cohomology of certain Thom spectra (denoted $A$, $B$, and $T(n)$) related to Ravenel's $X(p^n)$. We show that these conjectures imply that the orientations $\MSpin\to \ko$ and $\MString\to \tmf$ admit spectrum-level splittings. This is shown by generalizing a theorem of Hopkins and Mahowald, which constructs $\H\FF_p$ as a Thom spectrum, to construct $\BP{n-1}$, $\ko$, and $\tmf$ as Thom spectra (albeit over $T(n)$, $A$, and $B$ respectively, and not over the sphere). This interpretation of $\BP{n-1}$, $\ko$, and $\tmf$ offers a new perspective on Wood equivalences of the form $\bo \wedge C\eta \simeq \bu$: they are related to the existence of certain EHP sequences in unstable homotopy theory. This construction of $\BP{n-1}$ also provides a different lens on the nilpotence theorem. Finally, we prove a $C_2$-equivariant analogue of our construction, describing $\ul{\H\Z}$ as a Thom spectrum.
\end{abstract}

\tableofcontents
\newpage

\section{Introduction}

\subsection{Statement of the main results}

One of the goals of this article is to describe a program to prove the following old conjecture (studied, for instance, in \cite{laures-k1-local, laures-k2-local}, and discussed informally in many places, such as \cite[Section 7]{mahowald-rezk-tmf}):
\begin{conjecture}\label{string-conj}
    The Ando-Hopkins-Rezk orientation (see \cite{koandtmf}) $\MString\to \tmf$ admits a spectrum-level splitting.
\end{conjecture}
The key idea in our program is to provide a universal property for mapping out of the spectrum $\tmf$. We give a proof which is conditional on an old conjecture from unstable homotopy theory stemming from the Cohen-Moore-Neisendorfer theorem and a conjecture about the $\E{2}$-topological Hochschild cohomology of certain Thom spectra (the latter of which simplifies the proof of the nilpotence theorem from \cite{dhs-i}). This universal property exhibits $\tmf$ as a certain Thom spectrum, similarly to the Hopkins-Mahowald construction of $\H\Z_p$ and $\H\FF_p$ as Thom spectra.

To illustrate the gist of our argument in a simpler case, recall Thom's classical result from \cite{thom-original}: the unoriented cobordism spectrum $\MO$ is a wedge of suspensions of $\H\FF_2$. The simplest way to do so is to show that $\MO$ is an $\H\FF_2$-module, which in turn can be done by constructing an $\E{2}$-map $\H\FF_2\to \MO$. The construction of such a map is supplied by the following theorem of Hopkins and Mahowald:
\begin{thmno}[{Hopkins-Mahowald; see \cite{mahowald-thom} and \cite[Lemma 3.3]{mrs}}]
    Let $\mu:\Omega^2 S^3\to \BO$ denote the real vector bundle over $\Omega^2 S^3$ induced by extending the map $S^1\to \BO$ classifying the M\"obius bundle. Then the Thom spectrum of $\mu$ is equivalent to $\H\FF_2$ as an $\E{2}$-algebra.
   %
\end{thmno}
\begin{remark}\label{e2-split-MO}
The Thomification of the $\E{2}$-map $\mu:\Omega^2 S^3\to \BO$ produces the desired $\E{2}$-splitting $\H\FF_2\to \MO$.
\end{remark}

Our argument for Conjecture \ref{string-conj} takes this approach: we shall show that an old conjecture from unstable homotopy theory and a conjecture about the $\E{2}$-topological Hochschild cohomology of certain Thom spectra provide a construction of $\tmf$ (as well as $\bo$ and $\BP{n}$) as a Thom spectrum, and utilize the resulting universal property of $\tmf$ to construct an (unstructured) map $\tmf\to \MString$.

Mahowald was the first to consider the question of constructing spectra like $\bo$ and $\tmf$ as Thom spectra (see \cite{mahowald-bo-bu}). Later work by Rudyak in \cite{rudyak} sharpened Mahowald's results to show that $\bo$ and $\bu$ cannot appear as the Thom spectrum of a $p$-complete spherical fibration. In \cite{angeltveit-hill-lawson-e3}, Angeltveit-Hill-Lawson gave an alternative proof of this fact under the assumption that the $p$-complete spherical fibration is classified by a map of $\E{3}$-spaces. Recently, Chatham has shown in \cite{hood-tmf} that $\tmf^\wedge_2$ cannot appear as the Thom spectrum of a structured $2$-complete spherical fibration over a loop space.
Our goal is to argue that these issues are alleviated if we replace ``spherical fibrations'' with ``bundles of $R$-lines'' for certain well-behaved spectra $R$.

The first hint of $\tmf$ being a generalized Thom spectrum comes from a conjecture of Hopkins and Hahn regarding a construction of the truncated Brown-Peterson spectra $\BP{n}$ as Thom spectra. To state this conjecture, we need to recall some definitions. Recall (see \cite{dhs-i}) that $X(n)$ denotes the Thom spectrum of the map $\Omega \SU(n)\to \Omega \SU\simeq \BU$. Upon completion at a prime $p$, the spectra $X(k)$ for $p^n\leq k\leq p^{n+1}-1$ split as a direct sum of suspensions of certain homotopy commutative ring spectra $T(n)$, which in turn filter the gap between the $p$-complete sphere spectrum and $\BPP$ (in the sense that $T(0) = \S$ and $T(\infty) = \BPP$). Then:
\begin{conjecture}[Hahn, Hopkins; unpublished]\label{hopkins-conj}
    There is a map $f:\Omega^2 S^{|v_n|+3}\to \B\GL_1(T(n))$, which detects an indecomposable element $v_n\in \pi_{|v_n|} T(n)$ on the bottom cell of the source, whose Thom spectrum is a form of $\BP{n-1}$.
\end{conjecture}
The primary obstruction to proving that a map $f$ as in Conjecture \ref{hopkins-conj} exists stems from the failure of $T(n)$ to be an $\E{3}$-ring (due to Lawson; \cite[Example 1.5.31]{lawson-dl}). If $R$ is an $\E{1}$- or $\E{2}$-ring spectrum, let $\fr{Z}_3(R)$ denote the $\E{2}$-topological Hochschild cohomology of $R$ (see Definition \ref{ek-1-center}). Hahn suggested that one way to get past the failure of $T(n)$ to be an $\E{3}$-ring would be via the following conjecture:
\begin{conjecture}[Hahn]\label{hahn-conjecture}
    There is an indecomposable element $v_n\in \pi_{|v_{n}|} T(n)$ which lifts to the $\E{2}$-topological Hochschild cohomology $\fr{Z}_3(X(p^n))$ of $X(p^n)$.
\end{conjecture}
We do not know how to prove this conjecture (and have no opinion on whether or not it is true). We shall instead show that Conjecture \ref{hopkins-conj} is implied by the two conjectures alluded to above. We shall momentarily state these conjectures precisely as Conjectures \ref{moore-splitting} and \ref{centrality-conj}; let us first state our main results.

We need to introduce some notation. Let $y(n)$ (resp. $y_\Z(n)$) denote the Mahowald-Ravenel-Shick spectrum, constructed as a Thom spectrum over $\Omega J_{p^n-1}(S^2)$ (resp. $\Omega J_{p^n-1}(S^2)\langle 2\rangle$) introduced in \cite{mrs} to study the telescope conjecture (resp. in \cite{yz} as $z(n)$). Let $A$ denote the $\E{1}$-quotient $\S\mmod\nu$ of the sphere spectrum by $\nu\in \pi_3(\S)$; its mod $2$ homology is $\H_\ast(A) \cong \FF_2[\zeta_1^4]$. The spectrum $A$ has been intensely studied by Mahowald and his coauthors in (for instance) \cite{mahowald-thom, davis-mahowald, mahowald-v2-periodic, mahowald-bo-res, mahowald-imj, mahowald-unell-bott}, where it is often denoted $X_5$.  (See Remark \ref{em-quotient-terminology} for motivation for the term ``$\E{1}$-quotient''.) Let $B$ denote the $\E{1}$-ring introduced in \cite[Construction 3.1]{tmf-witten}; it has been briefly studied under the name $\ol{X}$ in \cite{hopkins-mahowald-orientations}. It may be constructed as the Thom spectrum of a vector bundle over an $\E{1}$-space $N$ which sits in a fiber sequence $\Omega S^9\to N\to \Omega S^{13}$. The mod $2$ homology of $B$ is $\H_\ast(B) \cong \FF_2[\zeta_1^8, \zeta_2^4]$.

We also need to recall some unstable homotopy theory. In \cite{cmn-1, cmn-2, cmn-3}, Cohen, Moore, and Neisendorfer constructed a map $\phi_n: \Omega^2 S^{2n+1} \to S^{2n-1}$ whose composite with the double suspension $E^2: S^{2n-1} \to \Omega^2 S^{2n+1}$ is the degree $p$ map. (The symbol $E$ stands for ``Einh\"angung'', which is German for ``suspension''.) Such a map was also constructed by Gray in \cite{gray-1, gray-2}. In Section \ref{gray}, we introduce the related notion of a \emph{charming map} (Definition \ref{gray-map-def}), one example of which is the Cohen-Moore-Neisendorfer map.

Our main result is then:
\begin{ABCthm}\label{main-thm}
    Suppose $R$ is a base spectrum of height $n$ as in the second line of Table \ref{the-table}. Let $K_{n+1}$ denote the fiber of a charming map $\Omega^2 S^{2p^{n+1}+1} \to S^{2p^{n+1}-1}$. Then Conjectures \ref{moore-splitting} and \ref{centrality-conj} imply that there is a map $\mu: K_{n+1}\to \B\GL_1(R)$ such that the mod $p$ homology of the Thom spectrum $K_{n+1}^\mu$ is isomorphic to the mod $p$ homology of the associated designer chromatic spectrum $\TT(R)$ as a Steenrod comodule\footnote{We elected to use the symbol $\TT$ because the first two letters of the English spelling of $\TT$ and of Thom's name agree.}.

    If $R$ is any base spectrum other than $B$, the Thom spectrum $K_{n+1}^\mu$ is equivalent to $\TT(R)$ upon $p$-completion for every prime $p$. If Conjecture \ref{tmf-conj} is true, then the same is true for $B$: the Thom spectrum $K_{n+1}^\mu$ is equivalent to $\TT(B) = \tmf$ upon $2$-completion.
\end{ABCthm}
\begin{table}[h!]
    \centering
    \begin{tabular}{c | c c c c c c c c c c}
	Height & 0 & 1 & 2 & $n$ & $n$ & $n$\\
	\hline
	Base spectrum $R$ & $\S^\wedge_p$ & $A$ & $B$ & $T(n)$ & $y(n)$ & $y_\Z(n)$ \\
	\hline
	Designer chromatic spectrum $\TT(R)$ & $\H\Z_p$ & $\bo$ & $\tmf$ & $\BP{n}$ & $k(n)$ & $k_\Z(n)$\\
    \end{tabular}
    \vspace{0.5cm}
    \caption{To go from a base spectrum ``of height $n$'', say $R$, in the second line to the third, one takes the Thom spectrum of a bundle of $R$-lines over $K_{n+1}$.}
    \label{the-table}
\end{table}
Making sense of Theorem \ref{main-thm} relies on knowing that $T(n)$ admits the structure of an $\E{1}$-ring; this is proved in \cite{tn-e1}; see also Warning \ref{tn-thom-def}. Note that the spectra $A$, $B$, $y(n)$, and $y_\Z(n)$ all admit $\E{1}$-structures by construction. In Remark \ref{nilpotence-proof}, we sketch how Theorem \ref{main-thm} relates to the proof of the nilpotence theorem.

Although the form of Theorem \ref{main-thm} does not resemble Conjecture \ref{hopkins-conj}, we show that Theorem \ref{main-thm} implies the following result.
\begin{ABCcor}\label{bpn}
    Conjectures \ref{moore-splitting} and \ref{centrality-conj} imply Conjecture \ref{hopkins-conj}.
\end{ABCcor}
In the case $n=0$, Corollary \ref{bpn} recovers the Hopkins-Mahowald theorem constructing $\H\FF_p$. Moreover, Corollary \ref{bpn} is true unconditionally when $n=0,1$.

Using the resulting universal property of $\tmf$, we obtain a result pertaining to Conjecture \ref{string-conj}.
\begin{ABCthm}\label{mstring}
    Assume that the composite $\fr{Z}_3(B)\to B\to \MString$ is an $\E{3}$-map. Then Conjectures \ref{moore-splitting}, \ref{centrality-conj}, and \ref{tmf-conj} imply that there is a spectrum-level unital splitting of the Ando-Hopkins-Rezk orientation $\MString_{(2)} \to \tmf_{(2)}$.
\end{ABCthm}
In particular, Conjecture \ref{string-conj} follows (at least after localizing at $p=2$; a slight modification of our arguments should work at any prime). We believe that the assumption that the composite $\fr{Z}_3(B)\to B\to \MString$ is an $\E{3}$-map is too strong: we believe that it can be removed using special properties of fibers of charming maps, and we will return to this in future work.

We stress that these splittings are unstructured; it seems unlikely that they can be refined to structured splittings. In \cite{tmf-witten}, we showed (unconditionally) that the Ando-Hopkins-Rezk orientation $\MString\to \tmf$ induces a surjection on homotopy, a result which is clearly implied by Theorem \ref{mstring}.

We remark that the argument used to prove Theorem \ref{mstring} shows that if the composite $\fr{Z}_3(A)\to A\to \MSpin$ is an $\E{3}$-map, then Conjectures \ref{moore-splitting} and \ref{centrality-conj} imply that there is a spectrum-level unital splitting of the Atiyah-Bott-Shapiro orientation $\MSpin\to \bo$. This splitting was originally proved unconditionally (i.e., without assuming Conjecture \ref{moore-splitting} or Conjecture \ref{centrality-conj}) by Anderson-Brown-Peterson in \cite{abp} via a calculation with the Adams spectral sequence.

\subsection{The statements of Conjectures \ref{moore-splitting}, \ref{centrality-conj}, and \ref{tmf-conj}}

We first state Conjecture \ref{moore-splitting}. The second part of this
conjecture is a compilation of several old conjectures
in unstable homotopy theory originally made by Cohen-Moore-Neisendorfer, Gray,
and Selick in \cite{cmn-1, cmn-2, cmn-3, gray-1, gray-2, selick}. The statement
we shall give momentarily differs slightly from the statements made in the
literature: for instance, in Conjecture \ref{moore-splitting}(b), we demand a
$\cQ_1$-space splitting (Notation \ref{cup-1-defn}), rather than merely an
H-space splitting.
\begin{ABCconj}\label{moore-splitting}
    The following statements are true:
    \begin{enumerate}
	\item The homotopy fiber of any charming map (Definition
	    \ref{gray-map-def}) is equivalent as a loop space to the loop space
	    on an Anick space (Definition \ref{cmn-map}).
	\item There exists a $p$-local charming map $f:\Omega^2 S^{2p^n+1}\to
	    S^{2p^n-1}$ whose homotopy fiber admits a $\cQ_1$-space retraction
	    off of $\Omega^2 (S^{2p^n}/p)$. There are also integrally defined
	    maps $\Omega^2 S^9\to S^7$ and $\Omega^2 S^{17} \to S^{15}$ whose
	    composites with the double suspension on $S^7$ and $S^{15}$
	    respectively are the degree $2$ maps. Moreover, their homotopy
	    fibers $K_2$ and $K_3$ (respectively) admit deloopings, and admit
	    $\cQ_1$-space retractions off $\Omega^2 (S^8/2)$ and $\Omega^2
	    (S^{16}/2)$ (respectively).
    \end{enumerate}
\end{ABCconj}
%

Next, we turn to Conjecture \ref{centrality-conj}. This conjecture is concerned
with the $\E{2}$-topological Hochschild cohomology of the Thom spectra
$X(p^n-1)_{(p)}$, $A$, and $B$ introduced above.
\begin{ABCconj}\label{centrality-conj}
    Let $n\geq 0$ be an integer. Let $R$ denote $X(p^{n+1}-1)_{(p)}$, $A$ (in
    which case $n=1$), or $B$ (in which case $n=2$). Then the element
    $\sigma_n\in \pi_{|v_n|-1} R$ lifts to the $\E{2}$-topological Hochschild
    cohomology $\fr{Z}_3(R)$ of $R$, and is $p$-torsion in $\pi_\ast
    \fr{Z}_3(R)$ if $R = X(p^{n+1} - 1)_{(p)}$, and is $2$-torsion in $\pi_\ast
    \fr{Z}_3(R)$ if $R = A$ or $B$.
\end{ABCconj}
Finally, we state Conjecture \ref{tmf-conj}. It is inspired by
\cite{adams-priddy, angeltveit-lind}. We believe this conjecture is the most
approachable of the conjectures stated here.
\begin{ABCconj}\label{tmf-conj}
    Suppose $X$ is a spectrum which is bounded below and whose homotopy groups
    are finitely generated over $\Z_p$. If there is an isomorphism
    $\H_\ast(X;\FF_p)\cong \H_\ast(\tmf;\FF_p)$ of Steenrod comodules, then
    there is a homotopy equivalence $X^\wedge_p\to \tmf^\wedge_p$ of spectra.
\end{ABCconj}

After proving Theorem \ref{main-thm} and Theorem \ref{mstring}, we explore
relationships between the different spectra appearing on the second line of
Table \ref{the-table} in the remainder of the article. In particular, we prove
analogues of Wood's equivalence $\bo\wedge C\eta \simeq \bu$ (see also
\cite{homologytmf}) for these spectra. We argue that these are related to the
existence of certain EHP sequences.

Finally, we describe a $C_2$-equivariant analogue of Corollary \ref{bpn} at
$n=1$ as Theorem \ref{C2-main-thm}, independently of a $C_2$-equivariant
analogue of Conjecture \ref{moore-splitting} and Conjecture
\ref{centrality-conj}. This result constructs $\H\ul{\Z}$ as a Thom spectrum of
an equivariant bundle of invertible $T(1)_\RR$-modules over $\Omega^\rho
S^{2\rho+1}$, where $T(1)_\RR$ is the free $\E{\sigma}$-algebra with a
nullhomotopy of the equivariant Hopf map $\wt{\eta}\in \pi_\sigma(\S)$, and
$\rho$ and $\sigma$ are the regular and sign representations of $C_2$,
respectively. This uses results of Behrens-Wilson and Hahn-Wilson from
\cite{behrens-wilson, hahn-wilson}. We believe there is a similar result at odd
primes, but we defer discussion of this. We discuss why our methods do not work
to yield $\BP{n}_\RR$ for $n\geq 1$ as in Corollary \ref{bpn}.


\subsection{Outline}

Section \ref{positive-negative} contains a review some of the theory of Thom
spectra from the modern perspective, as well as the proof of the classical
Hopkins-Mahowald theorem. The content reviewed in this section will appear in
various guises throughout this project, hence its inclusion.

In Section \ref{A-and-B}, we study certain $\E{1}$-rings; most of them appear as
Thom spectra over the sphere. For instance, we recall some facts about Ravenel's
$X(n)$ spectra, and then define and prove properties about the $\E{1}$-rings $A$
and $B$ used in the statement of Theorem \ref{main-thm}. We state Conjecture
\ref{centrality-conj}, and discuss (Remark \ref{nilpotence-proof}) its relation
to the nilpotence theorem, in this section.

In Section \ref{unstable-review}, we recall some unstable homotopy theory, such
as the Cohen-Moore-Neisendorfer map and the fiber of the double suspension.
These concepts do not show up often in stable homotopy theory, so we hope this
section provides useful background to the reader.
We state Conjecture \ref{moore-splitting}, and then explore properties of Thom
spectra of bundles defined over Anick spaces.

In Section \ref{proof-section}, we state and prove Theorem \ref{main-thm} and
Corollary \ref{bpn}, and state several easy consequences of Theorem
\ref{main-thm}.

In Section \ref{applications}, we study some applications of Theorem
\ref{main-thm}.  For instance, we use it to prove Theorem \ref{mstring}, which
is concerned with the splitting of certain cobordism spectra. 
In a previous version of this article, we had two subsections discussing Wood-like
equivalences, and topological Hochschild homology of the chromatic Thom spectra
of Table \ref{the-table}. However, while making revisions to this article, we
decided to split these two sections off into separate articles \cite{ehp-wood, thh-xn}.

In Section \ref{equiv-analogue}, we prove an equivariant analogue of Corollary
\ref{bpn} at height $1$. We construct equivariant analogues of $X(n)$ and $A$,
and describe why our methods fail to produce an equivariant analogue of
Corollary \ref{bpn} at all heights, even granting an analogue of Conjecture
\ref{moore-splitting} and Conjecture \ref{centrality-conj}.

Finally, in Section \ref{future}, we suggest some directions for future
research. There are also numerous interesting questions arising from our work,
which we have indicated in the body of the article.

\subsection*{Conventions}

Unless indicated otherwise, or if it goes against conventional notational
choices, a Latin letter with a numerical subscript (such as $x_5$) denotes an
element of degree given by its subscript. If $X$ is a space and $R$ is an
$\E{1}$-ring spectrum, then $X^\mu$ will denote the Thom spectrum of some bundle
of invertible $R$-modules determined by a map $\mu:X\to \B\GL_1(R)$. We shall
often quietly localize or complete at an implicit prime $p$. Although we have
tried to be careful, \emph{all limits and colimits will be homotopy limits and
colimits}; we apologize for any inconvenience this might cause.

We shall denote by $P^k(p)$ the mod $p$ Moore space $S^{k-1}\cup_p e^k$ with top
cell in dimension $k$. The symbols $\zeta_i$ and $\tau_i$ will denote the
\emph{conjugates} of the Milnor generators (commonly written nowadays as $\xi_i$
and $\tau_i$, although, as Haynes Miller pointed out to me, our notation for the
conjugates was Milnor's original notation) in degrees $2(p^i-1)$ and
$2p^i-1$ for $p>2$ and $2^i-1$ (for $\zeta_i$) at $p=2$. Unfortunately, we will
use $A$ to denote the $\E{1}$-ring in appearing in Table \ref{the-table}, and
write $A_\ast$ to denote the dual Steenrod algebra. We hope this does not cause
any confusion, since we will always denote the homotopy groups of $A$ by
$\pi_\ast A$ and not $A_\ast$.

If $\mathcal{O}$ is an operad, we will simply write \emph{$\mathcal{O}$-ring} to denote an
$\mathcal{O}$-algebra object in spectra. A map of $\mathcal{O}$-rings respecting the
$\mathcal{O}$-algebra structure will often simply be called a $\mathcal{O}$-map. Unless it is
clear that we mean otherwise, all modules over non-$\Eoo$-algebras will be left
modules.

Hood Chatham pointed out to me that $S^3\langle 4\rangle$ would be the correct
notation for what we denote by $S^3\langle 3\rangle = \fib(S^3\to K(\Z,3))$.
Unfortunately, the literature seems to have chosen $S^3\langle 3\rangle$ as the
preferred notation, so we stick to that in this project.

When we write that Theorem \ref{main-thm}, Corollary \ref{bpn}, or Theorem
\ref{mstring} implies a statement P, we mean that Conjectures
\ref{moore-splitting} and Conjecture \ref{centrality-conj} (and Conjecture
\ref{tmf-conj}, if the intended application is to $\tmf$) imply P via Theorem
\ref{main-thm}, Corollary \ref{bpn}, or Theorem \ref{mstring}.

\subsection*{Acknowledgements}

The genesis of this project took place in conversations with Jeremy Hahn, who
has been a great friend and mentor, and a fabulous resource; I'm glad to be able
to acknowledge our numerous discussions, as well as his comments on a draft of
this article. I'm extremely grateful to Mark Behrens and Peter May for working
with me over the summer of 2019 and for being fantastic advisors, as well as for
arranging my stay at UChicago, where part of this work was done. I'd also like
to thank Haynes Miller for patiently answering my numerous (often silly)
questions over the past few years. Conversations related to the topic of this
project also took place at MIT and Boulder, and I'd like to thank Araminta
Gwynne, Robert Burklund, Hood Chatham, Peter Haine, Mike Hopkins, Tyler Lawson, Kiran Luecke,
Andrew Senger, Neil Strickland, and Dylan Wilson for clarifying discussions.
Although I never got the chance to meet Mark Mahowald, my intellectual debt to
him is hopefully evident (simply \texttt{Ctrl+F} his name in this document!).
I would like to the anonymous referee for several comments which greatly improved this article.
Finally, I'm glad I had the opportunity to meet other math nerds at the UChicago
REU; I'm in particular calling out Ada, Anshul, Eleanor, and Wyatt --- thanks
for making my summer enjoyable.

Part of this work was done when the author was supported by the PD Soros Fellowship.

\newpage
\section{Background, and some classical positive and negative
results}\label{positive-negative}

\subsection{Background on Thom spectra}\label{thom-background}

In this section, we will recall some facts about Thom spectra and their
universal properties; the discussion is motivated by \cite{abghr-ii}.

\begin{definition}
    Let $A$ be an $\E{1}$-ring, and let $\mu:X\to \B\GL_1(A)$ be a map of spaces.
    The Thom $A$-module $X^\mu$ is defined as the homotopy pushout
    $$\xymatrix{
	\Sigma^\infty_+\GL_1(A)\ar[r] \ar[d] & \Sigma^\infty_+ \fib(\mu)
	\ar[d]\\
	A\ar[r] & X^\mu.
    }$$
\end{definition}

\begin{remark}
    Let $A$ be an $\E{1}$-ring, and let $\mu:X\to \B\GL_1(A)$ be a map of spaces.
    The Thom $A$-module $X^\mu$ is the homotopy colimit of the functor
    $X\xrightarrow{\mu} \B\GL_1(A)\to \Mod(A)$, where we have abused notation by
    identifying $X$ with its associated Kan complex. If $A$ is an
    $\E{1}$-$R$-algebra, then the $R$-module underlying $X$ can be identified
    with the homotopy colimit of the composite functor
    $$X\xrightarrow{\mu} \B\GL_1(A)\to \B\Aut_R(A)\to \Mod(R),$$
    where we have identified $X$ with its associated Kan complex. The space
    $\B\Aut_R(A)$ can be regarded as the maximal subgroupoid of $\Mod(R)$ spanned
    by the object $A$.
\end{remark}

The following is immediate from the description of the Thom spectrum as a Kan
extension:
\begin{prop}\label{thom-base}
    Let $R$ and $R'$ be $\E{1}$-rings with an $\E{1}$-ring map $R\to R'$
    exhibiting $R'$ as a right $R$-module. If $f:X\to \B\GL_1(R)$ is a map of
    spaces, then the Thom spectrum of the composite $X\to \B\GL_1(R)\to
    \B\GL_1(R')$ is the base-change $X^f\wedge_R R'$ of the (left) $R$-module
    Thom spectrum $X^f$.
\end{prop}
\begin{corollary}\label{thom-iso}
    Let $R$ and $R'$ be $\E{1}$-rings with an $\E{1}$-ring map $R\to R'$
    exhibiting $R'$ as a right $R$-module. If $f:X\to \B\GL_1(R)$ is a map of
    spaces such that the the composite $X\to \B\GL_1(R)\to B\GL_1(R')$ is null,
    then there is an equivalence $X^f\wedge_R R' \simeq R' \wedge
    \Sigma^\infty_+ X$.
\end{corollary}
Moreover (see e.g. \cite[Corollary 3.2]{barthel-thom}):
\begin{prop}
    Let $X$ be a $k$-fold loop space, and let $R$ be an $\E{k+1}$-ring. Then the
    Thom spectrum of an $\E{k}$-map $X\to \B\GL_1(R)$ is an $\E{k}$-$R$-algebra.
\end{prop}

We will repeatedly use the following classical result, which is again a
consequence of the observation that Thom spectra are colimits, as well as the
fact that total spaces of fibrations may be expressed as colimits; see also
\cite[Theorem 1]{beardsley-thom}.
\begin{prop}\label{thom}
    Let $X\xrightarrow{i} Y\to Z$ be a fiber sequence of $k$-fold loop spaces
    (where $k\geq 1$), and let $R$ be an $\E{m}$-ring for $m\geq k+1$. Suppose
    that $\mu:Y\to \B\GL_1(R)$ is a map of $k$-fold loop spaces. Then, there is a
    $k$-fold loop map $\phi:Z\to \B\GL_1(X^{\mu \circ i})$ whose Thom spectrum is
    equivalent to $Y^\mu$ as $\E{k-1}$-rings. Concisely, if arrows are labeled
    by their associated Thom spectra, then there is a diagram
    $$\xymatrix{
	X \ar[r]^-i \ar[dr]_-{X^{\mu \circ i}} & Y \ar[r] \ar[d]_-\mu^-{Y^\mu} &
	Z \ar[d]_-{\phi}^-{Y^\mu = Z^\phi}\\
	& \B\GL_1(R) \ar[r] & B\GL_1(X^{\mu \circ i}).
    }$$
\end{prop}
The argument to prove Proposition \ref{thom} also goes through with slight
modifications when $k=0$, and shows:
\begin{propo}
    Let $X\xrightarrow{i} Y\to Z$ be a fiber sequence of spaces with $Z$
    connected, and let $R$ be an $\E{m}$-ring for $m\geq 1$. Suppose that
    $\mu:Y\to \B\GL_1(R)$ is a map of Kan complexes. Then, there is a map
    $\phi:Z\to \B\Aut_R(X^{\mu \circ i})$ such that the homotopy colimit (i.e., ``Thom spectrum'') $Z^\phi$ of the following composite is equivalent to $Y^\mu$ as an $R$-module:
    \begin{equation}\label{hocolim}
        Z\xar{\phi} \B\Aut_R(X^{\mu \circ i}) \subseteq \Mod_R.
    \end{equation}
\end{propo}
We will abusively refer to this result in the sequel also as Proposition
\ref{thom}.
\begin{proof}[Proof of the second form of Proposition \ref{thom}]
It will be convenient to use the model for Thom spectra following \cite{abghr-ii}. Observe that a fibration $X \to Y \to Z$ implies (e.g., by \cite[Remark 2.4]{abghr-ii}) that there is a functor $Z \to \Top$ whose homotopy colimit is $Y$, and whose fiber over any vertex of $z\in Z$ is $X$. Since $X$ is connected, we may write $Y \simeq \hocolim_Z X$. The map $X \to Y$ is induced by the inclusion $\{z\}\hookrightarrow Z$.
Since $Y$ is a Kan complex, the Thom spectrum $Y^\mu$ can be identified (by \cite[Definition 1.4]{abghr-ii}) with the homotopy colimit of the composite $Y \xar{\mu} \B\GL_1(R) \simeq R\text{-}\mathrm{line} \subseteq \Mod_R$ (which we will temporarily denote by $\mu: Y \to \Mod_R$). We will write this as $Y^\mu \simeq \hocolim_Y R$. The left Kan extension of the map $Y \to Z$ along the functor $\mu: Y \to \Mod_R$ defines a functor $\phi: Z \to \Mod_R$, which sends $z\in Z$ to $X^{\mu \circ i} \simeq \hocolim(X \to Y \xar{\mu} \Mod_R)$. Since $Z$ is connected, this implies that $Y^\mu \simeq \hocolim_Y R$ is the homotopy colimit of the functor \eqref{hocolim}.
\end{proof}

The following is a slight generalization of \cite[Theorem 4.10]{barthel-thom}:
\begin{theorem}\label{thom-univ}
    Let $R$ be an $\E{k+1}$-ring for $k\geq 0$, and let $\alpha: Y \to
    \B\GL_1(R)$ be a map from a pointed space $Y$. For any $0\leq m\leq k$, let
    $\wt{\alpha}: \Omega^m \Sigma^m Y \to \B\GL_1(R)$ denote the extension of
    $\alpha$. Then the Thom spectrum $(\Omega^m \Sigma^m Y)^{\wt{\alpha}}$ is
    the free $\E{m}$-$R$-algebra $A$ for which the composite $Y\to \B\GL_1(R) \to
    \B\GL_1(A)$ is null. More precisely, if $A$ is any $\E{m}$-$R$-algebra, then
    $$\Map_{\Alg^{\E{m}}_R}((\Omega^m \Sigma^m Y)^{\wt{\alpha}}, A) \simeq
    \begin{cases}
	\Map_\ast(Y, \Omega^{\infty} A) & \text{if }\alpha: Y\to \B\GL_1(R)
	\to \B\GL_1(A) \text{ is null},\\
	\emptyset & \text{else}.
    \end{cases}$$
\end{theorem}
\begin{remark}\label{em-quotient-terminology}
    Say $Y = S^{n+1}$, so $\alpha$ detects an element $\alpha\in \pi_n R$.
    Theorem \ref{thom-univ} suggests interpreting the Thom spectrum $(\Omega^m
    S^{m+n+1})^{\wt{\alpha}}$ as an $\E{m}$-quotient; to signify this, we will
    denote it by $R\mmod_{\E{m}} \alpha$. If $m=1$, then we will simply denote
    it by $R\mmod \alpha$, while if $m=0$, then the $\E{m}$-quotient is simply
    the ordinary quotient $R/\alpha$. See \cite[Definition 4.3]{barthel-thom},
    where the quotient $R\mmod_{\E{m}} \alpha$ is called the \emph{versal
    $R$-algebra of characteristic $\alpha$}.
\end{remark}

\subsection{The Hopkins-Mahowald theorem}

The primary motivation for this project is the following miracle (see
\cite{mahowald-thom} for $p=2$ and \cite[Lemma 3.3]{mrs} for $p>2$, as well as
\cite[Theorem 5.1]{barthel-thom} for a proof of the equivalence as one of $\E{2}$-algebras):
\begin{theorem}[Hopkins-Mahowald]\label{hm}
    Let $\S^\wedge_p$ be the $p$-completion of the sphere at a prime $p$, and
    let $f:S^1\to \B\GL_1(\S^\wedge_p)$ detect the element $1-p\in \pi_1
    \B\GL_1(\S^\wedge_p) \simeq \Z_p^\times$. Let $\mu:\Omega^2 S^3\to
    \B\GL_1(\S^\wedge_p)$ denote the $\E{2}$-map extending $f$; then there is a
    $p$-complete equivalence $(\Omega^2 S^3)^\mu \to \H\FF_p$ of $\E{2}$-ring
    spectra.
\end{theorem}
It is not too hard to deduce the following result from Theorem \ref{hm}:
\begin{corollary}\label{hm-Z}
    Let $S^3\langle 3\rangle$ denote the $3$-connected cover of $S^3$. Then the
    Thom spectrum of the composite $\Omega^2 S^3\langle 3\rangle \to \Omega^2
    S^3\xrightarrow{\mu} \B\GL_1(\S^\wedge_p)$ is equivalent to $\H\Z_p$ as an
    $\E{2}$-ring.
\end{corollary}
\begin{remark}\label{hm-nilp}
    Theorem \ref{hm} implies a restrictive version of the nilpotence theorem: if
    $R$ is an $\E{2}$-ring spectrum, and $x\in \pi_\ast R$ is a simple
    $p$-torsion element which has trivial $\H\FF_p$-Hurewicz image, then $x$ is
    nilpotent. This is explained in \cite[Proposition
    4.19]{mathew-naumann-noel}. Indeed, to show that $x$ is nilpotent, it
    suffices to show that the localization $R[1/x]$ is contractible. Since $px =
    0$, the localization $R[1/x]$ is an $\E{2}$-ring in which $p=0$, so the
    universal property of Theorem \ref{thom-univ} implies that there is an
    $\E{2}$-map $\H\FF_p\to R[1/x]$. It follows that the unit $R\to R[1/x]$
    factors through the Hurewicz map $R\to R\wedge\H\FF_p$. In particular, the
    multiplication by $x$ map on $R[1/x]$ factors as the indicated dotted map:
    $$\xymatrix{
	\Sigma^{|x|} R \ar[d] \ar[r]^-x & R \ar[r] \ar[d] & R[1/x].\\
	\H\FF_p \wedge \Sigma^{|x|} R \ar[r]^-x & R \wedge \H\FF_p \ar@{-->}[ur]
	& 
    }$$
    However, the bottom map is null (because $x$ has trivial $\H\FF_p$-Hurewicz
    image), so $x$ must be null in $\pi_\ast R[1/x]$. This is possible if and
    only if $R[1/x]$ is contractible, as desired.
    See Proposition \ref{thom-Z-nilp} for the analogous connection
    between Corollary \ref{hm-Z} and nilpotence.
\end{remark}

Since an argument similar to the proof of Theorem \ref{hm} will be necessary
later in Step 2 of Section \ref{the-proof}, we will recall a proof of this
theorem. The key non-formal input is the following result of Steinberger's from
\cite[Theorems III.2.2 and III.2.3]{bmms}:
\begin{theorem}[Steinberger]\label{steinberger}
    Let $\zeta_i$ denote the conjugate to the Milnor generators $\xi_i$ of the
    dual Steenrod algebra, and similarly for $\tau_i$ at odd primes. Then
    \begin{equation}\label{steinberger-eqn}
	Q^{p^i} \zeta_i = \zeta_{i+1}, \ Q^{p^j} \tau_j = \tau_{j+1}
    \end{equation}
    for $i,j+1>0$.
\end{theorem}

\begin{proof}[Proof of Theorem \ref{hm}]
    By Corollary \ref{thom-univ}, the Thom spectrum $(\Omega^2 S^3)^\mu$ is the
    free $\E{2}$-ring with a nullhomotopy of $p$. Since $\H\FF_p$ is an
    $\E{2}$-ring with a nullhomotopy of $p$, we obtain an $\E{2}$-map $(\Omega^2
    S^3)^\mu \to \H\FF_p$. To prove that this map is a $p$-complete equivalence,
    it suffices to prove that it induces an isomorphism on mod $p$ homology.

    The mod $p$ homology of $(\Omega^2 S^3)^\mu$ can be calculated directly via
    the Thom isomorphism
    $\H\FF_p \wedge (\Omega^2 S^3)^\mu \simeq \H\FF_p \wedge \Sigma^\infty_+
    \Omega^2 S^3$.
    Note that this is \emph{not} an equivalence as $\H\FF_p\wedge
    \H\FF_p$-comodules: the Thom twisting is highly nontrivial.

    For simplicity, we will now specialize to the case $p=2$, although the same
    proof works at odd primes. The homology of $\Omega^2 S^3$ is classical: it
    is a polynomial ring generated by applying $\E{2}$-Dyer-Lashof operations to
    a single generator $x_1$ in degree $1$. Theorem \ref{steinberger} implies
    that the same is true for the mod $2$ Steenrod algebra: it, too, is a
    polynomial ring generated by applying $\E{2}$-Dyer-Lashof operations to the
    single generator $\zeta_1 = \xi_1$ in degree $1$. Since the map $(\Omega^2
    S^3)^\mu \to \H\FF_2$ is an $\E{2}$-ring map, it preserves
    $\E{2}$-Dyer-Lashof operations on mod $p$ homology. By the above discussion,
    it suffices to show that the generator $x_1\in \H_\ast(\Omega^2 S^3)^\mu
    \cong \H_\ast(\Omega^2 S^3)$ in degree $1$ is mapped to $\zeta_1\in
    \H_\ast\H\FF_2$.

    To prove this, note that $x_1$ is the image of the generator in degree $1$
    in homology under the double suspension $S^1\to \Omega^2 S^3$, and that
    $\zeta_1$ is the image of the generator in degree $1$ in homology under the
    canonical map $\S/p\to \H\FF_p$. It therefore suffices to show that the Thom
    spectrum of the spherical fibration $S^1\to \B\GL_1(\S^\wedge_p)$ detecting
    $1-p$ is simply $\S/p$. This is an easy exercise.
\end{proof}

\begin{remark}
    When $p=2$, one does not need to $p$-complete in Theorem \ref{hm}: the map
    $S^1\to \B\GL_1(\S^\wedge_2)$ factors as $S^1\to \BO \to B\GL_1(\S)$, where
    the first map detects the M\"obius bundle over $S^1$, and the second map is
    the $J$-homomorphism. 
\end{remark}

\begin{notation}\label{cup-1-defn}
    Let $\mathcal{Q}_1$ denote the (operadic nerve of the) \emph{cup-1 operad}
    from \cite[Example 1.3.6]{lawson-dl}: this is the operad whose $n$th space
    is empty unless $n=2$, in which case it is $S^1$ with the antipodal action
    of $\Sigma_2$. We will need to slightly modify the definition of $\cQ_1$
    when localized at an odd prime $p$: in this case, it will denote the operad
    whose $n$th space is a point if $n<p$, empty if $n>p$, and when $n=p$ is the
    ordered configuration space $\Conf_p(\RR^2)$ with the permutation action of
    $\Sigma_p$. Any homotopy commutative ring admits the structure of a
    $\mathcal{Q}_1$-algebra at $p=2$, but at other primes it is slightly
    stronger to be a $\cQ_1$-algebra than to be a homotopy commutative ring. If
    $k\geq 2$, any $\E{k}$-algebra structure on a spectrum restricts to a
    $\cQ_1$-algebra structure.
\end{notation}

\begin{remark}\label{Q1-opn}
    As stated in \cite[Proposition 1.5.29]{lawson-dl}, the operation $Q_1$
    already exists in the mod $2$ homology of any \emph{$\cQ_1$-ring} $R$, where
    $\mathcal{Q}_1$ is the cup-1 operad from Notation \ref{cup-1-defn} --- the
    entire $\E{2}$-structure is not necessary. With our modification of $\cQ_1$
    at odd primes as in Remark \ref{cup-1-defn}, this is also true at odd
    primes.
\end{remark}

\begin{remark}\label{cup-1}
    We will again momentarily specialize to $p=2$ for convenience. Steinberger's
    calculation in Theorem \ref{steinberger} can be rephrased as stating that
    $Q_1 \zeta_i = \zeta_{i+1}$, where $Q_1$ is the \emph{lower-indexed}
    Dyer-Lashof operation. (See \cite[Page 59]{bmms} for this notation.) As in
    Remark \ref{Q1-opn}, the operation $Q_1$ already exists in the mod $p$
    homology of any $\cQ_1$-ring $R$.
    Since homotopy commutative rings are $\mathcal{Q}_1$-algebras in spectra,
    this observation can be used to prove results of W\"urgler (\cite[Theorem
    1.1]{wurgler}) and Pazhitnov-Rudyak (\cite[Theorem in
    Introduction]{hf2-old}).
\end{remark}

\begin{remark}\label{hopping}
    The argument with Dyer-Lashof operations and Theorem \ref{steinberger} used
    in the proof of Theorem \ref{hm} will be referred to as the
    \emph{Dyer-Lashof hopping argument}. It will be used again (in the same
    manner) in the proof of Theorem \ref{main-thm}.
\end{remark}

\begin{remark}
    Theorem \ref{hm} is \emph{equivalent} to Steinberger's calculation (Theorem
    \ref{steinberger}), as well as to B\"okstedt's calculation of
    $\mathrm{THH}(\FF_p)$ (as a ring spectrum, and not just the calculation of
    its homotopy). 
    Let us sketch an argument.
    First, Theorem \ref{steinberger} implies Theorem \ref{hm} (by the proof
    above). The other direction (i.e., the calculation \eqref{steinberger-eqn})
    can be argued by observing that the Thom isomorphism $\H\FF_p \wedge \H\FF_p
    \simeq \H\FF_p \wedge \Sigma^\infty_+ \Omega^2 S^3$ is an equivalence of
    $\E{2}$-$\H\FF_p$-algebras, so that the Dyer-Lashof operations are
    determined by the operations in $\H_\ast(\Omega^2 S^3;\FF_p)$. But the
    Dyer-Lashof operations are \emph{defined} by classes in $\H_\ast(\Omega^2
    S^3;\FF_p)$, and Theorem \ref{steinberger} is a consequence of the fact that
    the iterates of $Q_1$ on the generator of $\H_1(\Omega^2 S^3;\FF_p)$
    describe all the polynomial generators $\H_\ast(\Omega^2 S^3;\FF_p)$.

    It remains to argue that Theorem \ref{hm} is equivalent to the calculation
    that $\mathrm{THH}(\FF_p) \simeq \FF_p[\Omega S^3]$ as an
    $\E{1}$-$\FF_p$-algebra. This is showed in \cite[Remark
    1.5]{krause-nikolaus-dvr}.
\end{remark}

\subsection{No-go theorems for higher chromatic heights}

In light of Theorem \ref{hm} and Corollary \ref{hm-Z}, it is natural to wonder
if appropriate higher chromatic analogues of $\H\FF_p$ and $\H\Z$, such as
$\BP{n}$, $\bo$, or $\tmf$ can be realized as Thom spectra of spherical
fibrations. The answer is known to be negative (see \cite{mahowald-bo-bu,
rudyak, hood-tmf}) in many cases:
\begin{theorem}[Mahowald, Rudyak, Chatham]\label{no-go}
    There is no space $X$ with a spherical fibration $\mu:X\to \B\GL_1(\S)$ (even
    after completion) such that $X^\mu$ is equivalent to $\BP{1}$ or $\bo$.
    Moreover, there is no $2$-local loop space $X'$ with a spherical fibration
    determined by an H-map $\mu:X'\to \B\GL_1(\S^\wedge_2)$ such that ${X'}^\mu$
    is equivalent to $\tmf^\wedge_2$.
\end{theorem}
The proofs rely on calculations in the unstable homotopy groups of spheres.

\begin{remark}\label{idk}
    Although not written down anywhere, a slight modification of the argument
    used by Mahowald to show that $\bu$ is not the Thom spectrum of a spherical
    fibration over a loop space classified by an H-map can be used to show that
    $\BP{2}$ at $p=2$ (i.e., $\tmf_1(3)$) is not the Thom spectrum of a
    spherical fibration over a loop space classified by an H-map. We do not know
    a proof that $\BP{n}$ is not the Thom spectrum of a spherical fibration over
    a loop space classified by an H-map for all $n\geq 1$ and all primes, but we
    strongly suspect this to be true.
\end{remark}

\begin{remark}
    A lesser-known no-go result due to Priddy appears in \cite[Chapter
    2.11]{lewis-thesis}, where it is shown that $\BPP$ cannot appear as the Thom
    spectrum of a double loop map $\Omega^2 X\to \B\GL_1(\S)$. In fact, the
    argument shows that the same result is true with $\BPP$ replaced by $\BP{n}$
    for $n\geq 1$; we had independently come up with this argument for $\BP{1}$
    before learning about Priddy's argument. Since Lewis' thesis is not
    particularly easy to acquire, we give a sketch of Priddy's argument. By the
    Thom isomorphism and the calculation (see \cite[Theorem
    4.3]{lawson-naumann}, as well as \cite[Proposition
    1.7]{wilson-omega-spectrum} and \cite[Proposition 5.3]{angeltveit-rognes}):
    $$\H_\ast(\BP{n-1};\FF_p) \cong \begin{cases}
	\FF_2[\zeta_1^2, \cdots, \zeta_{n-1}^2, \zeta_n^2, \zeta_{n+1}, \cdots]
	& p=2,\\
	\FF_p[\zeta_1, \zeta_2, \cdots]\otimes \Lambda_{\FF_p}(\tau_n,
	\tau_{n+1}, \cdots) & p>2,
    \end{cases}$$
    we find that the mod $p$ homology of $\Omega^2 X$ would be isomorphic as an
    algebra to a polynomial ring on infinitely many generators, possibly
    tensored with an exterior algebra on infinitely many generators. The
    Eilenberg-Moore spectral sequence then implies that the mod $p$ cohomology
    of $X$ is given by
    $$\H^\ast(X;\FF_p) \cong \begin{cases}
	\FF_2[b_1, \cdots, b_n, c_{n+1}, \cdots]
	& p=2,\\
	\FF_p[b_1, b_2, \cdots]\otimes \Lambda_{\FF_p}(c_{n+1}, \cdots) & p>2,
    \end{cases}$$
    where $|b_i| = 2p^i$ and $|c_i| = 2p^{i-1}+1$. If $p$ is odd, then since
    $|b_1| = 2p$, we have $\mathrm{P}^p(b_1) = b_1^p$. Liulevicius'
    formula for $\mathrm{P}^1$ in terms of secondary cohomology operations
    (\cite[Theorem 1]{liulevicius}) allows us to write $\mathrm{P}^p(b_1)$ as a
    sum $c_0 \cR(b_1) + \sum_\gamma c_{0,\gamma} \Gamma_\gamma(b_1)$, where
    $\cR(b_1)$ is a coset in $\H^{2p + 4(p-1)}(X; \FF_p)$ and $\Gamma_\gamma$ is
    an operation of odd degree, so that $\Gamma_\gamma(b_1)$ is in odd degree.
    We will not need to know what exactly the sum is indexed by, or what any of
    these operations are. Observe that $\Gamma_\gamma$ kills $b_1$ because
    everything is concentrated in even degrees in the relevant range, and $\cR$
    also kills $b_1$ since $|\cR(b_1)| = 4(p -1) + 2p^i$ is never a sum of
    numbers of the form $2p^k$ when $p > 2$. Using this, one can conclude that
    $b_1^p = 0$, which is a contradiction. A similar calculation works at $p=2$,
    using Adams' study of secondary mod $2$ cohomology operations in
    \cite{adams-hopf}.
\end{remark}
\begin{remark}
    Using the calculations of $\THH(\ko)$ and $\THH(\ku)$ from \cite{thh-bp1},
    Angeltveit-Hill-Lawson show in \cite{angeltveit-hill-lawson-e3} that neither
    $\ko$ not $\ku$ can appear as the Thom spectrum of a double loop map
    $\Omega^2 X\to \B\GL_1(\S)$.
\end{remark}

Our primary goal in this project is to argue that the issues in Theorem
\ref{no-go} are alleviated if we replace $\B\GL_1(\S)$ with the delooping of the
space of units of an appropriate replacement of $\S$. In the next section, we
will construct these replacements of $\S$.

\newpage
\section{Some Thom spectra}\label{A-and-B}

In this section, we introduce certain $\E{1}$-rings; most of them appear as Thom
spectra over the sphere. The following table summarizes the spectra introduced
in this section and gives references to their locations in the text. The spectra
$A$ and $B$ were introduced in \cite{tmf-witten}.
\begin{table}[h!]
    \centering
    \begin{tabular}{c | c c c}
	Thom spectrum & Definition & ``Height'' & $\BPP$-homology\\
	\hline
	$T(n)$ & Theorem \ref{tn-def} & $n$ & Theorem \ref{tn-def}\\
	$y(n)$ and $y_\Z(n)$ & Definition \ref{yn} & $n$ & Proposition
	\ref{yn-bp}\\
	$A$ & Definition \ref{A-def} & $1$ & Proposition \ref{A-bp}\\
	$B$ & Definition \ref{B-def} & $2$ & Proposition \ref{B-bp}\\
	\hline
    \end{tabular}
    \vspace{0.5cm}
    \caption{Certain Thom spectra and their homologies.}
    \label{thom-definitions}
\end{table}

\subsection{Ravenel's $X(n)$ spectra}\label{ravenel-xn}

The proof of the nilpotence theorem in \cite{dhs-i,dhs-ii} crucially relied upon
certain Thom spectra arising from Bott periodicity; these spectra first appeared
in Ravenel's seminal paper \cite{ravenel-loc}.
\begin{definition}
    Let $X(n)$ denote the Thom spectrum of the $\E{2}$-map $\Omega \SU(n)
    \subseteq \BU \xrightarrow{J} \B\GL_1(\S)$, where the first map arises from
    Bott periodicity.
\end{definition}

\begin{example}
    The $\E{2}$-ring $X(1)$ is the sphere spectrum, while $X(\infty)$ is $\MU$.
    Since the map $\Omega \SU(n) \to \BU$ is an equivalence in dimensions
    $\leq 2n-2$, the same is true for the map $X(n) \to \MU$; the first
    dimension in which $X(n)$ has an element in its homotopy which is not
    detected by $\MU$ is $2n-1$.
\end{example}

\begin{remark}\label{Xn-e3}
    The $\E{2}$-structure on $X(n)$ does \emph{not} extend to an
    $\E{3}$-structure (see \cite[Example 1.5.31]{lawson-dl}). If $X(n)$ admits
    such an $\E{3}$-structure, then the induced map $\H_\ast(X(n))\to
    \H_\ast(\H\FF_p)$ on mod $p$ homology would commute with $\E{3}$-Dyer-Lashof
    operations.  However, we know that the image of $\H_\ast(X(n))$ in
    $\H_\ast(\H\FF_p)$ is $\FF_p[\zeta_1^2, \cdots, \zeta_n^2]$; since
    Steinberger's calculation (Theorem \ref{steinberger}) implies that
    $Q_2(\zeta_i^2) = \zeta_{i+1}^2$ via the relation $Q_2(x^2) = Q_1(x)^2$, we
    find that the image of $\H_\ast(X(n))$ in $\H_\ast(\H\FF_p)$ cannot be
    closed under the $\E{3}$-Dyer-Lashof operation $Q_2$.
\end{remark}

\begin{remark}
    The proof of the nilpotence theorem shows that each of the $X(n)$ detects
    nilpotence. However, it is known (see \cite[Theorem 3.1]{ravenel-loc}) that
    $\langle X(n)\rangle > \langle X(n+1)\rangle$.
\end{remark}

After localizing at a prime $p$, the spectrum $\MU$ splits as a wedge of
suspensions of $\BPP$; this splitting comes from the Quillen idempotent on $\MU$.
The same is true of the $X(n)$ spectra, as explained in \cite[Section
6.5]{green}: a multiplicative map $X(n)_{(p)}\to X(n)_{(p)}$ is determined by a
polynomial $f(x) = \sum_{0\leq i\leq n-1} a_i x^{i+1}$, with $a_0 = 1$ and
$a_i\in\pi_{2i}(X(n)_{(p)})$. One can use this to define a truncated form of the
Quillen idempotent $\epsilon_n$ on $X(n)_{(p)}$ (see \cite[Proposition
1.3.7]{hopkins-thesis}), and thereby obtain a summand of $X(n)_{(p)}$. We
summarize the necessary results in the following theorem.

\begin{theorem}\label{tn-def}
    Let $n$ be such that $p^n\leq k\leq p^{n+1}-1$. Then $X(k)_{(p)}$ splits as
    a wedge of suspensions of the spectrum $T(n) = \epsilon_{p^n}\cdot
    X(p^n)_{(p)}$.
    \begin{itemize}
	\item The map $T(n) \to \BPP$ is an equivalence in dimensions $\leq
	    |v_{n+1}|-2$, so there is an indecomposable element $v_i\in \pi_\ast
	    T(n)$ which maps to an indecomposable element in $\pi_\ast \BPP$ for
	    $0\leq i\leq n$.
	\item This map induces the inclusion $\BPP_\ast T(n) =
	    \BPP_\ast[t_1,\cdots,t_n] \subseteq \BPP_\ast(\BPP)$ on
	    $\BPP$-homology, and the inclusions $\FF_2[\zeta_1^2, \cdots,
	    \zeta_n^2] \subseteq \FF_2[\zeta_1^2, \zeta_2^2, \cdots]$ and
	    $\FF_p[\zeta_1, \cdots, \zeta_n] \subseteq \FF_2[\zeta_1, \zeta_2,
	    \cdots]$ on mod $2$ and mod $p$ homology.
	\item $T(n)$ is a homotopy associative and $\cQ_1$-algebra spectrum.
    \end{itemize}
\end{theorem}

\begin{remark}\label{tn-thom-def}
    It is known that that $T(n)$ admits the structure of an $\E{1}$-ring (see \cite[Section 7.5]{tn-e1}). We will interpret the phrase ``Thom spectrum $X^\mu$ of
    a map $\mu:X\to \B\GL_1(T(n))$'' where $\mu$ arises via a map
    $X\xrightarrow{\mu_0} \B\GL_1(X(p^{n+1}-1))$ to mean the base-change
    $X^{\mu_0} \wedge_{X(p^{n+1}-1)} T(n)$.
\end{remark}
It is believed that $T(n)$ in fact admits more structure (see \cite[Section 6]{angelini2017segal} for some discussion):
\begin{conjecture}\label{tn-e2}
    The $\cQ_1$-ring structure on $T(n)$ extends to an $\E{2}$-ring structure.
\end{conjecture}
\begin{remark}
    This is true at $p=2$ and $n=1$. Indeed, in this case $T(1) = X(2)$ is the
    Thom spectrum of the bundle given by the $2$-fold loop map $\Omega S^3 =
    \Omega^2 \BSU(2)\to \BU$, induced by the inclusion $\BSU(2)\to
    \mathrm{B^3 U} = \BSU$.
\end{remark}
\begin{remark}\label{t2}
    Conjecture \ref{tn-e2} is true at $p=2$ and $n=2$. The Stiefel manifold
    $V_2(\HH^2)$ sits in a fiber sequence
    $$S^3\to V_2(\HH^2)\to S^7.$$
    There is an equivalence $V_2(\HH^2) \simeq \Sp(2)$, so $\Omega V_2(\HH^2)$
    admits the structure of a double loop space. There is an $\E{2}$-map
    $\mu:\Omega V_2(\HH^2)\to \BU$, given by taking double loops of the
    composite
    $$\B\Sp(2)\to \BSU(4)\to \BSU \simeq \B^3 \U.$$
    The map $\mu$ admits a description as the left vertical map in the following
    map of fiber sequences:
    $$\xymatrix{
	\Omega V_2(\HH^2) \ar[r] \ar[d]^-\mu & \Omega S^7 \ar[r] \ar[d] & S^3
	\ar[d]\\
	\BU \ar[r] & \ast \ar[r] & \B^2 \U.
    }$$
    Here, the map $S^3\to \B^2\U$ detects the generator of $\pi_2(\BU)$ (which
    maps to $\eta\in \pi_2(\BGL_1(\S))$ under the J-homomorphism).  The Thom
    spectrum $\Omega V_2(\HH^2)^\mu$ is equivalent to $T(2)$, and it follows
    that $T(2)$ admits the structure of an $\E{2}$-ring.
    We do not know whether $T(n)$ is the Thom spectrum of a $p$-complete
    spherical fibration over some space for $n\geq 3$.
\end{remark}
It is possible to construct $X(n+1)$ as an $X(n)$-algebra (see also
\cite{beardsley}):
\begin{construction}\label{chin}
    There is a fiber sequence
    $$\Omega\SU(n)\to \Omega\SU(n+1)\to \Omega S^{2n+1}.$$
    According to Proposition \ref{thom}, the spectrum $X(n+1)$ is the Thom
    spectrum of an $\E{1}$-map $\Omega S^{2n+1}\to \B\GL_1(\Omega \SU(n))^\mu =
    \B\GL_1(X(n))$. This $\E{1}$-map is the extension of a map $S^{2n}\to
    \B\GL_1(X(n))$ which detects an element $\chi_n\in \pi_{2n-1} X(n)$. This
    element is equivalently determined by the map $\Sigma^\infty_+ \Omega^2
    S^{2n+1} \to X(n)$ given by the Thomification of the nullhomotopic composite
    $$\Omega^2 S^{2n+1} \to \Omega\SU(n)\to \Omega\SU(n+1)\to \Omega \SU \simeq
    \BU,$$
    where the first two maps form a fiber sequence. By Proposition \ref{thom},
    $X(n+1)$ is the free $\E{1}$-$X(n)$-algebra with a nullhomotopy of $\chi_n$.
\end{construction}
\begin{remark}\label{alt-constr}
    Another construction of the map $\chi_n\in \pi_{2n-1} X(n)$ from
    Construction \ref{chin} is as follows. There is a map $i:\CP^{n-1} \to
    \Omega \SU(n)$ given by sending a line $\ell\subseteq \cc^n$ to the loop
    $S^1\to \SU(n) = \Aut(\cc^n, \langle,\rangle)$ defined as follows:
    $\theta\in S^1$ is sent to the (appropriate rescaling of the) unitary
    transformation of $\cc^n$ sending a vector to its rotation around the line
    $\ell$ by the angle $\theta$. The map $i$ Thomifies to a stable map
    $\Sigma^{-2} \CP^n\to X(n)$. The map $\chi_n$ is then the composite
    $$S^{2n-1}\to \Sigma^{-2} \CP^n\to X(n),$$
    where the first map is the desuspension of the generalized Hopf map
    $S^{2n+1} \to \CP^n$ which attaches the top cell of $\CP^{n+1}$. The fact
    that this map is indeed $\chi_n$ follows immediately from the commutativity
    of the following diagram:
    $$\xymatrix{
	S^{2n-1} \ar[r] \ar[d] & \CP^{n-1} \ar[r] \ar[d] & \CP^{n} \ar[d]\\
	\Omega^2 S^{2n+1} \ar[r] & \Omega \SU(n) \ar[r] & \Omega \SU(n+1),}$$
    where the top row is a cofiber sequence, and the bottom row is a fiber
    sequence.
\end{remark}

An easy consequence of the observation in Construction \ref{chin} is the
following lemma.
\begin{lemma}\label{tn-thom}
    Let $\sigma_n\in \pi_{|v_{n+1}|-1} T(n)$ denote the element
    $\chi_{p^{n+1}-1}$. Then the Thom spectrum of the composite $\Omega
    S^{|v_{n+1}|+1} \to \B\GL_1(X(p^{n+1}-1))\to \B\GL_1(T(n))$ is equivalent to
    $T(n+1)$.
\end{lemma}

\begin{example}\label{alpha1}
    The element $\sigma_0\in\pi_{|v_1|-1} T(0) = \pi_{2p-3} \S_{(p)}$ is
    $\alpha_1$. 
\end{example}

\begin{example}\label{sigma1}
    Let us specialize to $p=2$. Theorem \ref{tn-def} implies that $\H_\ast T(n)
    \cong \FF_2[\zeta_1^2, \cdots, \zeta_n^2]$. Using this, one can observe that
    the $6$-skeleton of $T(1)$ is the smash product $C\eta\wedge C\nu$, and so
    $\sigma_1\in \pi_5(C\eta\wedge C\nu)$. This element can be described very
    explicitly: the cell structure of $C\eta\wedge C\nu$ is shown in Figure
    \ref{Ceta-Cnu}, and the element $\sigma_1$ shown corresponds to the map
    defined by the relation $\eta\nu=0$.
    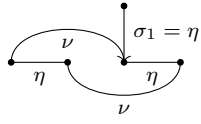
\begin{figure}
        \begin{tikzpicture}[scale=0.75]
            \draw [fill] (0, 0) circle [radius=0.05];
            \draw [fill] (1, 0) circle [radius=0.05];
            \draw [fill] (2, 1) circle [radius=0.05];
            \draw [fill] (2, 0) circle [radius=0.05];
            \draw [fill] (3, 0) circle [radius=0.05];
    
            \draw (0,0) to node[below] {\footnotesize{$\eta$}} (1,0);
            \draw (2,0) to node[below] {\footnotesize{$\eta$}} (3,0);
    
	    \draw [->] (2,1) to 
	    node[right] {\footnotesize{$\sigma_1 = \eta$}} (2,0);

            \draw (0,0) to[out=90,in=90] node[below] {\footnotesize{$\nu$}}
            (2,0);
	    \draw (1,0) to[out=-90,in=-90] node[below] {\footnotesize{$\nu$}}
            (3,0);
        \end{tikzpicture}
        \caption{$C\eta\wedge C\nu$ shown horizontally, with
        $0$-cell on the left. The element $\sigma_1$ is given by the map $\eta$
	on the $4$-cell defined by a nullhomotopy of $\eta\nu = 0\in \pi_4(S^0)$, as indicated in the diagram above.}
	\label{Ceta-Cnu}
    \end{figure}
\end{example}

\begin{example}\label{cobar}
    The element $\sigma_n$ in the Adams-Novikov spectral sequence for $T(n)$ is
    represented by the element $[t_{n+1}]$ in the cobar complex. See \cite[Section 1]{ravenel-infinite-descent}, where $\sigma_{n-1}$ is denoted by $\alpha(\hat{v}_1/p)$.
\end{example}
A calculation with the Adams-Novikov spectral sequence (as in \cite[Theorem 3.17]{ravenel-infinite-descent}) proves the following:
\begin{lemma}\label{chin-torsion}
    The class $\sigma_{n-1}$ is killed by $p$ in $\pi_{|v_n|-1} X(p^n-1)$.
\end{lemma}
\begin{proof}
The argument is essentially the same as the classical observation that $\alpha_1\in \pi_{2p-3}(S^0)$ is simple $p$-torsion.
As mentioned in Example \ref{cobar}, $\sigma_{n-1} = \alpha(\hat{v}_1/p)$ in the notation of \cite{ravenel-infinite-descent}. If $\Gamma(n+1) = \BPP_\ast(\BPP)/(t_1, \cdots, t_n)$ denotes the Hopf algebroid of \cite{ravenel-infinite-descent} (so that $\Ext_{\BPP_\ast \BPP}(\BPP_\ast, \BPP_\ast(T(n))) \cong \Ext_{\Gamma(n+1)}(\BPP_\ast, \BPP_\ast)$), then $\alpha$ is the connecting homomorphism in cohomology over $\Gamma(n+1)$ for the short exact sequence
$$0 \to \BPP_\ast \to p^{-1} \BPP_\ast \to p^{-1} \BPP_\ast/\BPP_\ast \to 0.$$
Since $\hat{v}_1/p$ is of order $p$ in $p^{-1} \BPP_\ast/\BPP_\ast$, we see that $\alpha(\hat{v}_1/p)$ is of order $p$ in the $E_2$-page of the Adams-Novikov spectral sequence computing $\pi_\ast T(n)$. The class $\alpha(\hat{v}_1/p)$ survives to the $E_\infty$-page; one observes there are are no possible additive extensions, so $p\sigma_{n-1} = 0\in \pi_\ast T(n)$.
\end{proof}
In particular, the element $\sigma_{n-1} = \chi_{p^n-1}\in \pi_{|v_n|-1}
X(p^n-1)$ is $p$-torsion, and the following is a consequence of Example
\ref{cobar}:
\begin{prop}\label{vn-toda}
    The class $\sigma_{n-1}\in \pi_{|v_n|-1} X(p^n-1)$ is null in $\pi_\ast X(p^n)$, and the Toda bracket $\langle p, \sigma_{n-1}, 1_{X(p^n)}\rangle$ in
    $\pi_{|v_n|} X(p^n)$ contains an indecomposable $v_n$.
\end{prop}
\begin{corollary}\label{CP-vn}
    The element $\sigma_{n-1}\in \pi_{|v_n|-1} X(p^n-1)$ lifts to
    $\pi_{|v_n|+1}(\CP^{|v_n|/2})$ along the map $\Sigma^{-2} \CP^{|v_n|/2} \to
    X(p^n-1)$.
\end{corollary}
\begin{proof}
    By Remark \ref{alt-constr}, the map $\sigma_{n-1}: S^{|v_n|-1} \to X(p^n-1)$
    is given by the composite of the generalized Hopf map $S^{|v_n|-1} \to
    \Sigma^{-2} \CP^{p^n-1}$ with the map $\Sigma^{-2} \CP^{p^n-1}\to X(p^n-1)$.
    Moreover, this generalized Hopf map is the desuspension of the unstable
    generalized Hopf map $S^{|v_n|+1} \to \CP^{p^n-1}$, and so $\sigma_{n-1}$
    lifts to an element of the unstable homotopy group
    $\pi_{|v_n|+1}(\CP^{|v_n|/2})$.
%
\end{proof}

\subsection{Related Thom spectra}

We now introduce several Thom spectra related to the $\E{1}$-rings $T(n)$
described in the previous section; some of these were introduced in
\cite{tmf-witten}. (Relationships to $T(n)$ will be further discussed in Section
\ref{thom-wood}.) For the reader's convenience, we have included a table of the
spectra introduced below with internal references to their definitions at the
beginning of this section.
\begin{remark}\label{homology-bp-fp}
Recall (e.g., from \cite[Section 4.4]{green}) that under the map
$$\BPP_\ast(\BPP) \cong \BPP_\ast[t_1, t_2, \cdots] \to \H_\ast(\BPP; \FF_p) \cong \begin{cases}
    \FF_2[\zeta_1^2, \zeta_2^2, \cdots] & p=2,\\
    \FF_p[\zeta_1, \zeta_2, \cdots] & p>2,
\end{cases}$$
the class $t_i$ is sent to $\zeta_i^2$ (resp. $\zeta_i$) modulo decomposables when $p=2$ (resp. $p>2$). Moreover, under the map
$$\H_\ast(\BPP; \FF_p) \to \H_\ast(\H\FF_p; \FF_p) \cong \begin{cases}
    \FF_2[\zeta_1, \zeta_2, \cdots] & p=2,\\
    \FF_p[\zeta_1, \zeta_2, \cdots] \otimes E(\tau_0, \tau_1, \cdots) & p>2,
\end{cases}$$
the classes $\zeta_{i+1}$ (resp. $\tau_i$) at $p=2$ (resp. $p>2$) detect a nullhomotopy of $v_i\in \pi_{2p^i-2} \BPP$ in $\H\FF_p \otimes \H\FF_p$. This implies, for instance, that if $X$ is a spectrum such that $\BPP_\ast(X) \simeq \BPP_\ast/(p, \cdots, v_{j-1})[t_1, \cdots, t_m]$ with $j\leq m$, then $\H_\ast(X; \FF_p) \cong \FF_p[\zeta_1, \cdots, \zeta_m] \otimes E(\tau_1, \cdots, \tau_{j-1})$ for $p>2$, and $\H_\ast(X; \FF_2) \cong \FF_p[\zeta_1^2, \cdots, \zeta_j^2, \zeta_{j+1}, \cdots, \zeta_m]$.
\end{remark}
The following Thom spectrum was introduced in \cite{mrs}.
\begin{definition}\label{yn}
    Let $y(n)$ denote the Thom spectrum of the composite
    $$\Omega J_{p^n-1}(S^2) \to \Omega^2 S^3\xrightarrow{1-p}
    \B\GL_1(\S^\wedge_p).$$
    If $J_{p^n-1}(S^2)\langle 2\rangle$ denotes the $2$-connected cover of
    $J_{p^n-1}(S^2)$, then let $y_\Z(n)$ denote the Thom spectrum of the
    composite
    $$\Omega J_{p^n-1}(S^2)\langle 2\rangle \to \Omega^2 S^3\langle 3\rangle \to \Omega^2 S^3 \xrightarrow{1-p} \B\GL_1(\S^\wedge_p),$$
    so that both $y(n)$ and $y_\Z(n)$ admit the structure of $\E{1}$-rings via \cite[Corollary 3.2]{barthel-thom}.
\end{definition}
\begin{prop}\label{yn-bp}
    As $\BPP_\ast \BPP$-comodules, we have
    $$\BPP_\ast(y(n)) \cong \BPP_\ast/I_n[t_1,\cdots,t_n], \ \BPP_\ast(y_\Z(n))
    \cong \BPP_\ast/(v_1, \cdots,v_{n-1})[t_1,\cdots,t_n],$$
    where $I_n$ denotes the invariant ideal $(p, v_1, \cdots, v_{n-1})$.
\end{prop}
\begin{proof}
The claim for $y(n)$ is \cite[Equation 2.8]{mrs}. There is an equivalence $y_\Z(n)/p \simeq y(n)$, so that $\BPP_\ast(y_\Z(n))/p \simeq \BPP_\ast(y(n))$. The Bockstein spectral sequence collapses, and the extensions on the $E_\infty$-page simply place $p$ in filtration $1$. This implies the second equivalence.
\end{proof}
One corollary is the following; this can be deduced from Proposition \ref{yn-bp} using Remark \ref{homology-bp-fp}.
We also refer to \cite[Lemma 2.3]{yz} for a direct proof of the following.
\begin{corollary}\label{yn-hfp}
    As $A_\ast$-comodules, we have
    $$\H_\ast(y(n);\FF_p) \cong \begin{cases}
	\FF_2[\zeta_1,\zeta_2,\cdots,\zeta_n] & p=2\\
	\FF_p[\zeta_1,\zeta_2,\cdots,\zeta_n] \otimes E(\tau_0, \cdots,
	\tau_{n-1}) & p\geq 3,
    \end{cases}$$
    and
    $$\H_\ast(y_\Z(n);\FF_p) \cong \begin{cases}
	\FF_2[\zeta_1^2,\zeta_2,\cdots,\zeta_n] & p=2\\
	\FF_p[\zeta_1,\zeta_2,\cdots,\zeta_n] \otimes E(\tau_1, \cdots,
	\tau_{n-1}) & p\geq 3.
    \end{cases}$$
\end{corollary}
We will now relate $y(n)$ and $y_\Z(n)$ to $T(n)$.
\begin{construction}\label{tn-in}
    Let $m\leq n$, and let $I_m$ be the ideal generated by
    $p,v_1,\cdots,v_{m-1}$, where the $v_i$ are some choices of indecomposables
    in $\pi_{|v_i|}(T(n))$ which form a regular sequence. Inductively define
    $T(n)/I_m$ as the cofiber of the map
    $$T(n)/I_{m-1} \xrightarrow{v_m\wedge 1} T(n) \wedge T(n)/I_{m-1}\to
    T(n)/I_{m-1}.$$
    The $\BPP$-homology of $T(n)/I_m$ is $\BPP_\ast/I_m[t_1,\cdots,t_n]$. The
    spectrum $T(n)/(v_1, \cdots, v_{m-1})$ is defined similarly.
\end{construction}
\begin{prop}\label{tn-yn}
    Let $p>2$. There is an equivalence between $T(n)/I_n$ (resp. $T(n)/(v_1,
    \cdots, v_{n-1})$) and the spectrum $y(n)$ (resp. $y_\Z(n)$) of Definition
    \ref{yn}.
\end{prop}
\begin{proof}
    We will prove the result for $y(n)$; the analogous proof works for
    $y_\Z(n)$. By \cite{gray}, the space $\Omega J_{p^n-1}(S^2)$ is homotopy
    commutative (since $p>2$). Moreover, the map $\Omega J_{p^n-1}(S^2)\to
    \Omega^2 S^3$ is an H-map, so $y(n)$ is a homotopy commutative $\E{1}$-ring
    spectrum. It is known (see \cite[Section 6.5]{green}) that homotopy
    commutative maps $T(n)\to y(n)$ are equivalent to partial complex
    orientations of $y(n)$, i.e., factorizations
    $$\xymatrix{
	\S \ar[r] \ar[dr]_-1 & \Sigma^{-2} \CP^{p^n-1}
	\ar@{-->}[d]^-{\gamma_n}\\
	& y(n).
    }$$
    Such a $\gamma_n$ indeed exists by obstruction theory: suppose $k < p^n-1$,
    and we have a map $\Sigma^{-2} \CP^k \to y(n)$. Since there is a cofiber
    sequence
    $$S^{2k-1} \to \Sigma^{-2} \CP^k \to \Sigma^{-2} \CP^{k+1}$$
    of spectra, the obstruction to extending along $\Sigma^{-2} \CP^{k+1}$ is an
    element of $\pi_{2k-1} y(n)$. However, the homotopy of $y(n)$ is
    concentrated in even degrees in the appropriate range, so a choice of
    $\gamma_n$ does indeed exist. Moreover, this choice can be made such that
    they fit into a compatible family in the sense that there is a commutative
    diagram
    $$\xymatrix{
	\Sigma^{-2} \CP^{p^n-1} \ar[r] \ar[d]_-{\gamma_n} & \Sigma^{-2}
	\CP^{p^{n+1}-1} \ar[d]^-{\gamma_{n+1}}\\
	y(n) \ar[r] & y(n+1).
    }$$
    The formal group law over $\H\FF_p$ has infinite height; this forces the
    elements $p,v_1,\cdots,v_{n-1}$ (defined for the ``$(p^n-1)$-bud'' on
    $\pi_\ast y(n)$) to vanish in the homotopy of $y(n)$. It follows that the
    orientation $T(n)\to y(n)$ constructed above factors through the quotient
    $T(n)/I_n$. The induced map $T(n)/I_n\to y(n)$ can be seen to be an
    isomorphism on homology (via, for instance, Definition \ref{yn} and
    Construction \ref{tn-in}).
\end{proof}
\begin{remark}\label{y1-v1}
    Since $y(n)$ has a $v_n$-self-map, we can form the spectrum $y(n)/v_n$; its
    mod $p$ homology is
    $$\H_\ast(y(n)/v_n;\FF_p) \cong 
    \begin{cases}
	\FF_2[\zeta_1, \cdots, \zeta_n]\otimes \Lambda_{\FF_2}(\zeta_{n+1}) &
	p=2\\
	\FF_p[\zeta_1,\cdots, \zeta_n]\otimes \Lambda_{\FF_p}(\tau_0, \cdots,
	\tau_{n-1}, \tau_n) & p\geq 3.
    \end{cases}$$
    It is in fact possible to give a construction of $y(1)/v_1$ as a spherical
    Thom spectrum. We will work at $p=2$ for convenience. Define $Q$ to be the
    fiber of the map $2\eta:S^3\to S^2$. There is a map of fiber sequences
    $$\xymatrix{
	Q \ar[r] \ar[d] & S^3 \ar[r]^-{2\eta} \ar[d] & S^2 \ar[d]\\
	\B\GL_1(\S) \ar[r] & \ast \ar[r] & \mathrm{B^2 GL}_1(\S).
    }$$
    By \cite[Theorem 3.7]{davis-mahowald}, the Thom spectrum of the leftmost map
    is $y(1)/v_1$.
\end{remark}

We end this section by recalling the definition of two Thom spectra which,
unlike $y(n)$ and $y_\Z(n)$, are not indexed by integers (we will see that they
are only defined at ``heights $1$ and $2$''). These were both studied in
\cite{tmf-witten}.
\begin{definition}\label{A-def}
    Let $S^4\to \BSpin$ denote the generator of $\pi_4 \BSpin \cong \Z$, and let
    $\Omega S^5\to \BSpin$ denote the extension of this map, which classifies a
    real vector bundle of virtual dimension zero over $\Omega S^5$. Let $A$
    denote the Thom spectrum of this bundle.
\end{definition}
\begin{remark}
    As mentioned in the introduction, the spectrum $A$ has been intensely
    studied by Mahowald and his coauthors in (for instance) \cite{mahowald-thom,
    davis-mahowald, mahowald-v2-periodic, mahowald-bo-res, mahowald-imj,
    mahowald-unell-bott}, where it is often denoted $X_5$.
\end{remark}
\begin{remark}
    The map $\Omega S^5\to \BSpin$ is one of $\E{1}$-spaces, so the Thom
    spectrum $A$ admits the structure of an $\E{1}$-ring with an $\E{1}$-map
    $A\to \MSpin$.
\end{remark}
\begin{remark}\label{A-univ-prop}
    There are multiple equivalent ways to characterize this Thom spectrum. For
    instance, the $J$-homomorphism $\BSpin\to \B\GL_1(\S)$ sends the generator
    of $\pi_4 \BSpin$ to $\nu\in \pi_4 \B\GL_1(\S) \cong \pi_3 \S$. The
    universal property of Thom spectra in Theorem \ref{thom-univ} shows that $A$
    is the free $\E{1}$-ring $\S\mmod\nu$ with a nullhomotopy of $\nu$. Note
    that $A$ is defined integrally, and not just $p$-locally for some prime $p$.
\end{remark}
\begin{remark}\label{map-A-T1}
    There is a canonical map $A \to T(1)$ of $\E{1}$-rings, constructed as
    follows. By the universal property of $A$, it suffices to prove that the
    unit $\S \to T(1)$ extends along the inclusion $\S \to C\nu$, i.e., that
    $\nu = 0\in \pi_3 T(1)$ up to units. To see this, let us compute $\pi_3 C\eta$ via the
    exact sequence
    $$\pi_3 S^1 \xar{\eta} \pi_3 S^0 \to \pi_3 C\eta \to \pi_2 S^0 \xar{\eta}
    \pi_1 S^0.$$
    This can be identified with
    $$\Z/2\{\eta^2\} \xar{\eta} \Z/8\{\nu\} \to \pi_3 C\eta \to \Z/2\{\eta\}
    \xar{\eta} \Z/2\{\eta^2\};$$
    the final map is an isomorphism, and the first map sends $\eta^2 \mapsto
    \eta^3 = 4\nu$. Therefore, $\pi_3 C\eta \cong \Z/4\{\nu\}$. Now, since the
    class in $\H_4(T(1);\FF_2)$ is detected by a nontrivial $\Sq^4$, the
    attaching map of the $4$-cell in $T(1)$ must be $\pm \nu$. Therefore, one of $\pm \nu$ must be null in $T(1)$, which implies that there must be a map $C\nu \to T(1)$ (or $C(-\nu) \to T(1)$) as claimed.
\end{remark}
The following result is \cite[Proposition 2.7]{tmf-witten}; it is proved there
at $p=2$, but the argument clearly works for $p=3$ too.
\begin{prop}\label{A-bp}
    There is an isomorphism $\BPP_\ast(A) \cong \BPP_\ast[y_2]$, where $|y_2| =
    4$. There is a map $A_{(p)}\to \BPP$. Under the induced map on
    $\BPP$-homology, $y_2$ maps to $t_1^2$ mod decomposables at $p=2$, and to
    $t_1$ mod decomposables at $p=3$.
\end{prop}
\begin{remark}
    For instance, when $p=2$, we have $\BPP_\ast(A) \cong \BPP_\ast[t_1^2 + v_1
    t_1]$.
\end{remark}
One corollary (using Remark \ref{homology-bp-fp}) is the following.
\begin{corollary}\label{A-hfp}
    As $A_\ast$-comodules, we have
    $$\H_\ast(A;\FF_p) \cong \begin{cases}
	\FF_2[\zeta_1^4] & p=2\\
	\FF_3[\zeta_1] & p=3\\
	\FF_p[x_4] & p\geq 5,
    \end{cases}$$
    where $x_4$ is a polynomial generator in degree $4$.
\end{corollary}
\begin{example}\label{sigma1-A}
    Let us work at $p=2$ for convenience. Example \ref{sigma1} showed that
    $\sigma_1$ is the element in $\pi_5(C\eta \wedge C\nu)$ given by the lift of
    $\eta$ to the $4$-cell (which is attached to the bottom cell by $\nu$) via a
    nullhomotopy of $\eta\nu$. In particular, $\sigma_1$ already lives in
    $\pi_5(C\nu)$, and as such defines an element of $\S\mmod\nu = A$ (by
    viewing $C\nu$ as the $4$-skeleton of $A$); note that, by construction, this
    element is $2$-torsion. The image of $\sigma_1\in \pi_5(A)$ under the
    canonical map of Remark \ref{map-A-T1} is its namesake in $\pi_5(T(1))$.
    See Figure \ref{A-15-skeleton}.
\end{example}
\begin{remark}\label{sigma1-integral}
    The element $\sigma_1\in \pi_5(A_{(2)})$ defined in Example \ref{sigma1-A}
    in fact lifts to an element of $\pi_5(A)$, because the relation $\eta\nu =
    0$ is true integrally, and not just $2$-locally. An alternate construction
    of this map is the following. The Hopf map $\eta_4:S^5\to S^4$ (which lives
    in the stable range) defines a map $S^5\to S^4\to \Omega S^5$ whose
    composite to $\BSpin$ is null (since $\pi_5(\BSpin) = 0$). Upon
    Thomification of the composite $S^5\to \Omega S^5\to \BSpin$, one therefore
    gets a map $S^5\to A$ whose composite with $A\to \MSpin$ is null. The map
    $S^5\to A$ is the element $\sigma_1\in \pi_5(A)$.
\end{remark}
\begin{figure}
    \begin{tikzpicture}[scale=0.75]
        \draw [fill] (0, 0) circle [radius=0.05];
        \draw [fill] (1, 0) circle [radius=0.05];
        \draw [fill] (1, 1) circle [radius=0.05];
        \draw [fill] (2, 0) circle [radius=0.05];
        \draw [fill] (3, 0) circle [radius=0.05];

        \draw (0,0) to node[below] {\footnotesize{$\nu$}} (1,0);
        \draw (1,0) to node[below] {\footnotesize{$2\nu$}} (2,0);
        \draw (2,0) to node[below] {\footnotesize{$3\nu$}} (3,0);

	\draw [->] (1,1) to node[left] {\footnotesize{$\eta$}} node[right]
	{\footnotesize{$\sigma_1$}} (1,0);

        \draw (0,0) to[out=-90,in=-90] node[below] {\footnotesize{$\sigma$}}
        (2,0);
    \end{tikzpicture}
    \caption{$15$-skeleton of $A$ at the prime $2$ shown horizontally, with
    $0$-cell on the left. The element $\sigma_1$ given by the map $\eta$ on the
    $4$-cell, as indicated in the diagram above.}
    \label{A-15-skeleton}
\end{figure}
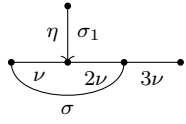

Finally, we have:
\begin{definition}\label{B-def}
    Let $\B N$ be the space defined by the homotopy pullback
    $$\xymatrix{
	\B N \ar[r] \ar[d] & S^{13} \ar[d]^-f\\
	\BO(9) \ar[r] & \BO(10),
    }$$
    where the map $f:S^{13}\to \BO(10)$ detects an element of $\pi_{12} \O(10)
    \cong \Z/12$. There is a fiber sequence
    $$S^9 \to \BO(9) \to \BO(10),$$
    and the image of $f$ under the boundary map in the long exact sequence of
    homotopy detects $2\nu\in \pi_{12}(S^9) \cong \Z/24$. In particular, there
    is a fiber sequence
    $$S^9\to \B N\to S^{13}.$$
    If $N$ is defined to be $\Omega \B N$, then there is a fiber sequence
    $$N\to \Omega S^{13}\to S^9.$$
    Define a map $N\to \BString$ via the map of fiber sequences
    $$\xymatrix{
	N \ar[r] \ar[d] & \Omega S^{13} \ar[r] \ar[d] & S^9 \ar[d]\\
	\BString \ar[r] & \ast \ar[r] & \mathrm{B^2 String},
    }$$
    where the map $S^9\to \mathrm{B^2 String}$ detects a generator of $\pi_8
    \BString$. Let $B$ denote the Thom spectrum of the induced bundle over $N$.
\end{definition}
\begin{remark}
    The map $N\to \BString$ is in fact one of $\E{1}$-spaces, so $B$ admits the
    structure of an $\E{1}$-ring. To prove this, it suffices to show that there
    is a map $\B N \to \mathrm{B^2 String}$. Recall that $\BString = \tau_{\geq
    8} \Omega^\infty \KO$, so the desired map is the same as a class in
    $\KO^1(\B N)$. Using the Serre spectral sequence for the fiber sequence
    defining $\B N$, one can calculate that there is a class in $\KO^1(\B N)$
    which lifts the generator of $\KO^1(S^9) \cong \pi_8 \KO \cong \Z$.
\end{remark}
We introduced the spectrum $B$ and studied its Adams-Novikov spectral sequence
in \cite{tmf-witten}. The Steenrod module structure of the $20$-skeleton of $B$
is shown in \cite[Figure 1]{tmf-witten}, and is reproduced here as Figure
\ref{B-cell}.
\begin{figure}
    \begin{tikzpicture}[scale=0.75]
        \draw [fill] (0, 0) circle [radius=0.05];
        \draw [fill] (2, 0) circle [radius=0.05];
        \draw [fill] (3, 0) circle [radius=0.05];
        \draw [fill] (4, 0) circle [radius=0.05];
        \draw [fill] (5, 0) circle [radius=0.05];

        \draw [fill] (3, 1) circle [radius=0.05];

	\draw (2,0) to node[below] {\footnotesize{$\nu$}} (3,0);
        \draw (4,0) to (5,0);

	\draw (0,0) to[out=90,in=90] node[below] {\footnotesize{$\sigma$}}
	(2,0);
        \draw (3,0) to[out=-90,in=-90] (5,0);

        \draw (0,0) to[out=-90,in=-90] (4,0);

	\draw [->] (3,1) to node[left] {\footnotesize{$\eta$}} node[right]
	{\footnotesize{$\sigma_2$}} (3,0);
    \end{tikzpicture}
    \caption{Steenrod module structure of the $20$-skeleton of $B$; the bottom
    cell (in dimension $0$) is on the left; straight lines are $\Sq^4$, and
    curved lines correspond to $\Sq^8$ and $\Sq^{16}$, in order of increasing
    length. The bottom two attaching maps of $B$ are labeled. The map $\sigma_2$
    is shown.}
    \label{B-cell}
\end{figure}
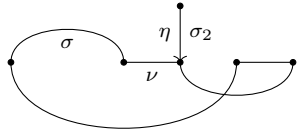
As mentioned in the introduction, the spectrum $B$ has
been briefly studied under the name $\ol{X}$ in
\cite{hopkins-mahowald-orientations}.
\begin{remark}\label{B-iterated-thom}
    As with $A$, there are multiple different ways to characterize $B$. There is
    a fiber sequence
    $$\Omega S^9 \to N\to \Omega S^{13},$$
    and the map $\Omega S^9 \to N\to \BString$ is an extension of the map
    $S^8\to \BString$ detecting a generator. Under the $J$-homomorphism
    $\BString\to \B\GL_1(\S)$, this generator maps to $\sigma\in \pi_8
    \B\GL_1(\S) \cong \pi_7 \S$, so the Thom spectrum of the bundle over $\Omega
    S^9$ determined by the map $\Omega S^9\to \BString$ is the free $\E{1}$-ring
    $\S\mmod\sigma$ with a nullhomotopy of $\nu$. Proposition \ref{thom} now
    implies that $N$ is the Thom spectrum of a map $\Omega S^{13}\to
    \B\GL_1(\S\mmod\sigma)$. While a direct definition of this map is not
    obvious, we note that the restriction to the bottom cell $S^{12}$ of the
    source detects an element $\wt{\nu}$ of $\pi_{12} \B\GL_1(\S\mmod\sigma)\cong
    \pi_{11} \S\mmod\sigma$. This in turn factors through the $11$-skeleton of
    $\S\mmod\sigma$, which is the same as the $8$-skeleton of $\S\mmod\sigma$
    (namely, $C\sigma$). This element is precisely a lift of the map $\nu:S^{11}
    \to S^8$ to $C\sigma$ determined by a nullhomotopy of $\sigma\nu$ in
    $\pi_\ast \S$. Although $\tilde{\nu}\in \pi_{11} C\sigma$ does not come from
    a class in $\pi_{11} \S$, its representative in the Adams spectral
    sequence for $C\sigma$ is the image of $h_{22}$ in the Adams spectral
    sequence for the sphere.
\end{remark}
The following result is \cite[Proposition 3.2]{tmf-witten}; it is proved there
at $p=2$, but the argument clearly works for $p\geq 3$ too.
\begin{prop}\label{B-bp}
    The $\BPP_\ast$-algebra $\BPP_\ast(B)$ is isomorphic to a polynomial ring
    $\BPP_\ast[b_4, y_6]$, where $|b_4| = 8$ and $|y_6| = 12$. There is a map
    $B_{(p)} \to \BPP$. On $\BPP_\ast$-homology, the elements $b_4$ and $y_6$
    map to $t_1^4$ and $t_2^2$ mod decomposables at $p=2$, and $y_6$ maps to
    $t_1^3$ mod decomposables at $p=3$.
\end{prop}
One corollary (using Remark \ref{homology-bp-fp}) is the following.
\begin{corollary}\label{B-hfp}
    As $A_\ast$-comodules, we have
    $$\H_\ast(B;\FF_p) \cong \begin{cases}
	\FF_2[\zeta_1^8,\zeta_2^4] & p=2\\
	\FF_3[\zeta_1^3, b_4] & p=3\\
	\FF_5[\zeta_1, x_{12}] & p=5\\
	\FF_p[x_8, x_{12}] & p\geq 7,
    \end{cases}$$
    where $x_8$ and $x_{12}$ are polynomial generators in degree $8$ and $12$,
    and $b_4$ is an element in degree $8$.
\end{corollary}
\begin{example}\label{sigma2-B}
    For simplicity, let us work at $p=2$. There is a canonical ring map $B\to
    T(2)$, and the element $\sigma_2\in \pi_{13} T(2)$ lifts to $B$. We can be
    explicit about this: the $12$-skeleton of $B$ is shown in Figure
    \ref{B-cell}, and $\sigma_2$ is the element of $\pi_{13}(B)$ existing thanks
    to the relation $\eta\nu = 0$ and the fact that the Toda bracket $\langle
    \eta, \nu, \sigma\rangle$ contains $0$. This also shows that $\sigma_2\in
    \pi_{13}(B)$ is $2$-torsion.
\end{example}
\begin{remark}\label{sigma2-integral}
    The element $\sigma_2\in \pi_{13}(B_{(2)})$ defined in Example
    \ref{sigma2-B} in fact lifts to an element of $\pi_{13}(B)$, because the
    relations $\nu\sigma = 0$, $\eta\nu = 0$, and $0\in \langle \eta, \nu,
    \sigma\rangle$ are all true integrally, and not just $2$-locally. An
    alternate construction of this map $S^{13}\to B$ is the following. The Hopf
    map $\eta_{12}:S^{13}\to S^{12}$ (which lives in the stable range) defines a
    map $S^{13}\to S^{12}\to \Omega S^{13}$. Moreover, the composite $S^{13}\to
    \Omega S^{13}\to S^9$ is null, since it detects an element of $\pi_{13}(S^9)
    = 0$; choosing a nullhomotopy of this composite defines a lift $S^{13}\to
    N$. (In fact, this comes from a map $S^{14}\to \B N$.) The composite
    $S^{13}\to N\to \BString$ is null (since $\pi_{13}(\BString) = 0$). Upon
    Thomification, we obtain a map $S^{13}\to B$ whose composite with $B\to
    \MString$ is null; the map $S^{13}\to B$ is the element $\sigma_2\in
    \pi_{13}(B)$.
\end{remark}
The following theorem packages some information contained in this section.
\begin{theorem}\label{sigman-exists}
    Let $R$ denote any of the spectra in Table \ref{thom-definitions}, and let
    $n$ denote its ``height''. If $R = T(n), y(n)$, or $y_\Z(n)$, then there is
    a map $T(n)\to R$, and if $R = A$ (resp. $B$), then there is a map from $R$ to
    $T(1)$ (resp. $T(2)$). In the first three cases, there is an element
    $\sigma_n\in \pi_{|v_{n+1}|-1} R$ coming from $\sigma_n\in \pi_{|v_{n+1}|-1}
    T(n)$, and in the cases $R = A$ and $B$, there are elements $\sigma_1\in
    \pi_5(A)$ and $\sigma_2\in \pi_{13}(B)$ mapping to the corresponding
    elements in $T(1)_{(2)}$ and $T(2)_{(2)}$, respectively.  Moreover,
    $\sigma_n$ is $p$-torsion in $\pi_\ast R$; similarly, $\sigma_1$ and
    $\sigma_2$ are $2$-torsion in $\pi_\ast A_{(2)}$ and $\pi_\ast B_{(2)}$.
\end{theorem}
\begin{proof}
    The existence statement for $T(n)$ is contained in Theorem \ref{tn-def},
    while the torsion statement is the content of Lemma \ref{chin-torsion}. The
    claims for $y(n)$ and $y_\Z(n)$ now follow from Proposition \ref{tn-yn}. The
    existence and torsion statements for $A$ and $B$ are contained in Examples
    \ref{sigma1-A} and \ref{sigma2-B}.
\end{proof}
The elements in Theorem \ref{sigman-exists} can in fact be extended to infinite
families; this is discussed in Section \ref{infinite}.

\subsection{Centers of Thom spectra}

In this section, we review some of the theory of $\E{k}$-centers and state
Conjecture \ref{centrality-conj}. We begin with the following important result,
and refer to \cite{francis} and \cite[Section 5.5.4]{HA} for proofs.
\begin{theorem}[{\cite[Example 5.5.4.16]{HA}, \cite[Definition
    2.5]{francis}}]\label{fact-hmlgy}
    Let $\cC$ be a symmetric monoidal presentable $\infty$-category, and let $A$
    be an $\E{k}$-algebra in $\cC$. Then the category of $\E{k}$-$A$-modules is
    equivalent to the category of left modules over the factorization homology
    $U(A) = \int_{S^{k-1}\times \RR} A$ (known as the \emph{enveloping algebra}
    of $A$), which is an $\E{1}$-algebra in $\cC$.
\end{theorem}
\begin{definition}\label{ek-1-center}
    The \emph{$\E{k+1}$-center} $\fr{Z}(A)$ of an $\E{k}$-algebra $A$ in $\cC$
    is the ($\E{k+1}$-)Hochschild cohomology $\End_{U(A)}(A)$, where $A$ is
    regarded as a left module over its enveloping algebra via Theorem
    \ref{fact-hmlgy}.
\end{definition}
\begin{remark}
    We are using slightly different terminology than the one used in
    \cite[Section 5.3]{HA}: our $\E{k+1}$-center is his $\E{k}$-center. In other
    words, Lurie's terminology expresses the structure on the input, while our
    terminology expresses the structure on the output.
\end{remark}
The following proposition summarizes some results from \cite{francis} and
\cite[Section 5.3]{HA}.
\begin{prop}[{\cite[Theorem 5.3.2.5]{HA}, \cite[Theorem
    1.1]{francis}}]\label{center-theorems}
    The $\E{k+1}$-center $\fr{Z}(A)$ of an $\E{k}$-algebra $A$ in a symmetric
    monoidal presentable $\infty$-category $\cC$ exists, and satisfies the
    following properties:
    \begin{enumerate}
	\item $\fr{Z}(A)$ is the universal $\E{k}$-algebra of $\cC$ which fits
	    into a commutative diagram
	    $$\xymatrix{ A \ar[r] \ar@{=}[dr] & A\otimes \fr{Z}(A) \ar[d]\\
		& A}$$
		in $\Alg_{\E{k}}(\cC)$.
	    \item The $\E{k}$-algebra $\fr{Z}(A)$ of $\cC$ defined via this
		universal property in fact admits the structure of an
		$\E{k+1}$-algebra in $\cC$.
	    \item There is a fiber sequence
		$$\GL_1(\fr{Z}(A)) \to \GL_1(A)\to \Omega^{k-1}
		\End_{\Alg_{\E{k}}(\cC)}(A)$$
		of $k$-fold loop spaces.
    \end{enumerate}
\end{prop}

In the sequel, we will need a more general notion:
\begin{definition}\label{ekm-center}
    Let $m\geq 1$. The \emph{$\E{k+m}$-center} $\fr{Z}_{{k+m}}(A)$ of an
    $\E{k}$-algebra $A$ in a presentable symmetric monoidal $\infty$-category
    $\cC$ with all limits is defined inductively as the $\E{k+m}$-center of
    the $\E{k+m-1}$-center $\fr{Z}_{k+m-1}(A)$. In other words, it is the
    universal $\E{k+m}$-algebra of $\cC$ which fits into a commutative diagram 
    $$\xymatrix{ \fr{Z}_{k+m-1}(A) \ar[r] \ar@{=}[dr] & \fr{Z}_{k+m-1}(A)\otimes
    \fr{Z}_{k+m}(A) \ar[d]\\
    & \fr{Z}_{k+m-1}(A)}$$
    in $\Alg_{\E{k+m-1}}(\cC)$.
\end{definition}
Proposition \ref{center-theorems} gives:
\begin{corollary}\label{center-ek}
    Let $m\geq 1$. The $\E{k+m-1}$-algebra $\fr{Z}_{{k+m}}(A)$ associated to an
    $\E{k}$-algebra object $A$ of $\cC$ exists, and in fact admits the structure
    of an $\E{k+m}$-algebra in $\cC$.
\end{corollary}
We can now finally state Conjecture \ref{centrality-conj}:
\begin{conj-cent}
    Let $n\geq 0$ be an integer. Let $R$ denote $X(p^{n+1}-1)_{(p)}$, $A$ (in
    which case $n=1$), or $B$ (in which case $n=2$). Then the element
    $\sigma_n\in \pi_{|\sigma_n|} R$ lifts to the $\E{3}$-center $\fr{Z}_3(R)$
    of $R$, and is $p$-torsion in $\pi_\ast \fr{Z}_3(R)$ if $R = X(p^{n+1} -
    1)_{(p)}$, and is $2$-torsion in $\pi_\ast \fr{Z}_3(R)$ if $R = A$ or $B$.
\end{conj-cent}
\begin{remark}
    If $R$ is $A$ or $B$, then $\fr{Z}_3(R)$ is the $\E{3}$-center of the
    $\E{2}$-center of $R$. This is a rather unwieldy object, so it would be
    quite useful to show that the $\E{1}$-structure on $A$ or $B$ admits an
    extension to an $\E{2}$-structure; we do not know if such extensions exist.
    Since neither $\Omega S^5$ nor $N$ admit the structure of a double loop
    space, such an $\E{2}$-structure would not arise from their structure as
    Thom spectra. In any case, if such extensions do exist, then $\fr{Z}_3(R)$
    in Conjecture \ref{centrality-conj} should be interpreted as the
    $\E{3}$-center of the $\E{2}$-ring $R$. However, we showed in \cite[Theorem
    4.2]{hodge} that $(\tmf \wedge A)[x_2]$ admits an $\E{2}$-algebra structure,
    where $|x_2| = 2$.
\end{remark}
\begin{remark}
    In the introduction, we stated Conjecture \ref{hahn-conjecture}, which
    instead asked about whether $v_n\in \pi_{|v_n|} X(p^n)$ lifts to $\pi_\ast
    \fr{Z}_3(X(p^n))$. It is natural to ask about the connection between
    Conjecture \ref{centrality-conj} and Conjecture \ref{hahn-conjecture}.
    Proposition \ref{vn-toda} implies that if $\fr{Z}_3(X(p^n))$ admitted an
    $X(p^n-1)$-orientation factoring the canonical $X(p^n-1)$-orientation
    $X(p^n-1)\to X(p^n)$, and $\sigma_{n-1}\in \pi_{|v_n|-1} X(p^n-1)$ was
    killed by the map $X(p^n-1)\to \fr{Z}_3(X(p^n))$, then Conjecture
    \ref{centrality-conj} implies Conjecture \ref{hahn-conjecture}. However, we
    do not believe that either of these statements are true.
\end{remark}
\begin{remark}\label{centrality-remark}
    One of the main results of \cite{klang} implies that the $\E{3}$-center of
    $X(n)$ (which, recall, is the Thom spectrum of a bundle over $\Omega^2
    \BSU(n)$) is $\Hom_{\SU(n)_+}(\S, X(n)) \simeq X(n)^{h\SU(n)}$, where
    $\SU(n)$ acts on $X(n)$ by a Thomification of the conjugation action on
    $\Omega \SU(n)$.
\end{remark}
\begin{remark}\label{grassmannian}
    Note that the conjugation action of $\SU(n)$ on $X(n)$ can be described very
    explicitly, via a concrete model for $\Omega \SU(n)$. As explained in
    \cite{pressley-segal, zhu-grass}, if $G$ is a reductive linear algebraic
    group over $\cc$, the loop space $\Omega G(\cc)$ of its complex points
    (viewed as a complex Lie group) is equivalent to the homogeneous space
    $G(\cc(\!(t)\!))/G(\cc[\![t]\!])$; this is also commonly studied as the
    complex points of the affine Grassmannian $\Gr_G$ of $G$. The conjugation
    action of $G(\cc)$ on $\Omega G(\cc)$ arises by restricting the descent (to
    $G(\cc(\!(t)\!))/G(\cc[\![t]\!])$) of the translation action by
    $G(\cc[\![t]\!])$ on $G(\cc(\!(t)\!))$ to the subgroup $G(\cc) \subseteq
    G(\cc[\![t]\!])$. Setting $G = \SL_n$ gives a description of the conjugation
    action of $\SU(n)$ on $\Omega \SU(n)$. In light of its connections to
    geometric representation theory, we believe that there may be an
    algebro-geometric approach to proving that $\chi_n$ is $\SU(n)$-trivial in
    $X(n)$ and in $\Omega \SU(n)$.
\end{remark}
\begin{example}\label{X2-central}
    The element $\chi_2\in \pi_3 X(2)$ is central. To see this, note that
    $\alpha\in \pi_\ast R$ (where $R$ is an $\E{k}$-ring) is in the
    $\E{k+1}$-center of $R$ if and only if $\alpha$ is in the $\E{k+1}$-center
    of $R_{(p)}$ for all primes $p\geq 0$. It therefore suffices to show that
    $\chi_2$ is central after $p$-localizing for all $p$.  First, note that
    $\chi_2$ is torsion, so it is nullhomotopic (and therefore central) after
    rationalization. Next, if $p>2$, then $X(2)_{(p)}$ splits as a wedge of
    suspensions of spheres. If $\chi_2$ is detected in $\pi_3$ of a sphere
    living in dimension $3$, then it could not be torsion, so it must be
    detected in $\pi_3$ of a sphere living in dimension $3-k$ for some $0\leq
    k\leq 2$. If $k=1$ or $2$, then $\pi_3(S^{3-k})$ is either $\pi_1(S^0)$ or
    $\pi_2(S^0)$, but both of these groups vanish for $p>2$. Therefore, $\chi_2$
    must be detected in $\pi_3$ of the sphere in dimension $0$, i.e., in $\pi_3
    X(1)$. This group vanishes for $p>3$, and when $p=3$, it is isomorphic to
    $\Z/3$ (generated by $\alpha_1$). Since $X(1) = S^0$ is an $\Eoo$-ring, we
    conclude that $\chi_2$ is central in $X(2)_{(p)}$ for all $p>2$.

    At $p=2$, we know the cell structure of $X(2)$ in the bottom few dimensions
    (see Example \ref{sigma1}; note that $\sigma_1$ is \emph{not} $\chi_2$). In
    dimensions $\leq 3$, it is equivalent to $C\eta$, so $\pi_3 X(2) \cong \pi_3
    C\eta$. However, it is easy to see that the canonical map $\pi_3 \S \simeq
    \Z/8\{\nu\} \to \pi_3 C\eta$ is surjective and exhibits an isomorphism
    $\pi_3 C\eta \cong \Z/4\{\nu\}$. Therefore, $\chi_2$ is in the image of the
    unit $\S \to X(2)$, and is therefore vacuously central. We conclude from the
    above discussion that $\chi_2$ is indeed central in $X(2)$.
\end{example}

\newpage
\section{Review of some unstable homotopy theory}\label{unstable-review}

\subsection{Charming and Gray maps}\label{gray}

A major milestone in unstable homotopy theory was Cohen-Moore-Neisendorfer's
result on the $p$-exponent of unstable homotopy groups of spheres from
\cite{cmn-1, cmn-2, cmn-3}. They defined for all $p>2$ and $k\geq 1$ a map
$\phi_n:\Omega^2 S^{2n+1}\to S^{2n-1}$ (the integer $k$ is assumed implicit)
such that the composite of $\phi_n$ with the double suspension $E^2:S^{2n-1} \to
\Omega^2 S^{2n+1}$ is homotopic to the $p^k$-th power map. By induction on $n$,
they concluded via a result of Selick's (see \cite{selick}) that $p^n$ kills the
$p$-primary component of the homotopy of $S^{2n+1}$. Such maps will be important
in the rest of this article, so we will isolate their desired properties in the
definition of a \emph{charming map}, inspired by \cite{gray-maps}. (Our choice
of terminology is non-standard, and admittedly horrible, but it does not seem
like the literature has chosen any naming convention for the sort of maps we
desire.)
\begin{definition}\label{gray-map-def}
    A $p$-local map $f:\Omega^2 S^{2np+1}\to S^{2np-1}$ is called a \emph{Gray
    map} if the composite of $f$ with the double suspension $E^2$ is the degree
    $p$ map, and the composite
    $$\Omega^2 S^{2n+1} \xrightarrow{\Omega H} \Omega^2 S^{2np+1}\xrightarrow{f}
    S^{2np-1}$$
    is nullhomotopic. Moreover, a $p$-local map $f:\Omega^2 S^{2np+1}\to
    S^{2np-1}$ is called a \emph{charming map} if the composite of $f$ with the
    double suspension $E^2$ is the degree $p$ map, the fiber of $f$ admits the
    structure of a $\cQ_1$-space, and if there is a space $BK$ which sits in a
    fiber sequence
    $$S^{2np-1} \to BK\to \Omega S^{2np+1}$$
    such that the boundary map $\Omega^2 S^{2np+1}\to S^{2np-1}$ is homotopic to
    $f$.
\end{definition}
\begin{remark}
    If $f$ is a charming map, then the fiber of $f$ is a loop space. Indeed,
    $\fib(f) \simeq \Omega BK$.
\end{remark}
\begin{example}\label{cmn-map}
    Let $f$ denote the Cohen-Moore-Neisendorfer map with $k=1$. Anick proved
    (see \cite{anick-book, anick-elementary}) that the fiber of $f$ admits a
    delooping, i.e., there is a space $T^{2np+1}(p)$ (now known as an
    \emph{Anick space}) which sits in a fiber sequence
    $$S^{2np-1} \to T^{2np+1}(p)\to \Omega S^{2np+1}.$$
    It follows that $f$ is a charming map.
\end{example}
\begin{remark}\label{anick-S3}
    We claim that $T^{2p+1}(p) = \Omega S^3\langle 3\rangle$, where $S^3\langle
    3\rangle$ is the $3$-connected cover of $S^3$. To prove this, we will
    construct a $p$-local fiber sequence
    $$S^{2p-1} \to \Omega S^3\langle 3\rangle\to \Omega S^{2p+1}.$$
    This fiber sequence was originally constructed by Toda in \cite{toda}. To
    construct this fiber sequence, we first note that there is a $p$-local fiber
    sequence
    $$S^{2p-1} \to J_{p-1}(S^2)\to \CP^\infty,$$
    where the first map is the factorization of $\alpha_1:S^{2p-1}\to \Omega
    S^3$ through the $2(p-1)$-skeleton of $\Omega S^3$, and the second map is
    the composite $J_{p-1}(S^2)\to \Omega S^3\to \CP^\infty$. This fiber
    sequence is simply an odd-primary version of the Hopf fibration $S^3\to
    S^2\to \CP^\infty$; the identification of the fiber of the map
    $J_{p-1}(S^2)\to \CP^\infty$ is a simple exercise with the Serre
    spectral sequence. Next, we have the EHP sequence
    $$J_{p-1}(S^2)\to \Omega S^3\to \Omega S^{2p+1}.$$
    Since $\Omega S^3\langle 3\rangle$ is the fiber of the map $\Omega S^3\to
    \CP^\infty$, the desired fiber sequence is obtained by taking vertical
    fibers in the following map of fiber sequences:
    $$\xymatrix{
	J_{p-1}(S^2) \ar[r] \ar[d] & \Omega S^3 \ar[r] \ar[d] & \Omega S^{2p+1}
	\ar[d]\\
	\CP^\infty \ar@{=}[r] & \CP^\infty \ar[r] & \ast.
    }$$
\end{remark}
\begin{example}
    Let $W_n$ denote the fiber of the double suspension $S^{2n-1}\to \Omega^2
    S^{2n+1}$. Gray proved in \cite{gray-1, gray-2} that $W_n$ admits a
    delooping $BW_n$, and that after $p$-localization, there is a fiber sequence
    $$BW_n\to \Omega^2 S^{2np+1} \xrightarrow{f} S^{2np-1}$$
    for some map $f$. As suggested by the naming convention, $f$ is a Gray map.
\end{example}

As proved in \cite{gray-maps}, Gray maps satisfy an important rigidity property:
\begin{prop}[Selick-Theriault]\label{selick-theriault}
    The fiber of any Gray map admits an H-space structure, and is H-equivalent
    to $BW_n$.
\end{prop}
\begin{remark}\label{motivate-conj}
    It has been conjectured by Cohen-Moore-Neisendorfer and Gray in the papers
    cited above that there is an equivalence $BW_n\simeq \Omega T^{2np+1}(p)$,
    and that $\Omega T^{2np+1}(p)$ retracts off of $\Omega^2 P^{2np+1}(p)$ as an
    H-space, where $P^k(p)$ is the mod $p$ Moore space $S^{k-1}\cup_p e^k$
    with top cell in dimension $k$. For our purposes, we shall require something
    slightly stronger: namely, the retraction should be one of $\cQ_1$-spaces.
    The first part of this conjecture would follow from Proposition
    \ref{selick-theriault} if the Cohen-Moore-Neisendorfer map were a Gray map.
    In \cite{amelotte-kervaire}, it is shown that the existence of $p$-primary
    elements of Kervaire invariant one would imply equivalences of the form
    $BW_{p^{n-1}}\simeq \Omega T^{2p^n+1}(p)$.
\end{remark}
Motivated by Remark \ref{motivate-conj} and Proposition \ref{selick-theriault},
we state the following conjecture; it is slightly weaker than the conjecture
mentioned in Remark \ref{motivate-conj}, and is an amalgamation of slight
modifications of conjectures of Cohen, Moore, Neisendorfer, Gray, and Mahowald
in unstable homotopy theory, as well as an analogue of Proposition
\ref{selick-theriault}. (For instance, we strengthen having an H-space
retraction to having a $\cQ_1$-space retraction).
\begin{conj-moore}
    The following statements are true:
    \begin{enumerate}
	\item The homotopy fiber of any charming map is equivalent as a loop
	    space to the loop space on an Anick space.
	\item There exists a $p$-local charming map $f:\Omega^2 S^{2p^n+1}\to
	    S^{2p^n-1}$ whose homotopy fiber admits a $\cQ_1$-space retraction
	    off of $\Omega^2 P^{2p^n+1}(p)$. There are also integrally defined
	    maps $\Omega^2 S^9\to S^7$ and $\Omega^2 S^{17} \to S^{15}$ whose
	    composite with the double suspension on $S^7$ and $S^{15}$
	    respectively is the degree $2$ map, whose homotopy fibers $K_2$ and
	    $K_3$ (respectively) admit deloopings, and which admits a
	    $\cQ_1$-space retraction off $\Omega^2 P^9(2)$ and $\Omega^2
	    P^{17}(2)$ (respectively).
    \end{enumerate}
\end{conj-moore}
\begin{remark}\label{hm-Z-alt}
    Conjecture \ref{moore-splitting} is already not known when $n=1$. In this
    case, it asserts that $\Omega^2 S^3\langle 3\rangle$ retracts off of
    $\Omega^2 P^{2p+1}(p)$. A theorem of Selick's states that $\Omega^2
    S^3\langle 3\rangle$ retracts off of $\Omega^2 S^{2p+1}\{p\}$ for $p$ odd,
    where $\Omega^2 S^{2p+1}\{p\}$ is the fiber of the degree $p$ map on
    $\Omega^2 S^{2p+1}$.  This implies that $\Omega^2 S^3\langle 3\rangle$
    retracts off of $\Omega^3 P^{2p+2}(p)$.  In \cite[Observation
    9.2]{cohen-course}, the question of whether $\Omega^2 S^3\langle 3\rangle$
    retracts off of $\Omega^2 P^{2p+1}(p)$ was shown to be equivalent to the
    question of whether there is a map $\Sigma^2 \Omega^2 S^3\langle 3\rangle\to
    P^{2p+1}(p)$ which is onto in homology. Some recent results regarding
    Conjecture \ref{moore-splitting} for $n=1$ can be found in
    \cite{kahn-priddy-unstable}.

    It follows that a retraction of $\Omega^2 S^3\langle 3\rangle$ off $\Omega^2
    P^{2p+1}(p)$ will be compatible with the canonical map $\Omega^2 S^3\langle
    3\rangle\to \Omega^2 S^3$ in the following manner. The $p$-torsion element
    $\alpha_1\in \pi_{2p}(S^3)$ defines a map $P^{2p-1}(p)\to \Omega^2 S^3$,
    which extends to an $\E{2}$-map $\Omega^2 P^{2p+1}(p)\to \Omega^2 S^3$. We
    will abusively denote this extension by $\alpha_1$. The resulting composite
    $$\Omega^2 S^3\langle 3\rangle \to \Omega^2 P^{2p+1}(p)
    \xrightarrow{\alpha_1} \Omega^2 S^3$$
    is homotopic to the canonical map $\Omega^2 S^3\langle 3\rangle\to \Omega^2
    S^3$.

    The element $\alpha_1\in \pi_{2p-3}(\S_{(p)})$ defines a map $S^{2p-2}\to
    \BGL_1(\S_{(p)})$, and since it is $p$-torsion, admits an extension to a map
    $P^{2p-1}(p)\to \BGL_1(\S_{(p)})$. (This extension is in fact unique,
    because $\pi_{2p-1}(\BGL_1(\S_{(p)})) \cong \pi_{2p-2}(\S_{(p)})$ vanishes.)
    Since $\BGL_1(\S_{(p)})$ is an infinite loop space, this map further extends
    to a map $\Omega^2 P^{2p+1}(p)\to \BGL_1(\S_{(p)})$. The discussion in the
    previous paragraph implies that if Conjecture \ref{moore-splitting} is true
    for $n=1$, then the map $\mu:\Omega^2 S^3\langle 3\rangle\to
    \BGL_1(\S_{(p)})$ from Corollary \ref{hm-Z} is homotopic to the composite
    $$\Omega^2 S^3\langle 3\rangle \to \Omega^2 P^{2p+1}(p)\to
    \BGL_1(\S_{(p)}).$$
\end{remark}
\subsection{Fibers of charming maps}\label{fib-thom}

We shall need the following proposition.
\begin{prop}\label{homology-anick}
    Let $f:\Omega^2 S^{2p^n+1}\to S^{2p^n-1}$ be a charming map. Then there are 
    isomorphisms of coalgebras:
    $$\H_\ast(\fib(f);\FF_p) \cong \begin{cases}
	\FF_p[x_{2^{n+1}-1}^2] \otimes \bigotimes_{k>1} \FF_p[x_{2^{n+k}-1}] &
	p=2\\
	\bigotimes_{k>0} \FF_p[y_{2(p^{n+k}-1)}]\otimes \bigotimes_{j>0}
	\Lambda_{\FF_p}[x_{2p^{n+j}-1}] & p>2.
    \end{cases}$$
\end{prop}
\begin{proof}
    This is an easy consequence of the Serre spectral sequence coupled with the
    well-known coalgebra isomorphisms
    $$\H_\ast(\Omega^2 S^{2n+1};\FF_p) \cong \begin{cases}
	\bigotimes_{k>0} \FF_p[x_{2^k n-1}] & p=2 \\
	\bigotimes_{k>0} \FF_p[y_{2(np^k-1)}]\otimes \bigotimes_{j\geq 0}
	\Lambda_{\FF_p}[x_{2np^j-1}] & p>2,
    \end{cases}$$
    where these classes are generated by the one in dimension $2n-1$ via the
    single Dyer-Lashof operation (coming already from the cup-1 operad; see
    Remark \ref{cup-1}).
\end{proof}

\begin{remark}
    The Anick spaces $T^{2np+1}(p)$ from Example \ref{cmn-map} sit in fiber
    sequences
    $$S^{2np-1} \to T^{2np+1}(p)\to \Omega S^{2np+1},$$
    and are homotopy commutative H-spaces. A Serre spectral sequence calculation
    gives an identification of coalgebras
    $$\H_\ast(T^{2np+1}(p);\FF_p)\cong \FF_p[a_{2np}]\otimes
    \Lambda_{\FF_p}[b_{2np-1}],$$
    with $\beta(a_{2np}) = b_{2np-1}$, where $\beta$ is the Bockstein
    homomorphism. An argument with the bar spectral sequence recovers the result
    of Proposition \ref{homology-anick} in this particular case.
\end{remark}
\begin{remark}\label{integral-homology}
    Suppose that $X$ is a space which sits in a fiber sequence
    $$S^{2np-1} \to X \to \Omega S^{2np+1}$$
    such that the boundary map
    $\Omega^2 S^{2np+1} \to S^{2np-1}$ has degree $p^j$ on the bottom cell of the source.
    The Serre spectral sequence then only has a differential on the $E_{2np-1}$-page, and:
    $$\H_i(BK;\Z) \cong \begin{cases}
	\Z & i = 0\\
	\Z/p^j k & \text{if }i = 2npk - 1\\
	0 & \text{else}.
    \end{cases}$$
\end{remark}

We conclude this section by investigating Thom spectra of bundles defined over
fibers of charming maps. Let $R$ be a $p$-local $\E{1}$-ring, and let $\mu:K\to
\B\GL_1(R)$ denote a map from the fiber $K$ of a charming map $f:\Omega^2
S^{2np+1} \to S^{2np-1}$.  There is a fiber sequence
$\Omega S^{2np-1} \to K\to \Omega^2 S^{2np+1}$
of loop spaces, so we obtain a map $\Omega S^{2np-1} \to \B\GL_1(R)$. Such a map
gives an element $\alpha\in \pi_{2np-3} R$ via the effect on the bottom cell
$S^{2np-2}$.

Theorem \ref{thom-univ} implies that the Thom spectrum of the map $\Omega
S^{2np-1} \to \B\GL_1(R)$ should be thought of as the $\E{1}$-quotient $R\mmod
\alpha$, although this may not make sense if $R$ is not at least $\E{2}$.
However, in many cases (such as the ones we are considering here), the Thom
$R$-module $R\mmod\alpha$ is in fact an $\E{1}$-ring such that the map $R\to
R\mmod{\alpha}$ is an $\E{1}$-map. By Proposition \ref{thom}, there is an
induced map $\phi:\Omega^2 S^{2np+1} \to \B\GL_1(R\mmod\alpha)$ whose Thom
spectrum is equivalent as an $\E{1}$-ring to $K^\mu$. We would like to determine
the element\footnote{Technically, this is bad terminology: there are multiple
possibilities for the map $\phi$, and each gives rise to a map $S^{2np-1}\to
\B\GL_1(R\mmod\alpha)$. The elements in $\pi_{2np-2}(R\mmod\alpha)$ determined in
this way need not agree, but they are the same modulo the indeterminacy of the
Toda bracket $\langle p, \alpha, 1_{R\mmod\alpha}\rangle$.} of $\pi_\ast
R\mmod\alpha$ detected by the restriction to the bottom cell $S^{2np-1}$ of the
source of $\phi$. First, we note:
\begin{lemma}
    The element $\alpha\in \pi_{2np-3} R$ is $p$-torsion.
\end{lemma}
\begin{proof}
    Since $f$ is a charming map, the composite
    $S^{2np-1} \to \Omega^2 S^{2np-1} \xrightarrow{f} S^{2np-1}$
    is the degree $p$ map. Therefore, the element $p\alpha\in \pi_{2np-3} R$ is
    detected by the composite
    $$S^{2np-2} \to \Omega S^{2np-1} \to \Omega^3 S^{2np-1} \xrightarrow{\Omega
    f} \Omega S^{2np-1} \to K \xrightarrow{\mu} \B\GL_1(R).$$
    But there is a fiber sequence
    $\Omega^2 S^{2np-1} \xrightarrow{f} S^{2np-1} \to BK$
    by the definition of a charming map, so the composite detecting $p\alpha$ is
    null, as desired.
\end{proof}
There is now a square
$$\xymatrix{
    S^{2np-2}/p \ar[r] \ar[d] & S^{2np-1} \ar[d]\\
    K \ar[r] \ar[d]_-\alpha & \Omega^2 S^{2np+1} \ar[d]\\
    \B\GL_1(R) \ar[r] & B\GL_1(R\mmod\alpha),
}$$
and the following result is a consequence of the lemma and the definition of
Toda brackets:
\begin{lemma}\label{important-bracket}
    The element in $\pi_{2np-2}(R\mmod \alpha)$ detected by the vertical map
    $S^{2np-1} \to \B\GL_1(R\mmod\alpha)$ lives in the Toda bracket $\langle p,
    \alpha, 1_{R\mmod\alpha}\rangle$.
\end{lemma}
The upshot of this discussion is the following:
\begin{prop}\label{toda-bracket}
    Let $R$ be a $p$-local $\E{1}$-ring, and let $\mu:K\to \B\GL_1(R)$ denote a
    map from the fiber $K$ of a charming map $f:\Omega^2 S^{2np+1} \to
    S^{2np-1}$, providing an element $\alpha\in \pi_{2np-3} R$. Assume that the
    Thom spectrum $R\mmod\alpha$ of the map $\Omega S^{2np-1}\to \B\GL_1(R)$ is
    an $\E{1}$-$R$-algebra. Then there
    is an element $v\in \langle p, \alpha, 1_{R\mmod\alpha}\rangle$ such
    that $K^\mu$ is equivalent to the Thom spectrum of the map $\Omega^2
    S^{2np+1}\xrightarrow{v} \B\GL_1(R\mmod\alpha)$.
\end{prop}
\begin{remark}\label{whitehead}
    Let $R$ be an $\E{1}$-ring, and let $\alpha\in \pi_d R$. Then $\alpha$
    defines a map $S^{d+1} \to \BGL_1(R)$, and it is natural to ask when
    $\alpha$ extends along $S^{d+1} \to \Omega S^{d+2}$, or at least along
    $S^{d+1} \to J_k(S^{d+1})$ for some $k$. This is automatic if $R$ is an
    $\E{2}$-ring, but not necessarily so if $R$ is only an $\E{1}$-ring. Recall
    that there is a cofiber sequence
    $$S^{(k+1)(d+1)-1} \to J_k(S^{d+1})\to J_{k+1}(S^{d+1}),$$ 
    where the first map is the $(k+1)$-fold iterated Whitehead product
    $[\iota_{d+1}, [\cdots, [\iota_{d+1}, \iota_{d+1}]], \cdots]$. In
    particular, the map $S^{d+1} \to \BGL_1(R)$ extends along the map $S^{d+1}
    \to J_k(S^{d+1})$ if and only if there are compatible nullhomotopies of the
    $n$-fold iterated Whitehead products $[\alpha, [\cdots, [\alpha, \alpha]],
    \cdots] \in \pi_\ast \BGL_1(R)$ for $n\leq k$. These amount to properties of
    Toda brackets in the homotopy of $R$. We note, for instance, that the
    Whitehead bracket $[\alpha, \alpha]\in \pi_{2d+1} \BGL_1(R) \cong \pi_{2d}
    R$ is the element $2\alpha^2$; therefore, the map $S^{d+1}\to \BGL_1(R)$
    extends to $J_2(S^{d+1})$ if and only if $2\alpha^2 = 0$.
\end{remark}
\begin{remark}
    Let $R$ be a $p$-local $\E{2}$-ring, and let $\alpha\in \pi_d(R)$ with $d$
    even. Then $\alpha$ defines an element $\alpha\in \pi_{d+2} \B^2\GL_1(R)$.
    The $p$-fold iterated Whitehead product $[\alpha, \cdots, \alpha]\in
    \pi_{p(d+2)-(p-1)} \B^2\GL_1(R) \cong \pi_{pd + (p-1)} R$ is given by $p!
    Q_1(\alpha)$ modulo decomposables.  This is in fact true more generally. Let
    $R$ be an $\E{n}$-ring, and suppose $\alpha\in \pi_d(R)$. Let $i<n$, so
    $\alpha$ defines an element $\alpha\in \pi_{d+i} \B^i\GL_1(R)$.  The
    $p$-fold iterated Whitehead product $[\alpha, \cdots, \alpha]\in
    \pi_{p(d+i)-(p-1)} \B^i\GL_1(R) \cong \pi_{pd + (i-1)(p-1)} R$ is given by
    $p! Q_{i-1}(\alpha)$ modulo decomposables.
    
    We will describe this in detail in forthcoming work: the basic idea is to reduce to the universal example of an $\E{n}$-ring, and relate Whitehead products on $\pi_\ast(S^n)$ to the $\E{d}$-Browder bracket on $\Omega^d S^n_+$ (where $d\geq n$). Recall the isomorphism $\pi_j S^n \cong \pi_{j-d} \Omega^d S^n$. If $\alpha\in \pi_i S^n$ and $\beta\in \pi_j S^n$, then we will show in future work that the stabilization of the Whitehead product $[\alpha, \beta]\in \pi_{i+j-1} S^n \cong \pi_{i+j-d} \Omega^d S^n$ is closely related to the $\E{d}$-Browder bracket $[\alpha, \beta]_{\E{d}}$.
\end{remark}

\newpage
\section{Chromatic Thom spectra}\label{proof-section}

\subsection{Statement of the theorem}

To state the main theorem of this section, we set some notation. Fix an integer
$n\geq 1$, and work in the $p$-complete stable category. For each Thom spectrum
$R$ of height $n-1$ in Table \ref{the-table}, let
$\sigma_{n-1}:S^{|\sigma_{n-1}|}\to \B\GL_1(R)$ denote a map detecting
$\sigma_{n-1}\in \pi_{|\sigma_{n-1}|}(R)$ (which exists by Theorem
\ref{sigman-exists}). Let $K_{n}$ denote the fiber of a $p$-local charming map
$\Omega^2 S^{2p^{n}+1} \to S^{2p^{n}-1}$ satisfying the hypotheses of Conjecture
\ref{moore-splitting}, and let $K_2$ (resp.  $K_3$) denote the fiber of an
integrally defined charming map $\Omega^2 S^9\to S^7$ (resp.  $\Omega^2
S^{17}\to S^{15}$) satisfying the hypotheses of Conjecture
\ref{moore-splitting}.

Then:
\begin{thm-main}
    Let $R$ be a height $n-1$ spectrum as in the second line of Table
    \ref{the-table}.
    Then Conjectures \ref{moore-splitting} and \ref{centrality-conj} imply that
    there is a map $K_{n}\to \B\GL_1(R)$ such that the mod $p$ homology of the
    Thom spectrum $K_{n}^\mu$ is isomorphic to the mod $p$ homology of the
    associated designer chromatic spectrum $\TT(R)$ as a Steenrod
    comodule.

    If $R$ is any base spectrum other than $B$, the Thom spectrum $K_{n}^\mu$
    is equivalent to $\TT(R)$ upon $p$-completion for every prime $p$. If
    Conjecture \ref{tmf-conj} is true, then the same is true for $B$: the Thom
    spectrum $K_{n}^\mu$ is equivalent to $\TT(B) = \tmf$ upon $2$-completion.
\end{thm-main}
We emphasize again that na\"ively making sense of Theorem \ref{main-thm} relies
on knowing that $T(n)$ admits the structure of an $\E{1}$-ring; we shall
interpret this phrase as in Warning \ref{tn-thom-def}.
\begin{remark}
    Theorem \ref{main-thm} is proved independently of the nilpotence theorem.
    (In fact, it is even independent of Quillen's identification of $\pi_\ast
    \MU$ with the Lazard ring, provided one regards the existence of designer
    chromatic spectra as being independent of Quillen's identification.) We
    shall elaborate on the connection between Theorem \ref{main-thm} and the
    nilpotence theorem in future work; a sketch is provided in Remark
    \ref{nilpotence-proof}.
\end{remark}
\begin{remark}\label{uncond}
    Theorem \ref{main-thm} is true unconditionally when $n=1$, since that case
    is simply Corollary \ref{hm-Z}.
\end{remark}
\begin{remark}
    Note that Table \ref{thom-definitions} implies that the homology of each of
    the Thom spectra in Table \ref{the-table} are given by the $Q_0$-Margolis
    homology of their associated designer chromatic spectra. In particular, the
    map $R\to \TT(R)$ is a rational equivalence.
\end{remark}

Before we proceed with the proof of Theorem \ref{main-thm}, we observe some
consequences.
\begin{cor-bpn}
    Conjecture \ref{moore-splitting} and Conjecture \ref{centrality-conj} imply
    Conjecture \ref{hopkins-conj}.
\end{cor-bpn}
\begin{proof}
    This follows from Theorem \ref{main-thm}, Proposition \ref{toda-bracket},
    and Proposition \ref{vn-toda}.
\end{proof}
\begin{remark}
    Corollary \ref{bpn} is true \emph{unconditionally} when $n=1$, since Theorem
    \ref{main-thm} is true unconditionally in that case by Remark \ref{uncond}.
    See also Remark \ref{anick-S3}.
\end{remark}
\begin{remark}
    We can attempt to apply Theorem \ref{main-thm} for $R = A$ in conjunction
    with Proposition \ref{toda-bracket}.  Theorem \ref{main-thm} states that
    Conjecture \ref{moore-splitting} and Conjecture \ref{centrality-conj} imply
    that there is a map $K_2\to \BGL_1(A)$ whose Thom spectrum is equivalent to
    $\bo$. There is a fiber sequence
    $$\Omega S^7\to K_2\to \Omega^2 S^9,$$
    so we obtain a map $\mu:\Omega S^7\to K_2\to \BGL_1(A)$. The proof of
    Theorem \ref{main-thm} shows that the bottom cell $S^6$ of the source
    detects $\sigma_1\in \pi_5(A)$. A slight variation of the argument used to
    establish Proposition \ref{toda-bracket} supplies a map $\Omega^2 S^9\to
    \B\Aut((\Omega S^7)^\mu)$ whose Thom spectrum is $\bo$. The spectrum
    $(\Omega S^7)^\mu$ has mod $2$ homology $\FF_2[\zeta_1^4, \zeta_2^2]$.
    However, unlike $A$, it does not naturally arise an $\E{1}$-Thom spectrum
    over the sphere spectrum; this makes it unamenable to study via techniques
    of unstable homotopy.

    More precisely, $(\Omega S^7)^\mu$ is not the Thom spectrum of an $\E{1}$-map
    $X\to \BGL_1(\S)$ from a loop space $X$ which sits in a fiber sequence
    $$\Omega S^5\to X\to \Omega S^7$$
    of loop spaces. Indeed, $\B X$ would be a $S^5$-bundle over $S^7$, which by
    \cite[Lemma 4]{mahowald-bo-bu} implies that $X$ is then equivalent as a loop
    space to $\Omega S^5\times \Omega S^7$. The resulting $\E{1}$-map $\Omega
    S^7\to \BGL_1(\S)$ is specified by an element of $\pi_5(\S) \cong 0$, so
    $(\Omega S^7)^\mu$ must then be equivalent as an $\E{1}$-ring to $A \wedge
    \Sigma^\infty_+ \Omega S^7$. In particular, $\sigma_1\in \pi_5(A)$ would map
    nontrivially to $(\Omega S^7)^\mu$, which is a contradiction.
\end{remark}
The proof of Theorem \ref{main-thm} will also show:
\begin{corollary}\label{tmf-univ}
    Let $R$ be a height $n-1$ spectrum as in the second line of Table
    \ref{the-table}, and assume Conjecture \ref{tmf-conj} if $R = B$. Let $M$ be
    an $\E{3}$-$R$-algebra.  Conjecture \ref{moore-splitting} and Conjecture
    \ref{centrality-conj} imply that if:
    \begin{enumerate}
	\item the composite $\fr{Z}_3(R)\to R \to M$ is an $\E{3}$-algebra map,
	\item the element $\sigma_{n-1}$ in $\pi_\ast M$ is nullhomotopic,
	\item and the bracket $\langle p, \sigma_{n-1}, 1_M\rangle$ contains
	    zero,
    \end{enumerate}
    then there is a unital map $\TT(R)\to M$.
\end{corollary}

\subsection{The proof of Theorem \ref{main-thm}}\label{the-proof}

This section is devoted to giving a proof of Theorem \ref{main-thm}, dependent
on Conjecture \ref{moore-splitting} and Conjecture \ref{centrality-conj}. The
proof of Theorem \ref{main-thm} will be broken down into multiple steps. The
result for $y(n)$ and $y_\Z(n)$ follow from the result for $T(n)$ by Proposition
\ref{tn-yn}, so we shall restrict ourselves to the cases of $R$ being $T(n)$,
$A$, and $B$.

Fix $n\geq 1$. If $R$ is $A$ or $B$, we will restrict to $p=2$, and let $K_2$
and $K_3$ denote the integrally defined spaces from Conjecture
\ref{moore-splitting}. By Remarks \ref{sigma1-integral} and
\ref{sigma2-integral}, the elements $\sigma_1\in \pi_5(A)$ and $\sigma_2\in
\pi_{13}(B)$ are defined integrally. We will write $\sigma_{n-1}$ to generically
denote this element, and will write it as living in degree $|\sigma_{n-1}|$.
We shall also write $R$ to denote $X(p^{n}-1)$ and \emph{not} $T(n)$; this will
be so that we can apply Conjecture \ref{moore-splitting}.
We apologize for the inconvenience, but hope that this is worth circumventing
the task of having to read through essentially the same proofs for these
slightly different cases.

\subsubsection*{Step 1}

We begin by constructing a map $\mu:K_n\to \B\GL_1(R)$ as required by the
statement of Theorem \ref{main-thm}; the construction in the case $n=1$ follows
Remark \ref{hm-Z-alt}. By Conjecture \ref{moore-splitting}, the space $K_n$
splits off $\Omega^2 P^{|\sigma_{n-1}|+4}(p)$ (if $R = T(n)$, then
$|\sigma_{n-1}|+4 = |v_n|+3$).  We are therefore reduced to constructing a map
$\Omega^2 P^{|\sigma_{n-1}|+4}(p)\to \B\GL_1(R)$.  Theorem \ref{sigman-exists}
shows that the element $\sigma_{n-1}\in \pi_\ast R$ is $p$-torsion, so the map
$S^{|\sigma_{n-1}|+1} \to \B\GL_1(R)$ detecting $\sigma_{n-1}$ extends to a map
\begin{equation}\label{map-from-moore}
    S^{|\sigma_{n-1}|+1}/p = P^{|\sigma_{n-1}|+2}(p)\to \B\GL_1(R).
\end{equation}
Since $\Omega^2 P^{|\sigma_{n-1}|+4}(p)\simeq \Omega^2 \Sigma^2
P^{|\sigma_{n-1}|+2}(p)$, we would obtain an extension $\wt{\mu}$ of this map
through $\Omega^2 P^{|\sigma_{n-1}|+4}(p)$ if $R$ admits an $\E{3}$-structure.

Unfortunately, this is not true; but this is where Conjecture
\ref{centrality-conj} comes in: it says that the element $\sigma_{n-1}\in
\pi_{|\sigma_{n-1}|} R$ lifts to the $\E{3}$-center $\fr{Z}_3(R)$, where it has
the same torsion order as in $R$. (Here, we are abusively writing
$\fr{Z}_3(T(n-1))$ to denote the $\E{3}$-center of $X(p^n-1)_{(p)}$.) The
lifting of $\sigma_{n-1}$ to $\pi_{|\sigma_{n-1}|} \fr{Z}_3(R)$ provided by
Conjecture \ref{centrality-conj} gives a factorization of the map from
\eqref{map-from-moore} as
$$S^{|\sigma_{n-1}|+1}/p = P^{|\sigma_{n-1}|+2}(p)\to \B\GL_1(\fr{Z}_3(R))\to
\B\GL_1(R).$$
Since $\fr{Z}_3(R)$ is an $\E{3}$-ring, $\B\GL_1(\fr{Z}_3(R))$ admits the
structure of an $\E{2}$-space. In particular, the map
$P^{|\sigma_{n-1}|+2}(p)\to \B\GL_1(\fr{Z}_3(R))$ factors through $\Omega^2
P^{|\sigma_{n-1}|+4}(p)$, as desired. We let $\wt{\mu}$ denote the resulting
composite
$$\wt{\mu}:\Omega^2 P^{|\sigma_{n-1}|+4}(p)\to \B\GL_1(\fr{Z}_3(R))\to
\B\GL_1(R).$$

\subsubsection*{Step 2}

Theorem \ref{main-thm} asserts that there is an identification between the Thom
spectrum of the induced map $\mu:K_n\to \B\GL_1(R)$ and the associated designer
chromatic spectrum $\TT(R)$ via Table \ref{the-table}. We shall identify the
Steenrod comodule structure on the mod $p$ homology of $K_n^\mu$, and show that
it agrees with the mod $p$ homology of $\TT(R)$.

\begin{table}[h!]
    \centering
    \begin{tabular}{c | @{}c@{}}
	Designer chromatic spectrum & Mod $p$ homology\\
	\hline
	$\BP{n-1}$ & \begin{tabular}{c|c}
	    $p=2$ & $\FF_2[\zeta_1^2, \cdots, \zeta_{n-1}^2, \zeta_n^2,
	    \zeta_{n+1}, \cdots]$\\
	    $p>2$ & $\FF_p[\zeta_1, \zeta_2, \cdots]\otimes
	    \Lambda_{\FF_p}(\tau_n, \tau_{n+1}, \cdots)$
	\end{tabular}\\
	\hline
	$k(n-1)$ & \begin{tabular}{c|c}
	    $p=2$ & $\FF_2[\zeta_1, \cdots, \zeta_{n-1}, \zeta_n^2, \zeta_{n+1},
	    \cdots]$\\
	    $p>2$ & $\FF_p[\zeta_1, \zeta_2, \cdots]\otimes
	    \Lambda_{\FF_p}(\tau_0, \cdots, \tau_{n-2}, \tau_n, \tau_{n+1},
	    \cdots)$
	\end{tabular}\\
	\hline
	$k_\Z(n-1)$ & \begin{tabular}{c|c}
	    $p=2$ & $\FF_2[\zeta_1^2, \zeta_2, \cdots, \zeta_{n-1}, \zeta_n^2,
	    \zeta_{n+1}, \cdots]$\\
	    $p>2$ & $\FF_p[\zeta_1, \zeta_2, \cdots]\otimes
	    \Lambda_{\FF_p}(\tau_1, \cdots, \tau_{n-2}, \tau_n, \tau_{n+1},
	    \cdots)$
	\end{tabular}\\
	\hline
	$\bo$ & \begin{tabular}{c|c}
	    $p=2$ & $\FF_2[\zeta_1^4, \zeta_2^2, \zeta_3, \cdots]$\\
	    $p>2$ & $\FF_p[x_4]/v_1 \otimes \FF_p[\zeta_1, \zeta_2,
	    \cdots]\otimes \Lambda_{\FF_p}(\tau_2, \tau_3, \cdots)$
	\end{tabular}\\
	\hline
	$\tmf$ & \begin{tabular}{c|c}
	    $p=2$ & $\FF_2[\zeta_1^8, \zeta_2^4, \zeta_3^2, \zeta_4, \cdots]$\\
	    $p=3$ & $\Lambda_{\FF_3}(b_4) \otimes \FF_3[\zeta_1^3, \zeta_2,
	    \cdots] \otimes \Lambda_{\FF_3}(\tau_3, \tau_4, \cdots)$\\
	    $p\geq 5$ & $\FF_p[c_4, c_6]/(v_1, v_2) \otimes \FF_p[\zeta_1,
	    \zeta_2, \cdots]\otimes \Lambda_{\FF_p}(\tau_3, \tau_4, \cdots)$
	\end{tabular}\\
	\hline
    \end{tabular}
    \vspace{0.5cm}
    \caption{The mod $p$ homology of designer chromatic spectra. 
    See \cite[Theorem 4.3]{lawson-naumann}, as well as \cite[Proposition 1.7]{wilson-omega-spectrum} and \cite[Proposition 5.3]{angeltveit-rognes} for a proof of the statement for $\H_\ast(\BP{n-1}; \FF_p)$; this implies the calculations of $\H_\ast(k(n-1); \FF_p)$ and $\H_\ast(k_\Z(n-1); \FF_p)$. See \cite[Proposition 6.1]{angeltveit-rognes} for a proof of the statements for $\H_\ast(\bo; \FF_2)$ and $\H_\ast(\tmf; \FF_2)$, and \cite[Theorem 21.5]{rezk-512} for $\H_\ast(\tmf; \FF_p)$ for any $p$. For odd $p$, $\bo_{(p)}$ is a sum of shifts of $\BP{1}$, which implies the statement about $\H_\ast(\bo; \FF_p)$.}
    \label{table-homology}
\end{table}

In Table \ref{table-homology}, we have recorded the mod $p$ homology of the
designer chromatic spectra in Table \ref{the-table} (see \cite[Theorem
4.3]{lawson-naumann} for $\BP{n-1}$).
It follows from Proposition \ref{homology-anick} that there is an isomorphism
$$\H_\ast(K_n^\mu) \cong
\begin{cases}
    \H_\ast(R) \otimes \FF_2[x_{2^{n+1}-1}^2] \otimes \bigotimes_{k>1}
    \FF_2[x_{2^{n+k}-1}] & p=2\\
    \H_\ast(R)\otimes \bigotimes_{k>0} \FF_p[y_{2(p^{n+k}-1)}] \otimes
    \bigotimes_{j>0} \Lambda_{\FF_p}[x_{2p^{n+j}-1}] & p>2.
\end{cases}$$
Combining this isomorphism with Theorem \ref{tn-def}, Proposition \ref{yn-hfp},
Proposition \ref{A-hfp}, and Proposition \ref{B-hfp}, we find that there is an
\emph{abstract} equivalence between the mod $p$ homology of $K_n^\mu$ and
the mod $p$ homology of $\TT(R)$.

We shall now work at $p=2$ for the remainder of the proof; the same argument
goes through with slight modifications at odd primes. We now identify the
Steenrod comodule structure on $\H_\ast(K_n^\mu)$. Recall that $\wt{\mu}$ is the
map $\Omega^2 P^{|\sigma_{n-1}|+4}(p)\to \B\GL_1(R)$ from Step 1. By
construction, there is a map $K_n^\mu\to \Omega^2
P^{|\sigma_{n-1}|+4}(p)^{\wt{\mu}}$. The map $\Phi$ factors through a map
$\wt{\Phi}: \Omega^2 P^{|\sigma_{n-1}|+4}(p)^{\wt{\mu}} \to \TT(R)$. The Thom
spectrum $\Omega^2 P^{|\sigma_{n-1}|+4}(p)^{\wt{\mu}}$ admits the structure of a
$\cQ_1$-ring. Indeed, it is the smash product $\Omega^2
P^{|\sigma_{n-1}|+4}(p)^{\phi} \wedge_{\fr{Z}_3(R)} R$, where $\phi: \Omega^2
P^{|\sigma_{n-1}|+4}(p)\to \B\GL_1(\fr{Z}_3(R))$; it therefore suffices to
observe that the Thom spectrum $\Omega^2 P^{|\sigma_{n-1}|+4}(p)^{\phi}$ admits
the structure of an $\E{1} \otimes \cQ_1$-ring. (Here, $\E{1}\otimes \cQ_1$ denotes the Boardman-Vogt tensor product of the $\E{1}$- and $\cQ_1$-operads.) Since there is a map $\cQ_1 \to \E{2}$ of $\infty$-operads, this is a consequence of the
fact that $\phi$ is a double loop map, and hence an $\E{1}\otimes \cQ_1$-algebra
map.  Moreover, the image of $\H_\ast(K_n^\mu)$ in $\H_\ast(\Omega^2
P^{|\sigma_{n-1}|+4}(p)^{\wt{\mu}})$ is generated under the single Dyer-Lashof
operation (arising from the cup-1 operad; see Remark \ref{cup-1}) by the
indecomposables in the image of the map $\H_\ast(R)\to \H_\ast(\Omega^2
P^{|\sigma_{n-1}|+4}(p)^{\wt{\mu}})$.

The Postnikov truncation map $\Omega^2 P^{|\sigma_{n-1}|+4}(p)^{\wt{\mu}} \to
\H\pi_0 \left(\Omega^2 P^{|\sigma_{n-1}|+4}(p)^{\wt{\mu}}\right)$ is one of
$\cQ_1$-rings. Since $\Omega^2 P^{|\sigma_{n-1}|+4}(p)$ is highly connected,
$\pi_0 \left(\Omega^2 P^{|\sigma_{n-1}|+4}(p)^{\wt{\mu}}\right) \cong \pi_0(R)$.
In particular, there is an $\Eoo$-map $\H\pi_0 \left(\Omega^2
P^{|\sigma_{n-1}|+4}(p)^{\wt{\mu}}\right)\to \H\FF_p$. The composite
$$\Omega^2 P^{|\sigma_{n-1}|+4}(p)^{\wt{\mu}} \to \H\pi_0 \left(\Omega^2
P^{|\sigma_{n-1}|+4}(p)^{\wt{\mu}}\right)\to \H\FF_p$$
is therefore a $\cQ_1$-algebra map. Moreover, the composite
$$R\to \Omega^2 P^{|\sigma_{n-1}|+4}(p)^{\wt{\mu}} \to \H\pi_0 \left(\Omega^2
P^{|\sigma_{n-1}|+4}(p)^{\wt{\mu}}\right)\to \H\FF_p$$
is simply the Postnikov truncation for $R$.
It follows that the indecomposables in $\H_\ast(\Omega^2
P^{|\sigma_{n-1}|+4}(p)^{\wt{\mu}})$ which come from the indecomposables in
$\H_\ast(R)$ are sent to the indecomposables in $\H_\ast(\H\FF_p)$. Using the
discussion in the previous paragraph, Steinberger's calculation (Theorem
\ref{steinberger}), and the Dyer-Lashof hopping argument of Remark
\ref{hopping}, we may conclude that the Steenrod comodule structure on
$\H_\ast(K_n^\mu)$ (which, recall, is abstractly isomorphic to
$\H_\ast(\TT(R))$) agrees with the Steenrod comodule structure on
$\H_\ast(\TT(R))$.

\subsubsection*{Step 3}

By Step 2, the mod $p$ homology of the Thom spectrum $K_{n}^\mu$ is isomorphic
to the mod $p$ homology of the associated designer chromatic spectrum $\TT(R)$
as a Steenrod comodule. The main result of \cite{angeltveit-lind} and
\cite[Theorem 1.1]{adams-priddy} now imply that unless $R = B$, the Thom
spectrum $K_{n}^\mu$ is equivalent to $\TT(R)$ upon $p$-completion for every
prime $p$. Finally, if Conjecture \ref{tmf-conj} is true, then the same
conclusion can be drawn for $B$: the Thom spectrum $K_{n}^\mu$ is equivalent to
$\TT(B) = \tmf$ upon $p$-completion for every prime $p$.

This concludes the proof of Theorem \ref{main-thm}.

\subsection{Remark on the proof}

Before proceeding, we note the following consequence of the proof of Theorem
\ref{main-thm}.
\begin{prop}\label{reduction}
    Let $p$ be an odd prime. Assume Conjecture \ref{moore-splitting} and
    Conjecture \ref{centrality-conj}. Then the composite
    $$g_2:\Omega^2 S^{|\sigma_{n-1}|+3} \to \Omega^2 P^{|\sigma_{n-1}|+4}(p)
    \xrightarrow{\wt{\mu}} \B\GL_1(X(p^n-1))\to \B\GL_1(\BP{n-1})$$
    is null.
\end{prop}
\begin{proof}
    Let $R = X(p^n-1)$, and let $\TT(R) = \BP{n-1}$.
    The map $g_2$ is the composite with $\B\GL_1(\fr{Z}_3(R)) \to B\GL_1(\TT(R))$
    with the extension of the map
    $$\sigma_{n-1}:S^{|\sigma_{n-1}|+1}\to \B\GL_1(\fr{Z}_3(R))$$
    along the double suspension $S^{|\sigma_{n-1}|+1}\to \Omega^2
    S^{|\sigma_{n-1}|+3}$. Since $\sigma_{n-1}$ is null in $\pi_\ast \TT(R)$, we
    would be done if $g_2$ was homotopic to the dotted extension
    \begin{equation*}
	\xymatrix{
	    S^{|\sigma_{n-1}|+1}\ar[r]^-{\sigma_{n-1}} \ar[d] &
	    \B\GL_1(\TT(R))\\
	    \Omega^2 S^{|\sigma_{n-1}|+3} \ar@{-->}[ur]_-{g_2'}
	    }
    \end{equation*}
    The potential failure of these maps to be homotopic stems from the fact that
    the composite $\fr{Z}_3(R)\to R \to \TT(R)$ need not be a map of
    $\E{3}$-rings. It is, however, a map of $\E{2}$-rings; therefore, the maps
    $$g_1:\Omega S^{|\sigma_{n-1}|+2} \to \B\GL_1(\fr{Z}_3(R)) \to
    \B\GL_1(\TT(R))$$
    and
    $$g_1':\Omega S^{|\sigma_{n-1}|+2} \to \B\GL_1(\TT(R))$$
    obtained by extending along the suspension $S^{|\sigma_{n-1}|+1}\to \Omega
    S^{|\sigma_{n-1}|+2}$ \emph{are} homotopic. We now utilize the following
    result of Serre's:
    \begin{prop}[{Serre, \cite[p. 281]{serre-loop-splitting}}]
	Let $p$ be an odd prime. Then the suspension $S^{2n-1} \to \Omega
	S^{2n}$ splits upon $p$-localization: there is a $p$-local equivalence
	$$E \times \Omega [\iota_{2n}, \iota_{2n}]: S^{2n-1} \times \Omega
	S^{4n-1} \to \Omega S^{2n}.$$
    \end{prop}
    This implies that the suspension map $\Omega S^{|\sigma_{n-1}| + 2} \to
    \Omega^2 S^{|\sigma_{n-1}|+3}$ admits a splitting as loop spaces. In
    particular, this implies that the map $g_2$ is homotopic to the composite
    $$\Omega^2 S^{|\sigma_{n-1}|+3} \to \Omega S^{|\sigma_{n-1}| + 2}
    \xrightarrow{g_1} \B\GL_1(\fr{Z}_3(R))\to \BGL_1(\TT(R)),$$
    and similarly for $g_2'$. Since $g_1$ and $g_1'$ are homotopic, and $g_1'$
    (and hence $g_2'$) is null, we find that $g_2$ is also null, as desired.
\end{proof}

\subsection{Infinite families and the nilpotence theorem}\label{infinite}

We now briefly discuss the relationship between Theorem \ref{main-thm} and the
nilpotence theorem. We begin by describing a special case of this connection.
Recall from Remark \ref{hm-nilp} that Theorem \ref{hm} implies that if $R$ is an
$\E{2}$-ring spectrum, and $x\in \pi_\ast R$ is a simple $p$-torsion element
which has trivial $\MU$-Hurewicz image, then $x$ is nilpotent. A similar
argument implies the following.
\begin{prop}\label{thom-Z-nilp}
    Assume Conjecture \ref{moore-splitting} when $n=1$. Then
    Corollary \ref{hm-Z} (i.e., Theorem \ref{main-thm} when $n=1$) implies that
    if $R$ is a $p$-local $\E{3}$-ring spectrum, and $x\in \pi_\ast R$ is a
    class with trivial $\H\Z_p$-Hurewicz image such that:
    \begin{itemize}
	\item $\alpha_1 x = 0$ in $\pi_\ast R$; and
	\item the Toda bracket $\langle p, \alpha_1, x\rangle$ contains zero;
    \end{itemize}
    then $x$ is nilpotent.
\end{prop}
\begin{proof}
    We claim that the composite
    \begin{equation}\label{baby}
	\Omega^2 S^3\langle 3\rangle\to \BGL_1(\S_{(p)})\to \BGL_1(R[1/x])
    \end{equation}
    is null. Remark \ref{hm-Z-alt} implies that Conjecture \ref{moore-splitting}
    for $n=1$ reduces us to showing that the composite
    $$\Omega^2 P^{2p+1}(p)\xrightarrow{\alpha_1} \BGL_1(\S_{(p)})\to
    \BGL_1(R[1/x])$$
    is null. Since this composite is one of double loop spaces, it further
    suffices to show that the composite
    \begin{equation}\label{wts}
	P^{2p-1}(p)\to \BGL_1(\S_{(p)})\to \BGL_1(R[1/x])
    \end{equation}
    is null. The bottom cell $S^{2p-2}$ of $P^{2p-1}(p)$ maps trivially to
    $\BGL_1(R[1/x])$, because the bottom cell detects $\alpha_1$ (by Remark
    \ref{hm-Z-alt}), and $\alpha_1$ is nullhomotopic in $R[1/x]$. Therefore, the
    map \eqref{wts} factors through the top cell $S^{2p-1}$ of $P^{2p-1}(p)$.
    The resulting map
    $$S^{2p-1}\to \BGL_1(\S_{(p)})\to \BGL_1(R[1/x])$$
    detects an element of the Toda bracket $\langle p, \alpha_1, x\rangle$, but
    this contains zero by hypothesis, so is nullhomotopic.

    Since the map \eqref{baby} is null, Corollary \ref{hm-Z} and Theorem
    \ref{thom-univ} implies that there is a ring map $\H\Z_p\to R[1/x]$. In
    particular, the composite of the map $x:\Sigma^{|x|} R\to R$ with the unit
    $R\to R[1/x]$ factors as shown:
    $$\xymatrix{
	\Sigma^{|x|} R \ar[r]^-x \ar[d] & R \ar[r] \ar[d] & R[1/x].\\
	\H\Z_p \wedge \Sigma^{|x|} R \ar[r]^-x & \H\Z_p \wedge R \ar@{-->}[ur]
    }$$
    The bottom map, however, is null, because $x$ has zero $\H\Z_p$-Hurewicz
    image. Therefore, the element $x\in \pi_\ast R[1/x]$ is null, and hence
    $R[1/x]$ is contractible.
\end{proof}
\begin{remark}
    One can prove by a different argument that Proposition \ref{thom-Z-nilp} is
    true without the assumption that Conjecture \ref{moore-splitting} holds when
    $n=1$. At $p=2$, this was shown by Astey in \cite[Theorem 1.1]{astey}.
\end{remark}

To discuss the relationship between Theorem \ref{main-thm} for general $n$ and
the nilpotence theorem (which we will expand upon in future work), we embark on
a slight digression. The following proposition describes the construction of
some infinite families.
\begin{prop}\label{inf-family}
    Let $R$ be a height $n-1$ spectrum as in the second line of Table
    \ref{thom-definitions}, and assume Conjecture \ref{centrality-conj} if $R =
    A$ or $B$. Then there is an infinite family $\sigma_{n-1,p^k}\in
    \pi_{p^k|v_n|-1}(R)$.
    Conjecture \ref{centrality-conj} implies that $\sigma_{n-1,p^k}$ lifts to
    $\pi_{p^k|v_n|-1}(\fr{Z}_3(R))$, where $\fr{Z}_3(R)$ abusively denotes the
    $\E{3}$-center of $X(p^n-1)$ if $R = T(n-1)$.
\end{prop}
\begin{proof}
    We construct this family by induction on $k$. The element $\sigma_{n-1, 1}$
    is just $\sigma_{n-1}$, so assume that we have defined $\sigma_{n-1, p^k}$.
    The element $\sigma_{n-1,p^k}\in \pi_{p^k|v_n|-1} R$ defines a map
    $\sigma_{n-1, p^k}:S^{p^k|v_n|} \to \BGL_1(R)$. When $R = T(n-1)$, Lemma
    \ref{tn-thom} (and the inductive hypothesis) implies that the map defined by
    $\sigma_n$ factors through the map $\BGL_1(X(p^n-1))\to \BGL_1(T(n-1))$.
    When $R = A$ or $B$, Conjecture \ref{centrality-conj} (and the inductive
    hypothesis) implies that the map defined by $\sigma_n$ factors through the
    map $\BGL_1(\fr{Z}_3(R))\to \BGL_1(R)$. This implies that for all $R$ as in
    the second line of Table \ref{thom-definitions}, the map $\sigma_{n-1,
    p^k}:S^{p^k|v_n|} \to \BGL_1(R)$ factors through an $\E{1}$-space, which we
    shall just denote by $\dZ_R$ for the purpose of this proof. If we assume
    Conjecture \ref{centrality-conj}, then we may take $\dZ_R =
    \BGL_1(\fr{Z}_3(R))$.
    
    Therefore, we get a map
    $\sigma_{n-1,p^k}:\Omega S^{p^k|v_n|+1} \to \BGL_1(R)$ via the composite
    $$\Omega S^{p^k|v_n|+1}\to \dZ_R \to \BGL_1(R).$$
    Since $\dZ_R$ is an $\E{1}$-space, the map $\Omega S^{p^k|v_n|+1}\to \dZ_R$
    is adjoint to a map
    $$\bigvee_{j\geq 1} S^{jp^k|v_n|+1} \simeq \Sigma \Omega S^{p^k|v_n|+1}\to
    \B\dZ_R;$$
    the source splits as indicated via the James splitting. These splittings are
    given by Whitehead products; in particular, the map $S^{p^{k+1}|v_n|+1} =
    S^{p(p^k|v_n|+1)-(p-1)} \to \B\dZ_R$ is given by the $p$-fold Whitehead
    product $[\sigma_{n-1, p^k}, \cdots, \sigma_{n-1, p^k}]$. This is divisible
    by $p$,
    so it yields a map $S^{p^{k+1}|v_n|} \to \dZ_R$, and hence a map
    $S^{p^{k+1}|v_n|}\to \BGL_1(R)$ given by composing with the map $\dZ_R\to
    \BGL_1(R)$. This defines the desired element $\sigma_{n-1,p^{k+1}}\in
    \pi_{p^{k+1} |v_n| - 1}(R)$. As the construction makes clear, assuming
    Conjecture \ref{centrality-conj} and taking $\dZ_R = \BGL_1(\fr{Z}_3(R))$
    implies that $\sigma_{n-1,p^k}$ lifts to $\pi_{p^k|v_n|-1}(\fr{Z}_3(R))$.
\end{proof}
\begin{remark}
    This infinite family is detected in the $1$-line of the ANSS for $R$ by
    $\delta(v_n^k)$, where $\delta$ is the boundary map induced by the map
    $\Sigma^{-1} R/p\to R$. This is a consequence of the geometric boundary theorem (see \cite[Theorem 2.3.4]{green}) applied to the cofiber sequence $R \xar{p} R \to R/p$.
\end{remark}
\begin{remark}
    The element $\sigma_{n-1,1}\in \pi_{2p^n-3}(R)$ is precisely $\sigma_{n-1}$.
\end{remark}
\begin{remark}
    When $n=1$, the ring $R$ is the ($p$-local) sphere spectrum. The infinite
    family $\sigma_{n-1, p^k}$ is the Adams-Toda $\alpha$-family: namely, $\alpha_{p^k}\in \pi_{2p^k(p-1)-1}(\S)$ maps to $\sigma_{0,p^k}\in \pi_{2p^k(p-1)-1} X(p-1)$ under the unit map $\S \to X(p-1)$.
\end{remark}
We now briefly sketch an argument relating Theorem \ref{main-thm} to the proof
of the nilpotence theorem; we shall elaborate on this discussion in forthcoming
work.  
\begin{remark}\label{nilpotence-proof}
    The heart of the nilpotence theorem is what is called Step III in
    \cite{dhs-i}; this step amounts to showing that certain self-maps of
    $T(n-1)$-module skeleta (denoted $G_k$ in \cite{dhs-i}) of $T(n)$ are
    nilpotent. Let us assume that $p>2$ for simplicity. Then these self-maps are given by multiplication by the $p$-fold Toda bracket $b_{n,k} = \langle \sigma_{n-1,
    p^k}, \cdots, \sigma_{n-1, p^k}\rangle$ at an odd prime $p$; this lives in degree $p|\sigma_{n-1,p^k}|+p-2 = 2p^k(p^n-1)-2$. (When $p=2$,
    the desired element $\sigma_{n-1, p^k}$ is denoted by $h$ in \cite[Theorem
    3]{hopkins-global}.) It therefore suffices to establish the nilpotency of
    the $b_{n,k}$.

    This can be proven through Theorem \ref{main-thm} via induction on $k$; we
    shall assume Conjecture \ref{moore-splitting} and Conjecture
    \ref{centrality-conj} for the remainder of this discussion. The motivation
    for this approach stems from the observation that if $R$ is any
    $\E{3}$-$\FF_2$-algebra and $x\in \pi_\ast(R)$, then there is a relation
    $Q_1(x)^2 = Q_2(x^2)$ (at odd primes, one has a relation involving the
    $p$-fold Toda bracket $\langle Q_1(x), \cdots, Q_1(x)\rangle$).
    In our setting, Proposition \ref{inf-family} implies that the elements
    $\sigma_{n-1,k}$ lift to $\pi_\ast \fr{Z}_3(X(p^n-1))$.  At $p=2$, one can
    prove (in the same way that the Cartan relation $Q_1(x)^2 = Q_2(x^2)$ is
    proven) that the construction of this infinite family implies that
    $\sigma_{n-1,p^{k+1}}^2$ can be described in terms of
    $Q_2(\sigma_{n-1,p^k}^2)$.
    At odd primes, there is a similar relation involving the $p$-fold Toda
    bracket defining $b_{n,k}$.
    In particular, induction on $k$ implies that the $b_{n,k}$ are all nilpotent
    in $\pi_\ast \fr{Z}_3(X(p^n-1))$ if $b_{n,1}$ is nilpotent. Note that $|b_{n,1}| = 2p^{n+1}-2p-2$.

    To argue that $b_{n,1}$ is nilpotent, one first observes that $\sigma_{n-1}
    b_{n,1}^p = 0$ in $\pi_\ast \fr{Z}_3(X(p^n-1))$; when $n=0$, this follows
    from the statement that $\alpha_1 \beta_1^p = 0$ in the sphere. To show that
    $b_{n,1}$ is nilpotent, it suffices to establish that
    $\fr{Z}_3(X(p^n-1))[1/b_{n,1}^p]$ is contractible; when $n=1$, this follows
    from Proposition \ref{thom-Z-nilp}. We give a very brief sketch of this
    nilpotence for general $n$, by arguing as in Proposition \ref{thom-Z-nilp},
    and with a generous lack of precision which will be remedied in forthcoming
    work. 
    
    For notational convenience, we now denote $d_{n,1} = b_{n,1}^p$, so that $|d_{n,1}| = 2p^{n+2}-2p^2-2p$. It suffices to show that the multiplication-by-$d_{n,1}$ map
    $$d_{n,1}: \Sigma^{|d_{n,1}|}\fr{Z}_3(X(p^n-1))\to
    \fr{Z}_3(X(p^n-1))[1/d_{n,1}]$$
    is nullhomotopic. Since $\sigma_{n-1}$ kills $d_{n,1}$, we know that
    $\sigma_{n-1}$ is nullhomotopic in $\fr{Z}_3(X(p^n-1))[1/d_{n,1}]$.
    Moreover, the bracket $\langle p, \sigma_{n-1},
    1_{\fr{Z}_3(X(p^n-1))[1/d_{n,1}]}\rangle$ contains zero. By arguing as in
    Proposition \ref{thom-Z-nilp}, we can conclude that the composite
    $$K_n\to \Omega^2 P^{|\sigma_{n-1}|+4} \xrightarrow{\phi}
    \BGL_1(\fr{Z}_3(X(p^n-1))) \to \BGL_1(\fr{Z}_3(X(p^n-1))[1/d_{n_0}])$$
    is nullhomotopic, where the map $\phi$ is constructed in Step 1 of the proof
    of Theorem \ref{main-thm}. (Recall that the proof of Theorem \ref{main-thm}
    shows that the Thom spectrum $(\Omega^2 P^{|\sigma_{n-1}|+4})^\phi$ is an
    $\E{1}\otimes \cQ_1$-$\fr{Z}_3(X(p^n-1))$-algebra such that $\BP{n-1}$
    splits off its base change along the map $\fr{Z}_3(X(p^n-1))\to T(n-1)$.)
    It follows from Theorem \ref{thom-univ} that the multiplication-by-$d_{n,1}$
    map factors as
    $$\xymatrix{
	\Sigma^{|d_{n,1}|} \fr{Z}_3(X(p^n-1)) \ar[r]^-{d_{n,1}} \ar[d] &
	\fr{Z}_3(X(p^n-1)) \ar[r] \ar[d] & \fr{Z}_3(X(p^n-1))[1/d_{n,1}].\\
	\Sigma^{|d_{n,1}|} (\Omega^2 P^{|\sigma_{n-1}|+4})^\phi
	\ar[r]^-{d_{n,1}} & (\Omega^2 P^{|\sigma_{n-1}|+4})^\phi \ar@{-->}[ur]
    }$$
    To show that the top composite is null, it therefore suffices to show that
    the self map of $K_n^\phi$ defined by $d_{n,1}$ is nullhomotopic. This
    essentially follows from the fact that $(\Omega^2
    P^{|\sigma_{n-1}|+4})^\phi$ is an $\E{1}\otimes
    \cQ_1$-$\fr{Z}_3(X(p^n-1))$-algebra: multiplication by $d_{n,1}$ is
    therefore null on $K_n^\phi$, because $d_{n,1}$ is built from $\sigma_{n-1}$
    (which is null in $(\Omega^2 P^{|\sigma_{n-1}|+4})^\phi$) via $\E{1}$-power
    operations.
\end{remark}

\newpage
\section{Applications}\label{applications}

\subsection{Splittings of cobordism spectra}\label{cob-splitting}

The goal of this section is to prove the following.
\begin{thm-string}
    Assume that the composite $\fr{Z}_3(B)\to B\to \MString_{(2)}$ is an $\E{3}$-map.
    Then
    Conjectures \ref{moore-splitting}, \ref{centrality-conj}, and \ref{tmf-conj}
    imply that there is a unital splitting of the Ando-Hopkins-Rezk orientation
    $\MString_{(2)}\to \tmf_{(2)}$.
\end{thm-string}
\begin{remark}
    We believe that the assumption that the composite $\fr{Z}_3(B)\to B\to
    \MString_{(2)}$ is an $\E{3}$-map is too strong: we believe that it can be
    removed using special properties of fibers of charming maps, and we will
    return to this in future work.
\end{remark}
We only construct unstructured splittings; it seems unlikely that they can be
refined to structured splittings.
A slight modification of our arguments should work at any prime.
\begin{remark}
    In fact, the same argument used to prove Theorem \ref{mstring} shows that if
    the composite $\fr{Z}_3(A)\to A\to \MSpin_{(2)}$ is an $\E{3}$-map, then
    Conjecture \ref{moore-splitting} and Conjecture \ref{centrality-conj} imply
    that there are unital splittings of the Atiyah-Bott-Shapiro orientation
    $\MSpin_{(2)} \to \bo_{(2)}$.
    This splitting was originally proved unconditionally (i.e., without assuming
    Conjecture \ref{moore-splitting} or Conjecture \ref{centrality-conj}) by
    Anderson-Brown-Peterson in \cite{abp} via a calculation with the Adams
    spectral sequence.
\end{remark}
\begin{remark}
    The inclusion of the cusp on $\ol{\Mell}$ defines an $\Eoo$-map $c:\tmf\to
    \bo$ as in \cite[Theorem 1.2]{lawson-naumann}. The resulting diagram
    $$\xymatrix{
	\MString_{(2)} \ar[r] \ar[d] & \MSpin_{(2)} \ar[d]\\
	\tmf_{(2)} \ar[r]^-c & \bo_{(2)}
    }$$
    commutes (see, e.g., \cite[Lemma 6.4]{tmf-witten}). The splitting $s:\tmf_{(2)}
    \to \MString_{(2)}$ of Theorem \ref{mstring} defines a composite
    $$\tmf_{(2)} \xrightarrow{s} \MString_{(2)} \to \MSpin_{(2)} \to \bo_{(2)}$$
    which agrees with $c$.
\end{remark}
\begin{remark}\label{family-surj}
    The Anderson-Brown-Peterson splitting implies that if $X$ is any compact
    space, then the Atiyah-Bott-Shapiro $\hat{A}$-genus (i.e., the index of the
    Dirac operator in families) $\MSpin^\ast(X) \to \bo^\ast(X)$ is surjective.
    Similarly, if the composite $\fr{Z}_3(B)\to B\to \MString$ is an
    $\E{3}$-map, then Conjectures \ref{moore-splitting}, \ref{centrality-conj},
    and \ref{tmf-conj} imply that the Ando-Hopkins-Rezk orientation (i.e., the
    Witten genus in families) $\MString^\ast(X) \to \tmf^\ast(X)$ is also
    surjective.
\end{remark}
\begin{remark}
    In \cite{tmf-witten}, we proved (unconditionally) that the map $\pi_\ast
    \MString\to \pi_\ast \tmf$ is surjective. Our proof proceeds by showing that
    the map $\pi_\ast B\to \pi_\ast \tmf$ is surjective via arguments with the
    Adams-Novikov spectral sequence and by exploiting the $\E{1}$-ring structure
    on $B$ to lift the powers of $\Delta$ living in $\pi_\ast \tmf$.
\end{remark}
The discussion preceding \cite[Remark 7.3]{mahowald-rezk-tmf} (in the arXiv version of the document) implies that for a particular model of $\tmf_0(3)$, we have:
\begin{corollary}
    Assume that the composite $\fr{Z}_3(B)\to B_{(2)}\to \MString_{(2)}$ is an $\E{3}$-map.
    Then Conjectures \ref{moore-splitting}, \ref{centrality-conj}, and
    \ref{tmf-conj} imply that $\Sigma^{16} \tmf_0(3)^\wedge_2$ is a summand of
    $\MString^\wedge_2$.
\end{corollary}

We now turn to the proof of Theorem \ref{mstring}. 
\begin{proof}[Proof of Theorem \ref{mstring}]
    First, note that such a splitting exists after rationalization. Indeed, it
    suffices to check that this is true on rational homotopy; since the
    orientations under considerations are $\Eoo$-ring maps, the induced map on
    homotopy is one of rings. It therefore suffices to lift the generators.

    We now show that the generators of $\pi_\ast \tmf\otimes \QQ \cong \QQ[c_4,
    c_6]$ lift to $\pi_\ast \MString\otimes \QQ$. Although one can argue this by
    explicitly constructing manifold representatives (as is done for $c_4$ in
    \cite[Corollary 6.3]{tmf-witten}), it is also possible to provide a more
    homotopy-theoretic proof: the elements $c_4$ and $c_6$ live in dimensions
    $8$ and $12$ respectively, and the map $\MString \to \tmf$ is known to be an
    equivalence in dimensions $\leq 15$ by \cite[Theorem 2.1]{hill-string}. It follows that the
    same is true rationally, so $c_4$ and $c_6$ indeed lift to $\pi_\ast
    \MString\otimes \QQ$, as desired.

    We will now construct a splitting after $p$-completion where $p=2$.
    By Corollary \ref{tmf-univ}, we obtain a unital map $\tmf \simeq \TT(B)\to
    \MString$ upon $p$-completion which splits the orientation
    $\MString\to \TT(B)$ because:
    \begin{enumerate}
	\item the map $\fr{Z}_3(B) \to B\to \MString$ is an $\E{3}$-ring map (by
	    assumption).
	\item the element $\sigma_{2}$ vanishes in $\pi_{13} \MString_{(2)}$
	    (because $\pi_{13} \MString_{(2)}\cong \pi_{13} \tmf_{(2)} \cong
	    0$),
	\item and the Toda bracket $\langle 2, \sigma_{2},
	    1_{\MString_{(2)}}\rangle\subseteq \pi_\ast \MString_{(2)}$ contains
	    zero because $\pi_{14} \MString_{(2)} \cong \pi_{14} \tmf_{(2)}$,
	    and the corresponding bracket $\langle 2, \sigma_{2},
	    1_{\MString_{(2)}}\rangle\subseteq \pi_{14} \tmf_{(2)}$ detects $v_3$, hence contains zero.
    \end{enumerate}
    To obtain a map $\tmf_{(p)} \to \MString_{(p)}$, we need to show that the
    induced map $\tmf \otimes \QQ \to \tmf^\wedge_p\otimes \QQ \to
    \MString^\wedge_p\otimes \QQ$ agrees with the rational splitting constructed
    in the previous paragraph. However, this is immediate from the fact that the
    splittings $\tmf^\wedge_p\to \MString^\wedge_p$ are constructed to be
    equivalences in dimensions $\leq 15$, and the fact that the map out of $\tmf
    \otimes \QQ$ is determined by its effect on the generators $c_4$ and $c_6$.
\end{proof}
\begin{remark}\label{unstable-hard}
    The proof recalled in Remark \ref{e2-split-MO} of Thom's splitting of $\MO$
    proceeded essentially unstably: there is an $\E{2}$-map $\Omega^2 S^3\to
    \BO$ of spaces over $\B\GL_1(\S)$, whose Thomification yields the desired
    $\E{2}$-map $\H\FF_2\to \MO$. This argument also works for $\MSO$: there is
    an $\E{2}$-map $\Omega^2 S^3\langle 3\rangle\to \BSO$ of spaces over
    $\B\GL_1(\S)$, whose Thomification yields the desired $\E{2}$-map $\H\Z\to
    \MSO$. One might hope for the existence of a similar unstable map which
    would yield Theorem \ref{mstring}. We do not know how to construct such a
    map. To illustrate the difficulty, let us examine how such a proof would
    work; we will specialize to the case of $\MString$, but the discussion is
    the same for $\MSpin$.

    According to Theorem \ref{main-thm}, Conjecture \ref{moore-splitting} and
    Conjecture \ref{centrality-conj} imply that there is a map $K_3\to
    \B\GL_1(B)$ whose Thom spectrum is equivalent to $\tmf$. There is a map
    $BN\to \mathrm{B^2 String}$, whose fiber we will denote by $Q$. Then there
    is a fiber sequence
    $$N\to \BString\to Q,$$
    and so Proposition \ref{thom} implies that there is a map $Q\to \B\GL_1(B)$
    whose Thom spectrum is $\MString$. Theorem \ref{mstring} would follow if
    there was a map $f:K_3\to Q$ of spaces over $\B\GL_1(B)$, since Thomification
    would produce a map $\tmf\to \MString$.
    
    Conjecture \ref{moore-splitting} reduces the construction of $f$ to the
    construction of a map $\Omega^2 P^{17}(2)\to Q$. This map would in
    particular imply the existence of a map $P^{15}(2)\to Q$ (and would be
    equivalent to the existence of such a map if $Q$ was a double loop space),
    which in turn stems from a $2$-torsion element of $\pi_{14}(Q)$. The long
    exact sequence on homotopy runs
    $$\cdots \to \pi_{14}(\BString) \to \pi_{14}(Q)\to \pi_{13}(N)\to \pi_{13}
    (\BString)\to \cdots$$
    Bott periodicity states that $\pi_{13} \BString \cong \pi_{14} \BString\cong
    0$, so we find that $\pi_{14}(Q) \cong \pi_{13}(N)$. The desired $2$-torsion
    element of $\pi_{14}(Q)$ is precisely the element of $\pi_{13}(N)$ described
    in Remark \ref{sigma2-integral}. Choosing a particular nullhomotopy of twice
    this $2$-torsion element of $\pi_{14}(Q)$ produces a map $g:P^{15}(2)\to Q$.
    To extend this map over the double suspension $P^{15}(2)\to \Omega^2
    P^{17}(2)$, it would suffice to show that there is a double loop space
    $\wt{Q}$ with a map $\wt{Q} \to Q$ such that $g$ factors through $\wt{Q}$.

    Unfortunately, we do not know how to prove such a result; this is the
    unstable analogue of Conjecture \ref{centrality-conj}. In fact, such an
    unstable statement would bypass the need for Conjecture
    \ref{centrality-conj} in Theorem \ref{main-thm}. (One runs into the same
    obstruction for $\MSpin$, except with the fiber of the map $S^5\to
    \mathrm{B^2 Spin}$.) These statements are reminiscent of the conjecture (see
    Section \ref{gray}) that the fiber $W_n = \fib(S^{2n-1} \to \Omega^2
    S^{2n+1})$ of the double suspension admits the structure of a double loop
    space. 
\end{remark}
\begin{remark}
    The following application of Theorem \ref{mstring} was suggested by Mike
    Hopkins. In \cite{hopkins-hovey}, the Anderson-Brown-Peterson splitting is
    used to show that the Atiyah-Bott-Shapiro orientation $\MSpin\to \KO$
    induces an isomorphism
    $$\MSpin_\ast(X)\otimes_{\MSpin_\ast} \KO_\ast \xrightarrow{\cong}
    \KO_\ast(X)$$
    of $\KO_\ast$-modules for all spectra $X$.
    In future work, we shall show that Theorem \ref{mstring} can be used to
    prove the following height $2$ analogue of this result: namely, Conjectures
    \ref{moore-splitting}, \ref{centrality-conj}, and \ref{tmf-conj} imply that
    the Ando-Hopkins-Rezk orientation $\MString\to \Tmf$ induces an isomorphism
    \begin{equation}\label{mstring-land}
	\MString_\ast(X)\otimes_{\MString_\ast} \Tmf_\ast \xrightarrow{\cong}
	\Tmf_\ast(X)
    \end{equation}
    of $\Tmf_\ast$-modules for all spectra $X$.
    The $K(1)$-analogue of this isomorphism was obtained by Laures in
    \cite{laures-k1-local}.
\end{remark}


\subsection{Wood equivalences}\label{thom-wood}


The Wood equivalence states that $\bo \wedge C\eta \simeq \bu$. There are
generalizations of this equivalence to $\tmf$ (see \cite{homologytmf}); for
instance, there is a $2$-local $8$-cell complex $DA_1$ whose cohomology is
isomorphic to the double of $A(1)$ as an $A(2)$-module such that $\tmf_{(2)}
\wedge DA_1 \simeq \tmf_1(3) \simeq \BP{2}$. Similarly, if $X_3$ denotes the
$3$-local $3$-cell complex $S^0\cup_{\alpha_1} e^4 \cup_{2\alpha_1} e^8$, then
$\tmf_{(3)} \wedge X_3 \simeq \tmf_1(2) \simeq \BP{2} \vee \Sigma^8 \BP{2}$. We
will use the umbrella term ``Wood equivalence'' to refer to equivalences of this
kind.

Our goal in this section is to revisit these Wood equivalences using the point
of view stemming from Theorem \ref{main-thm}. In particular, we propose that
these equivalences are suggested by the existence of certain EHP sequences; we
will greatly expand on this in a forthcoming document. We
find this to be a rather beautiful connection between stable and unstable
homotopy theory. 

The first Wood-style result was proved in Proposition \ref{tn-yn}. The next
result, originally proved in \cite[Section 2.5]{mahowald-thom} and \cite[Theorem
3.7]{davis-mahowald}, is the simplest example of a Wood-style equivalence which
is related to the existence of certain EHP sequences.
\begin{prop}\label{A-wood}
    Let $\S\mmod\eta = X(2)$ (resp. $\S\mmod 2$) denote the $\E{1}$-quotient of
    $\S$ by $\eta$ (resp. $2$). If $Y = C\eta \wedge \S/2$ and $A_1$ is a
    spectrum whose cohomology is isomorphic to $A(1)$ as a module over the
    Steenrod algebra, then there are equivalences
    $$A\wedge C\eta \simeq \S\mmod \eta, \ A\wedge Y \simeq \S\mmod 2, \ A\wedge
    A_1 \simeq y(1)/v_1$$
    of $A$-modules.
\end{prop}
\begin{remark}\label{bo-wood}
    Proposition \ref{A-wood} implies the Wood equivalence $\bo \wedge C\eta
    \simeq \bu$. Although this implication is already true before
    $2$-completion, we will work in the $2$-complete category for convenience.
    Recall that Theorem \ref{main-thm} states that Conjecture
    \ref{moore-splitting} and Conjecture \ref{centrality-conj} imply that there
    is a map $\mu:K_2 \to \B\GL_1(A)$ whose Thom spectrum is equivalent to $\bo$
    (as left $A$-modules).  Moreover, the Thom spectrum of the composite $K_2
    \xrightarrow{\mu} \B\GL_1(A) \to \B\GL_1(T(1))$ is equivalent to $\BP{1}$.
    Since this Thom spectrum is the base-change $K_2^\mu \wedge_A T(1)$, and
    Proposition \ref{A-wood} implies that $T(1) = X(2) \simeq A\wedge C\eta$, we
    find that
    $$\BP{1} \simeq K_2^\mu \wedge_A (A\wedge C\eta) \simeq K_2^\mu \wedge C\eta
    \simeq \bo\wedge C\eta,$$
    as desired. Similarly, noting that $\S\mmod 2 = y(1)$, we find that
    Proposition \ref{A-wood} also proves the equivalence $\bo \wedge Y \simeq
    k(1)$.
\end{remark}
\begin{remark}\label{A-imply}
    The argument of Remark \ref{bo-wood} in fact proves that Theorem
    \ref{main-thm} for $A$ implies Theorem \ref{main-thm} for $T(1)$, $y_\Z(1)$,
    and $y(1)$.
\end{remark}
\begin{proof}[Proof of Proposition \ref{A-wood}]
    For the first two equivalences, it suffices to show that $A\wedge C\eta
    \simeq \S\mmod\eta$ and that $\S\mmod\eta \wedge \S/2 \simeq \S\mmod 2$. We
    will prove the first statement; the proof of the second statement is exactly
    the same. There is a map $C\eta \to \S\mmod \eta$ given by the inclusion of
    the $2$-skeleton.  There is also an $\E{1}$-ring map $A\to \S\mmod \eta$
    given as follows. The multiplication on $\S\mmod\eta$ defines a unital map
    $C\eta \wedge C\eta \to \S\mmod \eta$.  But since the Toda bracket
    $\langle\eta, 2, \eta\rangle$ contains $\nu$, there is a unital map $C\nu\to
    C\eta \wedge C\eta$. This supplies a unital map $C\nu \to \S\mmod \eta$,
    which, by the universal property of $A = \S\mmod\nu$ (via Theorem
    \ref{thom-univ}), extends to an $\E{1}$-ring map $A\to \S\mmod\eta$.

    For the final equivalence, it suffices to construct a map $A_1\to y(1)/v_1$
    for which the induced map $A\wedge A_1\to y(1)/v_1$ gives an isomorphism on
    mod $2$ homology. Since $A_1$ may be obtained as the cofiber of a $v_1$-self
    map $\Sigma^2 Y\to Y$, it suffices to observe that the the following diagram
    commutes; our desired map is the induced map on vertical cofibers:
    $$\xymatrix{
	\Sigma^2 Y \ar[d]^-{v_1} \ar[r] & \Sigma^2 y(1) \ar[d]^-{v_1}\\
	Y \ar[r] & y(1).
    }$$
\end{proof}
\begin{remark}\label{A-ehp}
    There are EHP sequences
    $$S^1 \to \Omega S^2 \to \Omega S^3, \ S^2\to \Omega S^3\to \Omega S^5.$$
    Recall that $\S/2$, $C\eta$, $\S\mmod 2$, $\S\mmod\eta = X(2)$, and $A$ are
    Thom spectra over $S^1$, $S^2$, $\Omega S^2$, $\Omega S^3$, and $\Omega S^5$
    respectively. Proposition \ref{thom} therefore implies that there are maps
    $f:\Omega S^3\to \B\Aut(\S/2)$ and $g:\Omega S^5\to \B\Aut(C\eta)$ whose
    Thom spectra are equivalent to $\S\mmod 2$ and $\S\mmod\eta$, respectively.
    The maps $f$ and $g$ define local systems of spectra over $\Omega S^3$ and
    $\Omega S^5$ whose fibers are equivalent to $\S/2$ and $C\eta$
    (respectively), and one interpretation of Proposition \ref{A-wood} is that
    these local systems in fact factor as
    $$\Omega S^3\xrightarrow{\eta} \B\GL_1(\S)\to \B\Aut(\S/2), \ \Omega
    S^5\xrightarrow{\nu} \B\GL_1(\S)\to \B\Aut(C\eta).$$
    Proposition \ref{A-wood} is an immediate consequence of these
    factorizations. We argue this for the first case in Remark
    \ref{T1-ehp-split}, and for the second in Remark \ref{A-ehp-split}, thereby
    giving an alternative EHP-based argument for Proposition \ref{A-wood}.
\end{remark}
\begin{remark}\label{T1-ehp-split}
    The first EHP sequence in Remark \ref{A-ehp} splits via the Hopf map $S^3\to
    S^2$. The map $f:\Omega S^3\to \B\Aut(\S/2)$ in fact factors through the
    dotted map in the following diagram:
    $$\xymatrix{
	S^2 \ar[r] & \Omega S^2 \ar[d] \ar[r] & \Omega S^3 \ar[d]
	\ar@{-->}[dl]\\
	& \B\GL_1(\S) \ar[r] & \B\Aut(\S/2).
    }$$
    Indeed, the composite $\Omega S^3\to \Omega S^2\to \B\GL_1(\S)$ is a loop
    map, and therefore is determined by the composite $\phi:S^3\to S^2\to
    B^2\GL_1(\S)$. Since the map $S^2\to B^2\GL_1(\S)$ detects the element
    $-1\in \pi_0(\S)^\times$, the map $\phi$ does in fact determine a unit
    multiple of $\eta$. This implies the desired claim.
\end{remark}
\begin{remark}\label{A-ehp-split}
    The map $g:\Omega S^5\to \B\Aut(C\eta)$ from Remark \ref{A-ehp} factors
    through $\BGL_1(\S)$. To see this, let us begin with the following
    observation: view $\BU$ and $\BSU$ as H-spaces via the tensor product of
    vector bundles. Then the map $\BSU \times \CP^\infty \to \BU$ classifying
    $\mathcal{V} \boxtimes \cL$, with $\mathcal{V}$ the universal $\SU$-bundle
    over $\BSU$ and $\cL$ the universal line bundle over $\BU$, is an
    equivalence of H-spaces.  In particular, there is a fiber sequence
    $$\CP^\infty \to \BU \to \BSU.$$
    The map $\Omega S^3 \to \BGL_1(\S)$ defining $T(1)$ factors as
    $$\Omega S^3 \to \BU \xar{J} \BGL_1(\S);$$
    similarly, the map $\Omega S^5 \to \BGL_1(\S)$ defining $A$ factors as
    $$\Omega S^5 \to \BSU \xar{J} \BGL_1(\S).$$
    These factorizations make the following diagram of fiber sequences commute:
    $$\xymatrix{
	S^2 \ar[r] \ar[d] & \Omega S^3 \ar[r] \ar[d] & \Omega S^5 \ar[d] \\
	\CP^\infty \ar[r] & \BU \ar[r] & \BSU.
    }$$
    The map $\Omega S^5 \to \B\Aut(C\eta)$ was defined using Proposition
    \ref{thom}. It then follows from the splitting of the bottom fiber sequence
    in the above diagram that the dotted map exists in the followng diagram:
    $$\xymatrix{
	S^2 \ar[r] \ar[d] & \Omega S^3 \ar[r] \ar[d] & \Omega S^5 \ar[d] \\
	\CP^\infty \ar[r] \ar[dr] & \BU \ar[r] \ar[d]^-J & \BSU \ar[d]
	\ar@{-->}[dl]^-J\\
	& \BGL_1(\S) \ar[r] & B\Aut(C\eta).
    }$$
    The composite
    $$\Omega S^5 \to \BSU \xar{J} \BGL_1(\S) \to B\Aut(C\eta)$$
    is $g$, giving our desired factorization.
\end{remark}

Next, we have the following result at height $2$:
\begin{prop}\label{B-wood}
    Let $DA_1$ denote the double of $A_1$ (see \cite{homologytmf}). There are
    $2$-complete equivalences
    $$B\wedge DA_1 \simeq T(2), \ B\wedge Z \simeq y(2), \ B \wedge A_2 \simeq
    y(2)/v_2,$$
    where $Z$ is the spectrum ``$\frac{1}{2} A_2$'' from
    \cite{mahowald-thompson,bhattacharya-egger-z}\footnote{In the former source,
    $Z$ is denoted by $\ol{M}$.}, and $A_2$ is a spectrum whose cohomology is
    isomorphic to $A(2)$ as a module over the Steenrod algebra.
\end{prop}
\begin{remark}\label{tmf-wood}
    Arguing as in Remark \ref{bo-wood} shows that Proposition \ref{B-wood} and
    Theorem \ref{main-thm} imply the Wood equivalences
    $$\tmf\wedge DA_1 \simeq \tmf_1(3) = \BP{2}, \ \tmf \wedge Z \simeq k(2), \
    \tmf\wedge A_2 \simeq \H\FF_2.$$
\end{remark}
\begin{remark}\label{B-imply}
    Exactly as in Remark \ref{A-imply}, the argument of Remark \ref{tmf-wood} in
    fact proves that Theorem \ref{main-thm} for $B$ implies Theorem
    \ref{main-thm} for $T(2)$, $y_\Z(2)$, and $y(2)$.
\end{remark}
\begin{remark}
    The telescope conjecture \cite[Conjecture 10.5]{ravenel-loc}, which we
    interpet as stating that $L_n$-localization is the same as
    $L_n^f$-localization, is known to be true at height $1$. For odd primes, it
    was proved by Miller in \cite{miller-telescope}, and at $p=2$ it was proved
    by Mahowald in \cite{mahowald-bo-res, mahowald-imj}. Mahowald's approach was
    to calculate the telescopic homotopy of the type $1$ spectrum $Y$. In
    \cite{mrs}, Mahowald-Ravenel-Shick proposed an approach to \emph{dis}proving
    the telescope conjecture at height $2$: they suggest that for $n\geq 2$, the
    $\Ln$-localization and the $v_n$-telescopic localization of $y(n)$ have
    different homotopy groups. They show, however, that the $L_1$-localization
    and the $v_1$-telescopic localization of $y(1)$ agree, so this approach
    (thankfully) does not give a counterexample to the telescope conjecture at
    height $1$.

    Motivated by Mahowald's approach to the telescope conjecture,
    Behrens-Beaudry-Bhattacharya-Culver-Xu study the $v_2$-telescopic homotopy
    of $Z$ in \cite{tmf-resolution}, with inspiration from the
    Mahowald-Ravenel-Shick approach. Propositions \ref{A-wood} and \ref{B-wood}
    can be used to relate these two (namely, the finite spectrum and the Thom
    spectrum) approaches to the telescope conjecture. As in Section
    \ref{cob-splitting}, we will let $R$ denote $A$ or $B$. Moreover, let $F$
    denote $Y$ or $Z$ (depending on what $R$ is), and let $R'$ denote $y(1)$ or
    $y(2)$ (again depending on what $R$ is), so that $R\wedge F = R'$ by
    Propositions \ref{A-wood} and \ref{B-wood}. Then:
\end{remark}
\begin{corollary}\label{telescope-equivalence}
    If the telescope conjecture is true for $F$ (and hence any type $1$ or $2$
    spectrum) or $R$, then it is true for $R'$.
\end{corollary}
\begin{proof}
    Since $L_n$- and
    $L_n^f$-localizations are smashing, we find that if the telescope conjecture
    is true for $F$ or $R$, then Propositions \ref{A-wood} and \ref{B-wood}
    yield equivalences
    $$L_n^f R' \simeq R \wedge L_n^f F \simeq R\wedge L_n F \simeq L_n R'.$$
\end{proof}
Finally, we prove Proposition \ref{B-wood}.
\begin{proof}[Proof of Proposition \ref{B-wood}]
    We first construct maps $B\to T(2)$ and $DA_1\to T(2)$. The top cell of
    $DA_1$ is in dimension $12$, and the map $T(2)\to \BPP$ is an equivalence in
    dimensions $\leq 12$. It follows that constructing a map $DA_1\to T(2)$ is
    equivalent to constructing a map $DA_1\to \BPP$. However, both $\BPP$ and
    $DA_1$ are concentrated in even degrees, so the Atiyah-Hirzebruch spectral
    sequence collapses, and we find that $\BPP^\ast(DA_1) \cong
    \H^\ast(DA_1;\BPP_\ast)$. The generator in bidegree $(0,0)$ produces a map
    $DA_1\to T(2)$; its effect on homology is the additive inclusion
    $\FF_2[\zeta_1^2, \zeta_2^2]/(\zeta_1^8, \zeta_2^4)\to \FF_2[\zeta_1^2,
    \zeta_2^2]$.

    The map $B\to T(2)$ may be defined via the universal property of Thom
    spectra from Section \ref{thom-background} and Remark \ref{B-iterated-thom}.
    Its effect on homology is the inclusion $\FF_2[\zeta_1^8, \zeta_2^4]\to
    \FF_2[\zeta_1^2, \zeta_2^2]$.  We obtain a map $B\wedge DA_1\to T(2)$ via
    the multiplication on $T(2)$, and this induces an isomorphism in mod $2$
    homology.

    For the second equivalence, we argue similarly: the map $B\to T(2)$ defines
    a map $B\to T(2)\to y(2)$. Next, recall that $Z$ is built through iterated
    cofiber sequences:
    $$\Sigma^2 Y\xrightarrow{v_1} Y\to A_1, \ \Sigma^5 A_1\wedge
    C\nu\xrightarrow{\sigma_1} A_1\wedge C\nu\to Z.$$
    As an aside, we note that the element $\sigma_1$ is intimately related to
    the element discussed in Example \ref{sigma1}; namely, it is given by the
    self-map of $A_1\wedge C\nu$ given by smashing $A_1$ with the following
    diagram:
    $$\xymatrix{
	\Sigma^5 C\nu \ar@{-->}[drrr]_-{\sigma_1} \ar[r] & \Sigma^5 A
	\ar[r]^-{\sigma_1\wedge \mathrm{id}} & A\wedge A\ar[r] & A\\
	& & & C\nu.\ar[u]
    }$$
    Using these cofiber sequences and Proposition \ref{tn-yn}, one obtains a map
    $Z\to y(2)$, which induces the additive inclusion $\FF_2[\zeta_1,
    \zeta_2]/(\zeta_1^8, \zeta_2^4)\to \FF_2[\zeta_1, \zeta_2]$ on mod $2$
    homology. The multiplication on $y(2)$ defines a map $B\wedge Z\to y(2)$,
    which induces an isomorphism on mod $2$ homology.

    For the final equivalence, it suffices to construct a map $A_2\to y(2)/v_2$
    for which the induced map $B\wedge A_2\to y(2)/v_2$ gives an isomorphism on
    mod $2$ homology. Since $A_2$ may be obtained as the cofiber of a $v_2$-self
    map $\Sigma^6 Z\to Z$, it suffices to observe that the the following diagram
    commutes; our desired map is the induced map on vertical cofibers:
    $$\xymatrix{
	\Sigma^6 Z \ar[d]^-{v_2} \ar[r] & \Sigma^6 y(2) \ar[d]^-{v_2}\\
	Z \ar[r] & y(2).
    }$$
\end{proof}
Arguing exactly as in the proof of Proposition \ref{B-wood} shows the following
result at the prime $3$:
\begin{prop}\label{B-wood-3}
    Let $X_3$ denote the $8$-skeleton of $T(1) = \S\mmod\alpha_1$. There are
    $3$-complete equivalences
    $$B\wedge X_3 \simeq T(2) \vee \Sigma^8 T(2), \ B\wedge X_3\wedge \S/(3,v_1)
    \simeq y(2) \vee \Sigma^8 y(2).$$
\end{prop}
In forthcoming work, we will discuss the relation between Proposition
\ref{B-wood} and EHP sequences, along the lines of Remark \ref{A-ehp}.

\newpage
\section{$C_2$-equivariant analogue of Corollary
\ref{bpn}}\label{equiv-analogue}

Our goal in this section is to study a $C_2$-equivariant analogue of Corollary
\ref{bpn} at height $1$. The odd primary analogue of this result is deferred to
the future; it is considerably more subtle.

\subsection{$C_2$-equivariant analogues of Ravenel's spectra}

In this section, we construct the $C_2$-equivariant analogue of $T(n)$ for all
$n$. We $2$-localize everywhere until mentioned otherwise. There is a
$C_2$-action on $\Omega \SU(n)$ given by complex conjugation, and the resulting
$C_2$-space is denoted $\Omega \SU(n)_\RR$. Real Bott periodicity gives a
$C_2$-equivariant map $\Omega \SU(n)_\RR \to \BU_\RR$ whose Thom spectrum is the
(genuine) $C_2$-spectrum $X(n)_\RR$. This admits the structure of an
$\E{\rho}$-ring, since it is the Thom spectrum of an $\E{\rho}$-map $\Omega^\rho \B^\sigma \SU(n)_\RR \to \Omega^\rho \B^\rho \BU_\RR \simeq \Omega^\rho \BSU_\RR$. As in the nonequivariant case, the equivariant Quillen
idempotent on $\MU_\RR$ restricts to one on $X(m)_\RR$, and therefore defines a
summand $T(n)_\RR$ of $X(m)_\RR$ for $2^n\leq m\leq 2^{n+1}-1$. Again, this
summand admits the structure of an $\E{1}$-ring.

\begin{construction}\label{c2-chin}
    There is an equivariant fiber sequence $$\Omega \SU(n)_\RR \to \Omega
    \SU(n+1)_\RR \to \Omega S^{n\rho+1},$$ where $\rho$ is the regular
    representation of $C_2$; the equivariant analogue of Proposition \ref{thom}
    then shows that there is a map $\Omega S^{n\rho+1} \to \B\GL_1(X(n)_\RR)$
    (detecting an element $\chi_n\in \pi_{n\rho-1} X(n)_\RR$) whose Thom
    spectrum is $X(n+1)_\RR$. Here, $\B\GL_1(X(n)_\RR)$ is the delooping of the $\E{\rho}$-space $\GL_1(X(n)_\RR)$, and the $C_2$-equivariant notion of Thom spectrum is taken in the sense of \cite[Theorem 3.2]{equiv-thom}. (The constructions from \textit{loc. cit.} can be verified to go through for equivariant maps to $\B\GL_1(X(n)_\RR)$; for example, when $n=\infty$, the idea of taking Thom spectra for an equivariant map to $\B\GL_1(\MU_\RR)$ was already used in \cite[Section 3]{hahn-shi}.)
\end{construction}
If $\wt{\sigma}_n$ denotes the image of the element $\chi_{2^{n+1}\rho-1}$ in
$\pi_{(2^{n+1}-1)\rho-1} T(n)_\RR$, then we have a $C_2$-equivariant analogue of
Lemma \ref{tn-thom}:
\begin{lemma}\label{C2-tn}
    The Thom spectrum of the map $\Omega S^{(2^{n+1}-1)\rho+1} \to
    \B\GL_1(X(2^{n+1}-1)_\RR)$ detecting $\wt{\sigma}_n$ is a direct sum of shifts of $T(n+1)_\RR$.
\end{lemma}
\begin{example}
    For instance, $T(1)_\RR = X(2)_\RR$ is the Thom spectrum of the map $\Omega
    S^{\rho+1} \to \BU_\RR$; upon composing with the equivariant J-homomorphism
    $\BU_\RR \to \B\GL_1(\S)$, this detects the element $\wt{\eta} \in \pi_\sigma
    \S$, and the extension of the map $S^\rho \to \B\GL_1(\S)$ to $\Omega
    S^{\rho+1}$ uses the $\E{1}$-structure on $\B\GL_1(\S)$. The case of
    $X(2)_\RR$ exhibits a curious property: $S^{\rho+1}$ is the loop space
    $\Omega^\sigma \HP^\infty_\RR$, and there are equivalences (see
    \cite[Proposition 3.4 and Proposition 3.6]{hahn-wilson})
    $$\Omega S^{\rho+1} \simeq \Omega^{\sigma+1} \HP^\infty_\RR \simeq
    \Omega^\sigma (\Omega \HP^\infty_\RR).$$
    However, $\Omega \HP^\infty_\RR \simeq S^{\rho+\sigma}$, so $\Omega
    S^{\rho+1} = \Omega^\sigma S^{\rho+\sigma}$. The map $\Omega^\sigma
    S^{\rho+\sigma} \to \B\GL_1(\S)$ still detects the element $\wt{\eta}\in
    \pi_\sigma \S$ on the bottom cell, but the extension of the map $S^\rho \to
    \B\GL_1(\S)$ to $\Omega^\sigma S^{\rho+\sigma}$ is now defined via the
    $\E{\sigma}$-structure on $\B\GL_1(\S)$.  The upshot of this discussion is
    that $X(2)_\RR$ is not only the free $\E{1}$-ring with a nullhomotopy of
    $\wt{\eta}$, but also the free $\E{\sigma}$-algebra with a nullhomotopy of
    $\wt{\eta}$.
\end{example}
\begin{warning}\label{not-torsion}
    Unlike the nonequivariant setting, the element $\wt{\eta}\in \pi_\sigma \S$
    is neither torsion nor nilpotent. This is because its geometric fixed points
    is $\Phi^{C_2} \wt{\eta} = 2\in \pi_0 \S$; see \cite[Proposition C.5]{dugger-isaksen}, although note that the orientations chosen there are the opposite of ours. Briefly, the map $\wt{\eta}$ is obtained by $\rho$-desuspending the unstable equivariant Hopf map $S^{\rho+\sigma} = \cc^2-\{0\} \to \CP^1 = S^\rho$, whose homotopy fiber is $S^\sigma$. In other words, there is a fiber sequence $S^\sigma \to S^{\rho+\sigma} \xar{\wt{\eta}} S^\rho$. On geometric fixed points, this produces the fiber sequence $S^0 = C_2 \to S^1 \to S^1$, which forces the map $\Phi^{C_2} \wt{\eta}$ to have degree $2$ (or $-2$, depending on the choice of orientation).
\end{warning}

\begin{example}\label{sigma1-equiv}
    Consider the element $\wt{\sigma}_1\in \pi_{3\rho-1} T(1)_\RR$. The
    underlying nonequivariant element of $\pi_5 T(1)_\RR$ is simply $\sigma_1$.
    To determine $\Phi^{C_2} \wt{\sigma}_1\in \pi_2 \Phi^{C_2} T(1)_\RR$, we
    first note that $\Phi^{C_2} T(1)_\RR$ is the Thom spectrum of the map
    $\Phi^{C_2} \wt{\eta}: \Phi^{C_2} \Omega S^{\rho+1}\to \B\GL_1(\S)$. Since
    $\Phi^{C_2} \Omega S^{\rho+1} = \Omega S^2$ and $\Phi^{C_2} \wt{\eta} = 2$,
    we find that $\Phi^{C_2} T(1)_\RR$ is the $\E{1}$-quotient $\S\mmod 2$. The
    element $\Phi^{C_2} \wt{\sigma}_1\in \pi_2 \S\mmod 2\cong \pi_2 \S/2$ is
    simply a map $S^2\to \S/2$ which is $\eta$ on the top cell. Such a map
    exists because $2\eta = 0$.
\end{example}

As an aside, we mention that there is a $C_2$-equivariant lift of the spectrum
$A$:
\begin{definition}
    Let $A_{C_2}$ denote the Thom spectrum of the map $\Omega S^{2\rho+1} \to \BGL_1(\S)$ defined by the extension of the map $S^{2\rho} \to \BGL_1(\S)$ which detects the equivariant Hopf map $\wt{\nu}\in \pi_{2\rho-1} \S$.
\end{definition}
\begin{remark}\label{equiv-A}
    The underlying spectrum of $A_{C_2}$ is $A$. To determine the geometric
    fixed points of $A_{C_2}$, 
    $\Phi^{C_2} A_{C_2}$ is the Thom spectrum of the map $\Phi^{C_2} \wt{\nu}: \Phi^{C_2}
    \Omega S^{2\rho+1} \to \B\GL_1(\S)$. We claim that $\wt{\Phi}^{C_2} \wt{\nu} = \eta$; indeed, the map $\wt{\nu}$ is obtained by $2\rho$-desuspending the unstable equivariant map $S^{4\rho-1} = \mathbf{H}^2-\{0\} \to \mathbf{H}P^1 = S^{2\rho}$. The homotopy fiber of this map is $S^{2\rho-1} = S^{\rho+\sigma}$, so that there is an equivariant fiber sequence $S^{\rho+\sigma} \to S^{4\rho-1} \to S^{2\rho}$. On geometric fixed points, we obtain a fiber sequence $S^1 \to S^3 \to S^2$, which implies that $\Phi^{C_2} \wt{\nu}$ be identified with the Hopf fibration $S^3 \to S^2$.
    Now, since $\Phi^{C_2} \Omega S^{2\rho+1} = \Omega S^3$, we find that $\Phi^{C_2}
    A_{C_2} = T(1)$. In particular, $A_{C_2}$ may be thought of as the free
    $C_2$-equivariant $\E{1}$-ring with a nullhomotopy of $\wt{\nu}$.
\end{remark}
\begin{example}
    The element $\wt{\sigma}_1$ lifts to $\pi_{3\rho-1} A_{C_2}$. Indeed, Remark
    \ref{sigma1-integral} works equivariantly too: the equivariant Hopf map
    $S^{3\rho-1}\to S^{2\rho}$ defines a composite $S^{3\rho-1}\to S^{2\rho}\to
    \Omega S^{2\rho+1}$. The composite $S^{3\rho-1}\to \Omega S^{2\rho+1}\to
    \BSU_\RR$ is null, since $\pi_{3\rho-1} \BSU_\RR = 0$. It follows that upon
    Thomification, the map $S^{3\rho-1}\to \Omega S^{2\rho+1}$ defines an
    element $\wt{\sigma}_1'$ of $\pi_{3\rho-1} A_{C_2}$. In order to show that
    this element indeed deserves to be called $\wt{\sigma}_1$, we use
    Proposition \ref{equiv-A-t1}. The map $A_{C_2}\to T(1)_\RR$ from the
    proposition induces a map $\pi_{3\rho-1} A_{C_2}\to \pi_{3\rho - 1}
    T(1)_\RR$, and we need to show that the image of $\wt{\sigma}_1'\in
    \pi_{3\rho-1} A_{C_2}$ is in fact $\wt{\sigma}_1$. By Example
    \ref{sigma1-equiv}, it suffices to observe that the underlying
    nonequivariant map corresponding to $\wt{\sigma}_1'\in \pi_{3\rho-1}
    T(1)_\RR$ is $\sigma_1$, and that the geometric fixed points 
    ${\Phi}_{C_2} \wt{\sigma}_1'\in \pi_2 \S\mmod 2$ is the lift of $\eta$
    appearing in Example \ref{sigma1-equiv}.
\end{example}
We now prove the proposition used above.
\begin{prop}\label{equiv-A-t1}
    There is a genuine $C_2$-equivariant $\E{1}$-map $A_{C_2}\to T(1)_\RR$.
\end{prop}
\begin{proof}
    By Remark \ref{equiv-A}, it suffices to show that $\wt{\nu} = 0\in
    \pi_{3\rho-1} T(1)_\RR$. The underlying map is null, because $\nu = 0\in
    \pi_5 T(1)$. The geometric fixed points are also null, because $\Phi^{C_2}
    \wt{\nu} = \eta$ is null in $\pi_2 \Phi^{C_2} T(1)_\RR = \pi_2 \S\mmod 2$.
    Therefore, $\wt{\nu}$ is null in $\pi_{3\rho-1} T(1)_\RR$.
\end{proof}
In fact, it is easy to prove the following analogue of Proposition \ref{A-wood}:
\begin{prop}\label{equiv-A-wood}
    There is a $C_2$-equivariant equivalence $A_{C_2} \wedge C\wt{\eta} \simeq
    T(1)_\RR$.
\end{prop}
\begin{proof}
    There are maps $A_{C_2} \to T(1)_\RR$ and $C\wt{\eta} \to T(1)_\RR$, which
    define a map $A_{C_2} \wedge C\wt{\eta} \to T(1)_\RR$ via the multiplication
    on $T(1)_\RR$. This map is an equivalence on underlying by Proposition
    \ref{A-wood}, and on geometric fixed points induces the map $T(1) \wedge
    \S/2\to \S\mmod 2$. This was also proved in the course of Proposition
    \ref{A-wood}.
\end{proof}
\begin{remark}
    As in Remark \ref{bo-wood}, one might hope that this implies the
    $C_2$-equivariant Wood equivalence $\bo_{C_2} \wedge C\wt{\eta} \simeq
    \bu_\RR$ via some equivariant analogue of Theorem \ref{main-thm}.
\end{remark}
\begin{remark}
    The equivariant analogue of Remark \ref{A-ehp} remains true: the equivariant
    Wood equivalence of Proposition \ref{equiv-A-wood} stems from the EHP
    sequence $S^\rho\to \Omega S^{\rho+1} \to \Omega S^{2\rho + 1}$. To prove the existence of such a fiber sequence, we use \cite[Construction 4.26]{ehp-haine} to get the Hopf map $h: \Omega S^{\rho+1} \to \Omega S^{2\rho + 1}$, as well as a nullhomotopy of the composite $S^\rho\to \Omega S^{\rho+1} \to \Omega S^{2\rho + 1}$. In particular, if $F = \fib(h)$, there is an equivariant map $S^\rho \to F$. We claim that this map is an equivalence: it suffices to prove that $S^\rho \to F$ is an equivalence on underlying and on geometric fixed points, since these functors preserve homotopy limits and colimits, and these functors are jointly conservative. The desired equivalence on underlying spaces follows from the classical EHP sequence $S^2 \to \Omega S^3 \to \Omega S^5$, and the equivalence on geometric fixed points follows from the splitting $\Omega S^2 \simeq S^1 \times \Omega S^3$.
\end{remark}


\subsection{The $C_2$-equivariant analogue of Corollary \ref{bpn} at $n=1$}

Recall (see \cite{hu-kriz}) that there are indecomposable classes $\ol{v}_n\in
\pi_{(2^n-1)\rho} \BPP_\RR$; as in Theorem \ref{tn-def}, these lift to classes
in $\pi_\star T(m)_\RR$ if $m\geq n$. The main result of this section is the
following:
\begin{theorem}\label{C2-main-thm}
    There is a map $\Omega^\rho S^{2\rho+1} \to \B\GL_1(T(1)_\RR)$ detecting an
    indecomposable in $\pi_\rho T(1)_\RR$ on the bottom cell, whose Thom
    spectrum is $\ul{\H\Z}$.
\end{theorem}
Note that, as with Corollary \ref{bpn} at $n=1$, this result is
\emph{unconditional}. The argument is exactly as in the proof of Corollary
\ref{bpn} at $n=1$, with practically no modifications. We need the following analogue of Theorem \ref{hm}, originally proved in \cite{behrens-wilson,hahn-wilson}.
\begin{prop}[Behrens-Wilson, Hahn-Wilson]\label{equiv-hm}
    Let $p$ be any prime, and let $\lambda$ denote the $2$-dimensional standard
    representation of $C_p$ on $\cc$. The Thom spectrum of the map
    $\Omega^\lambda S^{\lambda+1} \to \B\GL_1(S^0)$ extending the map $1-p:S^1
    \to \B\GL_1(S^0)$ is equivalent to $\ul{\H\FF_p}$ as an $\E{\lambda}$-ring.
    Moreover, if $S^{\lambda+1}\langle \lambda+1\rangle$ denotes the
    $(\lambda+1)$-connected cover of $S^{\lambda+1}$ (i.e., the fiber of the map $S^{\lambda+1} \to \Omega^\infty \Sigma^{\lambda+1} \ul{\H\Z}$), then the Thom spectrum of
    the induced map $\Omega^\lambda S^{\lambda+1}\langle \lambda+1\rangle \to
    \B\GL_1(S^0)$ is equivalent to $\ul{\H\Z}$ as an $\E{\lambda}$-ring.
\end{prop}
We can now prove Theorem \ref{C2-main-thm}.
\begin{proof}[Proof of Theorem \ref{C2-main-thm}]
    In \cite{hahn-wilson}, the authors prove that there is an equivalence of
    $C_2$-spaces between $\Omega^\lambda S^{\lambda+1}$ and $\Omega^\rho
    S^{\rho+1}$, and that $\ul{\H\FF_2}$ is in fact the Thom spectrum of the
    induced map $\Omega^\rho S^{\rho+1} \to \B\GL_1(S^0)$ detecting $-1$. Since
    both $\Omega^\rho S^{\rho+1}\langle \rho+1\rangle$ and $\Omega^\lambda
    S^{\lambda+1}\langle \lambda+1\rangle$ are defined as fibers of maps to
    $S^1$ which are degree one on the bottom cell, Hahn and Wilson's equivalence
    lifts to a $C_2$-equivariant equivalence $\Omega^\rho S^{\rho+1}\langle
    \rho+1\rangle\simeq \Omega^\lambda S^{\lambda+1}\langle \lambda+1\rangle$,
    and we find that $\ul{\H\Z}$ is the Thom spectrum of the map $\Omega^\rho
    S^{\rho+1}\langle \rho+1\rangle\to \B\GL_1(S^0)$.

    Since $T(1)_\RR$ is the Thom spectrum of the composite map $\Omega
    S^{\rho+1} \to \Omega^\rho S^{\rho+1}\langle \rho+1\rangle \to \B\GL_1(S^0)$
    detecting $\wt{\eta}$ on the bottom cell of the source, it follows from
    the $C_2$-equivariant analogue of Proposition \ref{thom} and the above
    discussion that it is sufficient to define a fiber sequence
    $$\Omega S^{\rho+1} \to \Omega^\rho S^{\rho+1}\langle \rho+1\rangle \to
    \Omega^\rho S^{2\rho+1},$$
    and check that the induced map $\Omega^\rho S^{2\rho+1} \to
    \B\GL_1(T(1)_\RR)$ detects an indecomposable element of $\pi_\rho T(1)_\RR$.
    See Remark \ref{anick-S3} for the nonequivariant analogue of this fiber
    sequence.

    Since there is an equivalence $\Omega S^{\rho+1} \simeq \Omega^\sigma
    S^{\rho + \sigma}$, it suffices to prove that there is a fiber sequence
    \begin{equation}\label{C2-cmn}
	S^{\rho+\sigma} \to \Omega S^{\rho+1}\langle \rho+1\rangle \to \Omega
	S^{2\rho+1};
    \end{equation}
    taking $\sigma$-loops produces the desired fiber sequence.
    The fiber sequence \eqref{C2-cmn} can be obtained by taking vertical fibers
    in the following map of fiber sequences
    $$\xymatrix{
	S^\rho \ar[r] \ar[d] & \Omega S^{\rho+1} \ar[r] \ar[d] & \Omega S^{2\rho
	+ 1} \ar[d]\\
	\CP^\infty_\RR \ar@{=}[r] & \CP^\infty_\RR \ar[r] & \ast.}$$
    Here, the top horizontal fiber sequence is the EHP fiber sequence
    $$S^\rho \to \Omega S^{\rho+1} \to \Omega S^{2\rho+1}.$$
    To identify the fibers, note that there is the Hopf fiber sequence
    $$S^{\rho+\sigma} \xar{\wt{\eta}} S^\rho\to \CP^\infty_\RR.$$
    The fiber of the middle vertical map is $\Omega S^{\rho+1} \langle \rho +
    1\rangle$ via the definition of $S^{\rho+1}\langle \rho+1\rangle$ as the
    homotopy fiber of the map $S^{\rho+1}\to B \CP^\infty_\RR$.

    It remains to show that the map $\Omega^\rho S^{2\rho+1} \to
    \B\GL_1(T(1)_\RR)$ detects an indecomposable element of $\pi_\rho T(1)_\RR$.
    Indecomposability in $\pi_\rho T(1)_\RR \cong \pi_\rho \BPP_\RR$ is the same
    as not being divisible by $2$, so we just need to show that the dotted map
    in the following diagram does not exist:
    $$\xymatrix{
	S^{\rho+1} \ar[d]_-{E^2} \ar[dr]^-2 & \\
	\Omega^\rho S^{2\rho+1} \ar[r] \ar[d] & S^{\rho+1} \ar@{-->}[dl]\\
	\B\GL_1(T(1)_\RR)}$$
	If this factorization existed, there would be an orientation
	$\ul{\H\Z}\to T(1)_\RR$, which is absurd.
\end{proof}

We now explain why we do not know how to prove the equivariant analogue of
Corollary \ref{bpn} at higher heights. One could propose an equivariant analogue
of Conjecture \ref{moore-splitting}, and such a conjecture would obviously be
closely tied with the existence of some equivariant analogue of the work of
Cohen-Moore-Neisendorfer. We do not know if any such result exists, but it would
certainly be extremely interesting.


Suppose that one wanted to prove a result like Corollary \ref{bpn}, stating that
the equivariant analogues of Conjecture \ref{moore-splitting} and Conjecture
\ref{centrality-conj} imply that there is a map $\Omega^\rho S^{2^n\rho+1} \to
\B\GL_1(T(n)_\RR)$ detecting an indecomposable in $\pi_{(2^n-1) \rho} T(n)_\RR$
on the bottom cell, whose Thom spectrum is $\BP{n-1}_\RR$. One could then try to
run the same proof as in the nonequivariant case by constructing a map from the
fiber of a charming map $\Omega^\rho S^{2^n\rho+1}\to S^{(2^n-1)\rho+1}$ to
$\B\GL_1(T(n-1)_\RR)$, but the issue comes in replicating Step 1 of Section
\ref{the-proof}: there is \emph{no} analogue of Lemma \ref{chin-torsion}, since
the equivariant element $\wt{\sigma}_n\in \pi_\star T(n)$ is neither torsion nor
nilpotent. See Warning \ref{not-torsion}.  This is intimately tied with the
failure of an analogue of the nilpotence theorem in the equivariant setting. In
future work, we shall describe a related project connecting the $T(n)$ spectra
to the Andrews-Gheorghe-Miller $w_n$-periodicity in $\cc$-motivic homotopy
theory (see \cite{andrews-miller, gheorghe, krause-thesis}).

However, since there is a map $\Omega^\lambda S^{\lambda+1} \langle
\lambda+1\rangle\to \B\GL_1(\S)$ as in Proposition \ref{equiv-hm}, there may
nevertheless be a way to construct a suitable map from the fiber of a charming
map $\Omega^\rho S^{2^n\rho+1}\to S^{(2^n-1)\rho+1}$ to $\B\GL_1(T(n-1)_\RR)$.
Such a construction would presumably provide a more elegant construction of the
nonequivariant map used in the proof of Theorem \ref{main-thm}.

\newpage
\section{Future directions}\label{future}
In this section, we suggest some directions for future investigation.
This is certainly not an exhaustive list; there are numerous questions we do not
know how to address that are spattered all over this document, but we have tried
to condense some of them into the list below. We have tried to order the
questions in order of our interest in them. We have partial progress on many of
these questions.
\begin{enumerate}
    \item Some obvious avenues for future work are the conjectures studied in
	this article: Conjectures \ref{moore-splitting}, \ref{centrality-conj},
	\ref{tmf-conj}, and \ref{tn-e2}.
	Can the $\E{3}$-assumption in the statement of Theorem \ref{mstring} be
	removed?
    \item One of the Main Goals\textsuperscript{TM} of this project is to
	rephrase the proof of the nilpotence theorem from \cite{dhs-i, dhs-ii}.
	As mentioned in Remark \ref{hm-nilp}, the Hopkins-Mahowald theorem for
	$\H\FF_p$ immediately implies the nilpotence theorem for simple
	$p$-torsion classes in the homotopy of a homotopy commutative ring
	spectrum (see also \cite{hopkins-thesis}). We will expand on the
	relation between the results of this article and the nilpotence theorem in
	forthcoming work; see Remark \ref{nilpotence-proof} for a sketch.

	From this point of view, Theorem \ref{main-thm} is very interesting: it
	connects torsion in the unstable homotopy groups of spheres (via
	Cohen-Moore-Neisendorfer) to nilpotence in the stable homotopy groups of
	spheres. We are not sure how to do so, but could the
	Cohen-Moore-Neisendorfer bound for the exponents of unstable homotopy
	groups of spheres be used to obtain bounds for the nilpotence exponent
	of the stable homotopy groups of spheres?
    \item It is extremely interesting to contemplate the interaction between
	unstable homotopy theory and chromatic homotopy theory apparent in this
	article. Connections between unstable homotopy theory and the chromatic
	picture have appeared elsewhere in the literature (e.g., in
	\cite{arone-mahowald, arone-iterate, mahowald-imj, mahowald-thompson}),
	but their relationship to the content of this project is not clear to
	me. It would be interesting to have this clarified. One na\"ive hope is
	that such connection could stem from a construction of a charming map
	(such as the Cohen-Moore-Neisendorfer map) via Weiss calculus.
    \item Let $R$ denote $\S$ or $A$. The map $R\to \TT(R)$ is an equivalence in
	dimensions $<|\sigma_n|$. Moreover, the $\TT(R)$-based Adams-Novikov
	spectral sequence has a vanishing line of slope $1/|\sigma_n|$ (see
	\cite{mahowald-bo-res} for the case $R = A$). Can another proof of this
	vanishing line be given using general arguments involving Thom spectra?
	We have some results in this direction which we shall address in future
	work.
    \item The unit maps from each of the Thom spectra on the second line of
	Table \ref{the-table} to the corresponding designer spectrum on the
	third line are surjective on homotopy. In the case of $\tmf$, this
	requires some computational effort to prove, and has been completed in
	\cite{tmf-witten}. This behavior is rather unexpected: in general, the
	unit map from a structured ring to some structured quotient will not be
	surjective on homotopy. Is there a conceptual reason for this
	surjectivity?
    \item In \cite{tmf-resolution}, the $\tmf$-resolution of a certain type $2$
	spectrum $Z$ is studied. Mahowald uses the Thom spectrum $A$ to study
	the $\bo$-resolution of the sphere in \cite{mahowald-bo-res}, so perhaps
	the spectrum $B$ could be used to study the $\tmf$-resolution of $Z$.
	This is work in progress. See also Corollary \ref{telescope-equivalence}
	and the discussion preceding it.
    \item Is there an equivariant analogue of Theorem \ref{main-thm} at higher
	heights and other primes? Currently, we have such an analogue at height
	$1$ and at $p=2$; see Section \ref{equiv-analogue}.
    \item The Hopkins-Mahowald theorem may used to define Brown-Gitler spectra.
	Theorem \ref{main-thm} produces ``relative'' Brown-Gitler spectra for
	$\BP{n}$, $\bo$, and $\tmf$. In future work, we will study these spectra
	and show how they relate to the Davis-Mahowald \emph{non}-splitting of
	$\tmf \wedge \tmf$ as a wedge of shifts of $\bo$-Brown-Gitler spectra
	smashed with $\tmf$ from \cite{davis-mahowald-tmf}.
    \item The story outlined in the introduction above could fit into a general
	framework of ``fp-Mahowaldean spectra'' (for ``finitely presented
	Mahowaldean spectrum'', inspired by \cite{mahowald-rezk}), of which $A$,
	$B$, $T(n)$, and $y(n)$ would be examples. One might then hope for a
	generalization of Theorem \ref{main-thm} which relates fp-Mahowaldean
	spectra to fp-spectra. It would also be interesting to prove an analogue
	of Mahowald-Rezk duality for fp-Mahowaldean spectra which recovers their
	duality for fp-spectra upon taking Thom spectra as above.
    \item One potential approach to the question about surjectivity raised above
	is as follows. The surjectivity claim at height $0$ is the (trivial)
	statement that the unit map $\S \to \H\Z$ is surjective on homotopy. The
	Kahn-Priddy theorem, stating that the transfer $\lambda: \Sigma^\infty
	\RP^\infty \to \S$ is surjective on $\pi_{\ast \geq 1}$, can be
	interpreted as stating that $\pi_\ast \Sigma^\infty \RP^\infty$ contains
	those elements of $\pi_\ast \S$ which are \emph{not} detected by $\H\Z$.
	One is then led to wonder: for each of the Thom spectra $R$ on the
	second line of Table \ref{the-table}, is there a spectrum $P$ along with
	a map $\lambda_R: P\to R$ such that each $x\in \pi_\ast R$ in the kernel
	of the map $R\to \TT(R)$ lifts along $\lambda_R$ to $\pi_\ast P$? (The
	map $R\to \TT(R)$ is an equivalence in dimensions $<|\sigma_n|$ (if $R$
	is of height $n$), so $P$ would have bottom cell in dimension
	$|\sigma_n|$.)
	
	Since $\Sigma^\infty \RP^\infty \simeq \Sigma^{-1} \Sym^2(\S)/\S$, the
	existence of such a result is very closely tied to an analogue of the
	Whitehead conjecture (see \cite{kuhn-whitehead}; the Whitehead
	conjecture implies the Kahn-Priddy theorem). In particular, one might
	expect the answer to the question posed above to admit some interaction
	with Goodwillie calculus.
    \item Let $p\geq 5$. Is there a $p$-primary analogue of $B$ which would
	provide a Thom spectrum construction (via Table \ref{the-table}) of the
	conjectural spectrum $\mathrm{eo}_{p-1}$? Such a spectrum would be the
	Thom spectrum of a $p$-complete spherical fibration over a $p$-local
	space built via $p-1$ fiber sequences from the loop spaces $\Omega
	S^{2k(p-1)+1}$ for $2\leq k\leq p$.
    \item The spectra $T(n)$ and $y(n)$ have algebro-geometric interpretations:
	the stack $\M_{T(n)}$ associated (see \cite[Chapter 9]{tmf}; this stack
	is well-defined since $T(n)$ is homotopy commutative) to $T(n)$
	classifies $p$-typical formal groups with a coordinate up to degree
	$p^{n+1}-1$, while $y(n)$ is the closed substack of $\M_{T(n)}$ defined
	by the vanishing locus of $p,v_1,\cdots,v_{n-1}$. What are the moduli
	problems classified by $A$ and $B$? We do not know if this question even
	makes sense at $p=2$, since $A$ and $B$ are \emph{a priori} only
	$\E{1}$-rings. Nonetheless, in \cite{hodge}, we provide a description of
	$\tmf \wedge A$ in terms of the Hodge filtration of the universal
	elliptic curve (even at $p=2$); we also showed that $(\tmf \wedge
	A)[x_2]$ admits an $\E{2}$-algebra structure, where $|x_2| = 2$.
    \item Theorem \ref{main-thm} shows that the Hopkins-Mahowald theorem for
	$\H\Z_p$ can be generalized to describe forms of $\BP{n}$; at least for
	small $n$, these spectra have associated algebro-geometric
	interpretations (see \cite[Chapter 9]{tmf}). What is the
	algebro-geometric interpretation of Theorem \ref{main-thm}?
\end{enumerate}

\newpage

\bibliographystyle{alpha}
\bibliography{main}
\end{document}